\documentclass[12pt]{article}

\usepackage{fullpage}
\let\hat\widehat

\usepackage[pdfencoding=auto, psdextra]{hyperref}
\hypersetup{
	colorlinks,
	citecolor=blue,
	filecolor=black,
	linkcolor=black,
	urlcolor=black
}
\usepackage{bookmark}
\bookmarksetup{depth=2}  % Set depth to exclude paragraphs

\usepackage{amsmath}
\usepackage{amssymb}
\usepackage{amsthm}
\usepackage{derivative}
\usepackage{mathtools} % provides dcases
\usepackage{thmtools,thm-restate}
\mathtoolsset{showonlyrefs=true}
\usepackage{enumitem} % list with letters

\usepackage{graphicx}
\graphicspath{ {./img/} }
\usepackage{caption}
\usepackage{subcaption} % for subfigures
\usepackage{float}

\usepackage{etoc}

\usepackage{natbib}
\theoremstyle{plain}
\newtheorem{theorem}{Theorem}
\newtheorem{lemma}{Lemma}
\newtheorem*{lemma*}{Lemma}

\newtheorem{proposition}{Proposition}
\newtheorem{definition}{Definition}
\newtheorem{property}{Property}
\newtheorem*{property*}{Property}
\newtheorem{remark}{Remark}
\DeclareMathOperator{\iid}{\stackrel{iid}{\sim}}
\DeclareMathOperator{\st}{\text{ s.t. }}
\DeclareMathOperator{\period}{\text{.}}

\DeclareMathOperator{\vs}{\text{ versus }}
\DeclareMathOperator{\where}{\text{ where }}

\DeclareMathOperator{\textif}{\text{if }}

\DeclareMathOperator{\textand}{\text{ and }}
\DeclareMathOperator{\textor}{\text{ or }}
\DeclareMathOperator{\for}{\text{ for }}
\DeclareMathOperator{\Corr}{CS}

\DeclareMathOperator{\Rank}{rank}
\DeclareMathOperator{\Span}{span}
\DeclareMathOperator{\supp}{supp}

\DeclareMathOperator*{\op}{op}
\DeclareMathOperator*{\proj}{Proj}
\DeclareMathOperator*{\intrinsic}{Int}

\DeclareMathOperator*{\trace}{Tr}

\DeclareMathOperator*{\arginf}{arg\,inf}

\let\norm\relax
\newcommand{\norm}[1]{\lVert #1 \rVert}
\newcommand{\inner}[2]{\left\langle #1 , #2 \right\rangle}

\newcommand{\Normal}{\mathcal{N}}

\newcommand{\Epsilon}{\mathcal{E}}
\newcommand{\R}{\mathbb{R}} % reals

\usepackage{bbm} % for indicator functions
\usepackage[bb=boondox]{mathalfa} % bb for numbers
\usepackage[bbgreekl]{mathbbol} % bb for greek letters

\DeclareMathOperator{\LR}{LR}

\begin{document}

\begin{center}
{\bf{\Large{Optimal Inference with Black-box Predictions}}}

\vspace*{.2in}

\begin{tabular}{cc}
Lucas Kania$^{\dagger,*}$
& Abhinav Chakraborty$^{\diamond,*}$
\end{tabular}
% \vspace{.12in}
\begin{tabular}{ccc}
Edward Kennedy$^{\S}$
& Larry Wasserman$^{\S, \ddagger}$
& Sivaraman Balakrishnan$^{\S,\ddagger}$
\end{tabular}
% \end{adjustwidth}

\vspace{.14in}

\begin{tabular}{c}
    $^\dagger$Department of Statistics and Data Science, University of Pennsylvania\\
    $^\diamond$Department of Statistics, Columbia University\\
    $^\S$Department of Statistics and Data Science, Carnegie Mellon University\\
	$^\ddagger$Machine Learning Department, Carnegie Mellon University \\[0.12in]
\end{tabular}
\begin{tabular}{cc}
    \texttt{lkania@upenn.edu},
    \texttt{ac4662@columbia.edu},
     \texttt{\{edward,larry,siva\}@stat.cmu.edu}
\end{tabular}
% \phantom{\footnote{Email}}

\vspace{.14in}

\today

\end{center}

\begin{abstract}
Powerful black-box predictive models have motivated many proposals for combining observed data with predictions to perform valid statistical inference. Despite this progress, the field lacks a unifying principle that explains how hypothesis tests should integrate data and predictions in a way that is both valid and efficient. In this work, we address this gap in the high-dimensional Gaussian sequence model. We characterize the information-theoretic limits of inference with black-box predictions when their accuracies are known and, for orthogonal predictions, when they are unknown. Building on these characterizations, we develop practical hypothesis tests that adapt to the unknown accuracies of the predictions while benefiting from strong alignment among them.
\end{abstract}

\begingroup
\renewcommand\thefootnote{*}
\footnotetext[1]{These authors contributed equally.}
\endgroup

\etocdepthtag.toc{mtchapter}
\etocsettagdepth{mtchapter}{section}
\etocsettagdepth{mtappendix}{none}
\etocsettagdepth{mtreferences}{section}
{
\normalsize
\parskip=0em
\renewcommand{\contentsname}{\normalsize Table of contents}
\tableofcontents
}

\pagebreak

\section{Introduction}

Modern scientific research often centers on evaluating whether available data supports or contradicts a given hypothesis. Statistical hypothesis testing provides a structured framework for addressing this question. At a high level, practitioners specify a reference distribution and apply a test to determine whether the data deviates from this baseline. Each test, either implicitly or explicitly, targets a specific kind of deviation and is designed to detect whether the data aligns more closely with an alternative distribution.

Early developments in parametric and semi-parametric statistics emphasized the design of tests with strong detection power against a small set of well-specified alternative distributions \citep{bickelTailormadeTestsGoodness2006,janssenGlobalPowerFunctions2000}. However, even minor deviations from the assumed alternative distributions can substantially reduce power \citep{breimanStatisticalModelingTwo2001}. To overcome this limitation, researchers developed nonparametric tests that guarantee detection power across broad classes of alternatives. Yet, spreading power across infinitely many possibilities often weakens performance.

In practice, effective testing strategies should strike a balance between these extremes: concentrating power on alternatives that the data suggests are plausible, while spreading power broadly when no strong alternative is apparent. One promising direction involves assumption-lean methods that leverage black-box predictions to concentrate power on data-driven alternatives adaptively.

Notable examples include prediction-powered inference \citep{angelopoulosPredictionPoweredInference2023,zrnic2024cross}, dimension-agnostic inference for M-estimators \citep{kimDimensionagnosticInferenceUsing2024,takatsuBridgingRoot$n$Nonstandard2025}, and universal inference \citep{wassermanUniversalInference2020}. These methods construct tests that integrate observed data with black-box approximations of the data distribution. Many rely on likelihood-ratio tests, either directly or via classifier-based approximations, as in likelihood-free inference \citep{dalmassoLikelihoodFreeFrequentistInference2023,gerberKernelBasedTestsLikelihoodFree2023}. However, there is no consensus on how to best fuse black-box predictions with empirical data.

In this work, we study how tests can optimally integrate black-box predictions of the data distribution with observed data in the Gaussian sequence model. Most importantly, we focus on test procedures that adapt to the unknown quality of the predictions. These adaptive tests automatically adjust their power allocation: they concentrate power on a narrow range of alternative distributions when the predictions are informative, and distribute it more broadly when the predictions are unreliable.

\noindent \textbf{Outline.} Section \ref{sec:overview_blackbox} summarizes our main results. Section \ref{sec:minimax} develops the minimax framework that we use to study adaptation to predictions with unknown quality. Section \ref{sec:oracle} characterizes the optimal performance of an oracle that knows the accuracy of each prediction. Section \ref{sec:orthogonal} then derives an optimal test that adapts to the unknown accuracy of orthogonal predictions. Section \ref{sec:arbitrary} extends this approach to arbitrary predictions. Finally, Section \ref{sec:simulations_blackbox} illustrates our results through simulations, and Section \ref{sec:discussion_blackbox} places our results in the broader context of the field and identifies directions for future research.

\noindent \textbf{Notation.} We write $a_n \lesssim b_n$ if there exists a positive constant $C$ such that $a_n \leq C \cdot b_n$ for all sufficiently large $n$. Analogously, $a_n \asymp b_n$ denotes that $a_n \lesssim b_n$ and $b_n \lesssim a_n$. We use $[k]=\{1,\dots,k\}$ to denote the set of the first $k$ natural numbers. Furthermore, we use the notation $I(\text{condition})$ to denote the indicator function that equals 1 if the condition is true and 0 otherwise. Whenever clear from context, we omit stating the random variable in expectations and variances: $E_{P}[T] =  E_{X\sim P}[T(X)]$ and $V_{P}[T] = V_{X\sim P}[T(X)]$. Furthermore, let $B_2^d(a,r)$ denote the $d$-dimensional Euclidean ball of radius $r$ centered at $a \in \R^d$: $B_2^d(a,r) = \{\theta \in \R^d: \norm{\theta-a}_2 \leq r\}$, and $S_2^{d-1}(a,r)$ denote the surface of the sphere of radius $r$ centered at $a \in \R^d$: $S_2^{d-1}(a,r) = \{\theta \in \R^d: \norm{\theta-a}_2 = r\}$. Throughout the paper, we use the inequality sign "$\geq$" in different contexts. Namely, for vectors $a,b \in R^d$, we state that $a \geq b$ if the inequality is satisfied entry-wise, while for functions $f,g$ whose domains are the reals, we state that $f\geq g$ if $f(r)\geq g(r)$ for all $r \in \R$. Finally, we let $\Corr(a,b)$ denote the cosine similarity between vectors $a,b\in \R^d$, that is, $\Corr(a,b)=\inner{a}{b}/(\norm{a}_2\cdot \norm{b}_2)$ if $a\not=0$ and $b\not=0$, where $\inner{a}{b}=a^Tb$ is the inner product between the vectors. If either of the vectors is zero, we let $\Corr(a,b)=0$.

\begin{figure}[t!]
\centering
\begin{subfigure}[t]{0.32\textwidth}
\centering
\includegraphics[width=\linewidth]{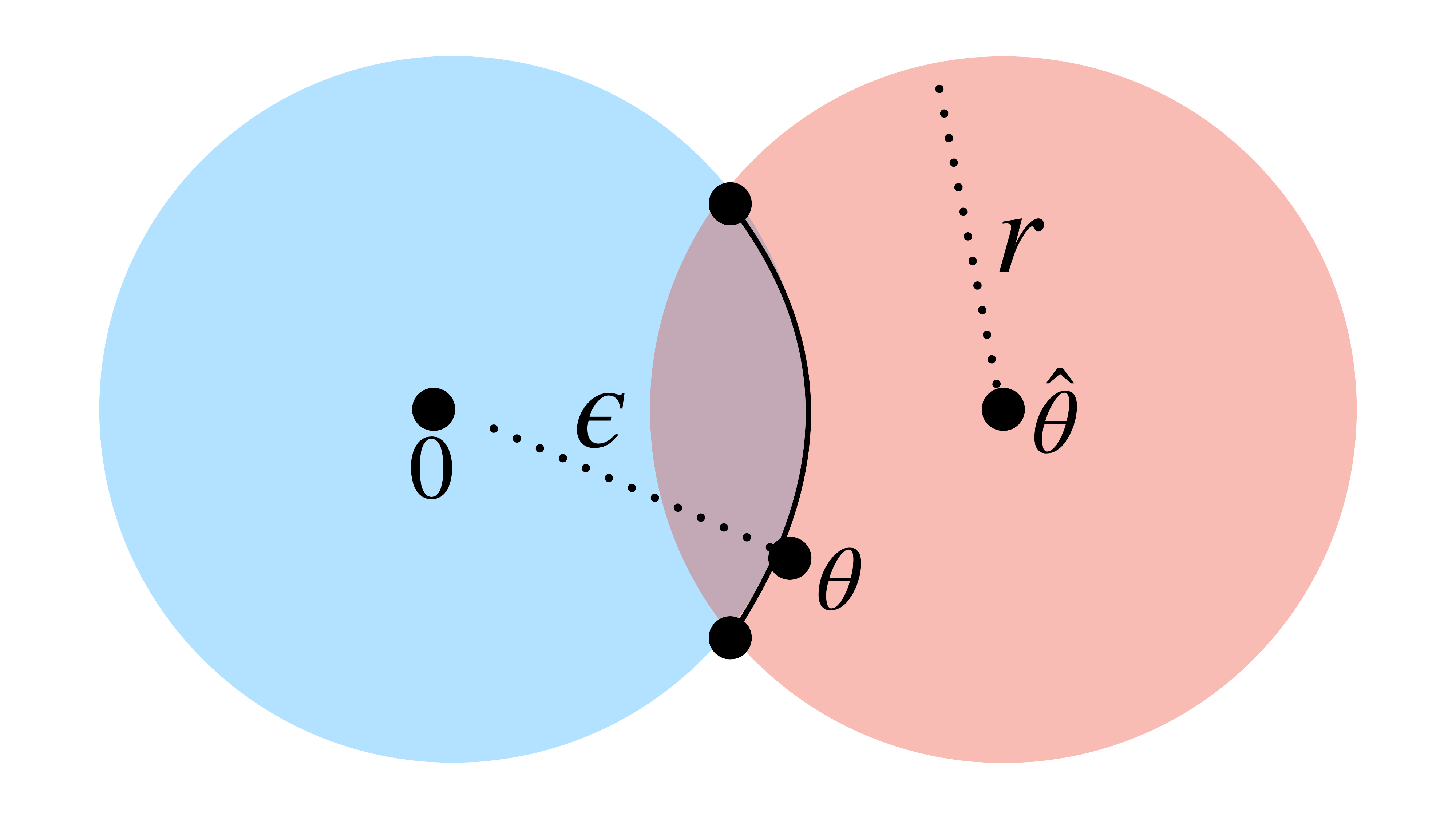}
\caption{Pictorial representation of distributions under the alternative hypothesis.}
\label{fig:intersection}
\end{subfigure}
\begin{subfigure}[t]{0.32\textwidth}
\centering
\includegraphics[width=0.8\linewidth]{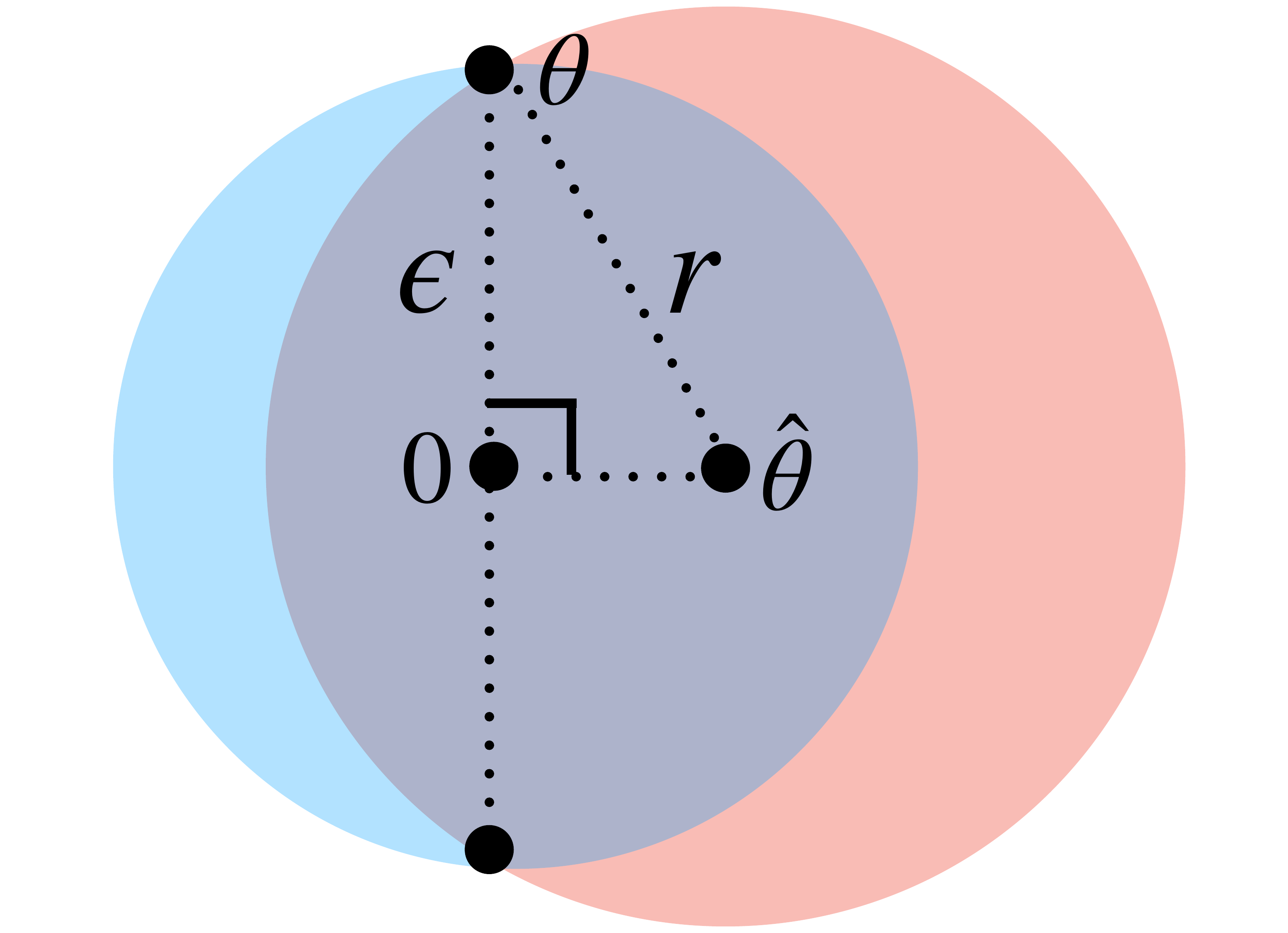}
\caption{Exactly half of $S_\epsilon$ is contained in $B_r$.}
\label{fig:half}
\end{subfigure}
\caption{Illustrative scenarios under the alternative hypothesis.}
\end{figure}

\subsection{Overview of main results}\label{sec:overview_blackbox}

We study how to optimally aggregate data with auxiliary black-box predictions in the Gaussian sequence model. This model provides a canonical setting for minimax hypothesis testing and a useful benchmark for more general nonparametric regression problems~\citep{ingsterNonparametricGoodnessofFitTesting2003,johnstone2019}. In this setting, one
is given access to $n$ observations from a standard $d$-dimensional Gaussian distribution $X_1,\dots,X_n \sim \Normal(\theta,I_d)$ with some unknown mean $\theta \in \mathbb{R}^d$, and $m$ predictions of the mean $\{\hat{\theta}_1,\dots,\hat{\theta}_m\}$ whose accuracy is unknown. The goal is to detect whether the mean of the data is zero ($\theta=0$), or far from it.  In this work, \begin{itemize}
\item We characterize the optimal testing strategy of an oracle test that has access to the unknown predictions' accuracies (Theorems \ref{thm:oracle_rates_single_prediction} and \ref{thm:oracle_rates_orthogonal_predictions}).
\item For orthogonal predictions, we develop a test that adapts to the unknown accuracy of the predictions and establish its optimality (Theorems \ref{thm:upperbound_adaptive_test_orthogonal_predictions} and \ref{thm:lowerbound_adaptive_test_orthogonal_predictions}).
\item For arbitrary predictions, we develop a test that adapts to the unknown accuracy of the predictions while benefiting from high cosine similarity among the predictions (Theorem \ref{thm:upperbound_adaptive_test_arbitrary_predictions}). This case is practically relevant because predictions often come from the same black-box predictive model and are not orthogonal.
\end{itemize}

% Finally, we remark that our results show that although functional estimation requires access to high-accuracy predictions, testing can use lower-accuracy predictions insofar as they have high cosine similarity with the mean of the underlying data distribution. %We further elaborate on the consequences of this fact for future research directions in Section \ref{sec:discussion_blackbox}. 

Throughout the paper, $X$ denotes the sample mean of the $n$ observations, so that $X \sim \mathcal{N}(\theta,I_d/n)$, where $I_d$ is the $d \times d$ identity matrix. We also assume that every prediction is nonzero, i.e., $\hat{\theta}_i \neq 0$ for all $i \in [m]$. Let us first consider the case where we do not have access to any black-box prediction. That is, one must decide if $\theta=0$ or if it is far away from the origin as measured by the $\ell_2$ norm: \begin{equation}\label{eq:gsn_seq_expo_large_r}
H_0: \theta = 0 \vs H_1: \norm{\theta}_2 \geq \epsilon.
\end{equation} It has been known since the work of \citet{ingsterAsymptoticallyMinimaxHypothesisI1993} that the hypotheses can be distinguished with nontrivial power if and only if $n \gtrsim \sqrt{d}\cdot \epsilon^{-2}$, which we call the prediction-free rate. The optimal test is based on the chi-squared test statistic: \begin{equation}\label{eq:expo_chi_squared_test}
\text{ reject $H_0$ if } \norm{X}_2^2 \text{  is large,}
\end{equation} which intuitively indicates that the strong dependence on the dimension of the data arises due to the optimal test checking for deviations from the null hypothesis in all directions.

In this work, we show that arbitrary predictions can substantially improve the prediction-free rate without requiring strong assumptions about their quality. Informative predictions help the test focus on a smaller set of directions where the null distribution may fail, which increases power when the predictions are reliable. Poor predictions, however, can point the test in unhelpful directions and reduce its effectiveness. We address this tension with tests that adapt to the unknown accuracy of the predictions, which we call \textit{adaptive tests}. As a result, the tests improve over the prediction-free rate when predictions are accurate, while still recovering that rate when predictions are uninformative.

Throughout this work, we consider black-box predictions as fixed non-random objects, which is commonplace in the model selection literature \citep{juditsky2000,nemirovski2000,tsybakov2003optimal}. Alternatively, one can adopt a conditional point of view \citep{balakrishnanFundamentalLimitsStructureAgnostic2023}, in which predictions are obtained from an independent sample of the data. Regardless of the chosen point of view, the focus of this work is on instance-wise optimality, that is, the optimal use of given fixed predictions with an unknown fixed accuracy. That is, we do not exploit any statistical guarantees of the process that generates the predictions, since such guarantees are rarely available in practice.  

To understand how a prediction can improve inference, let us first consider an oracle setting where we have access to a single black-box prediction $\hat{\theta} \in \R^d$, and we know its accuracy:
$\norm{\theta-\hat{\theta}}_2 \leq r$. We can incorporate the additional information in the hypotheses \eqref{eq:gsn_seq_expo_large_r} as follows:\begin{equation}\label{eq:hypotheses_oracle_single_prediction_expo}
H_0: \theta=0,
\vs H_1: \norm{\theta}_2 \geq \epsilon \textand \norm{\theta-\hat{\theta}}_2  \leq r.
\end{equation} Note that the accuracy of the prediction is not used in the definition of the null hypothesis, which implies that a test must be valid regardless of the quality of a black-box prediction. This setting parallels the attractive feature of methods such as universal or prediction-powered inference \citep{wassermanUniversalInference2020,angelopoulosPredictionPoweredInference2023}, whose validity does not depend on the quality of the predictions. Nevertheless, note that the prediction plays a key role in localizing the alternative hypothesis, which implies that the power of any test that distinguishes the hypotheses \eqref{eq:hypotheses_oracle_single_prediction_expo} must be affected by the accuracy of the prediction.

\begin{figure}[t!]
\centering
\begin{subfigure}[t]{0.31\textwidth}
  \centering
  \includegraphics[width=\linewidth]{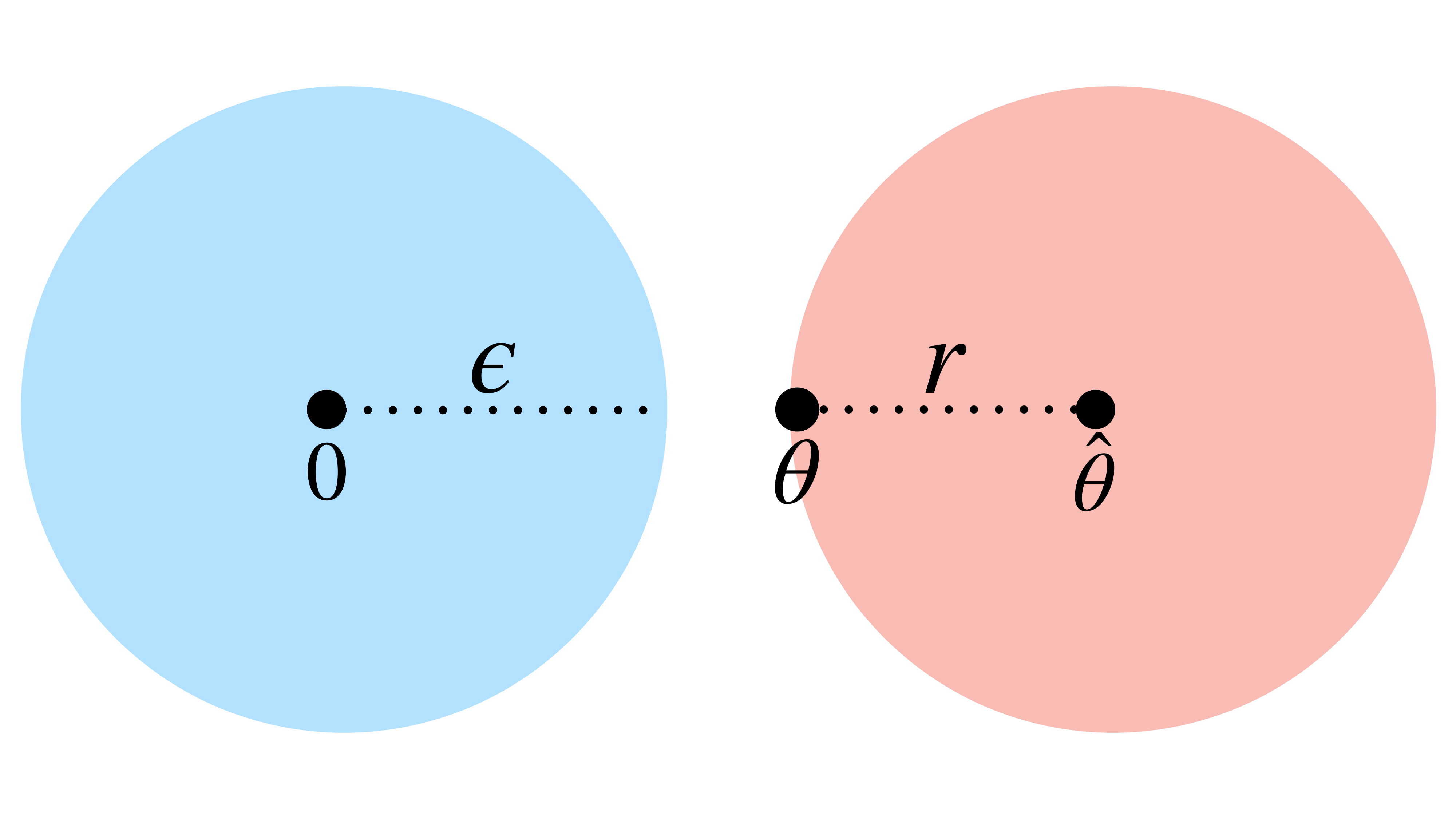}
  \caption{Extremely informative prediction when $S_\epsilon \cap B_r = \emptyset$.}
  \label{fig:disjoint}
\end{subfigure}%\hfill
\begin{subfigure}[t]{0.31\textwidth}
  \centering
  \includegraphics[width=\linewidth]{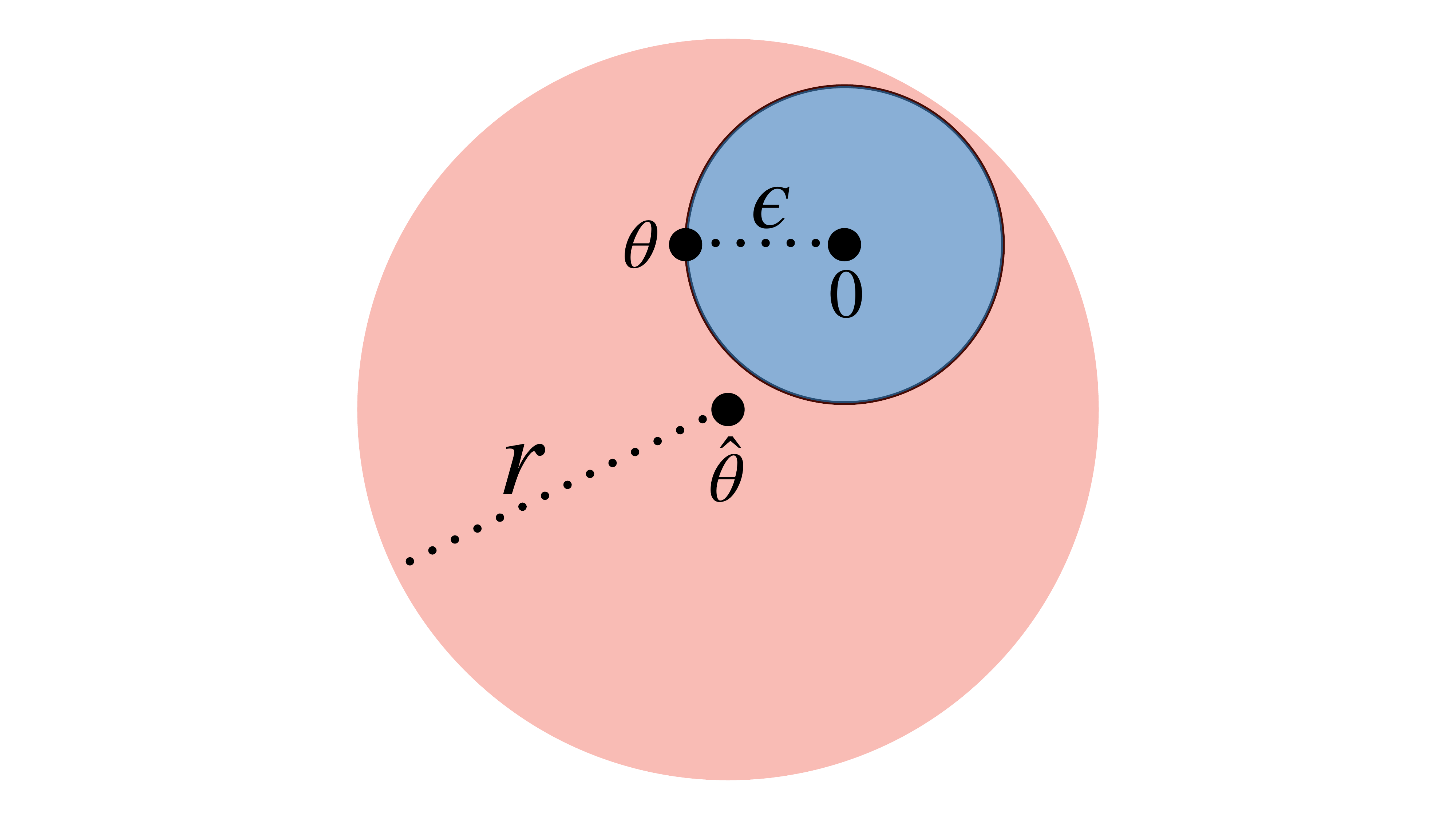}
  \caption{Uninformative prediction when $S_\epsilon \subset B_r$.}
  \label{fig:uninformative}
\end{subfigure}
\caption{Illustrative scenarios under the alternative hypothesis.}
\end{figure}

To intuitively understand how a prediction might change the hardness of the testing problem, let us consider an adversary that attempts to confuse a tester by choosing data distributions under the alternative hypothesis that are as close as possible to the null distribution. Going back to the alternative hypothesis in \eqref{eq:gsn_seq_expo_large_r}, without a prediction, an adversary may choose any point on the surface of the sphere $S_2^{d-1}(0,\epsilon)$, denoted by $S_\epsilon$ henceforth. A test must therefore search for deviations from the null hypothesis in every direction on this sphere. 

A prediction restricts this geometry. Under the alternative hypothesis of \eqref{eq:hypotheses_oracle_single_prediction_expo}, the mean of the data distribution must belong to the ball $B_2^d(\hat{\theta},r)$, denoted by $B_r$ henceforth. Thus, in the worst case, the adversary can only choose points in $S_\epsilon \cap B_r$, a spherical cap of $S_\epsilon$, as shown in Figure \ref{fig:intersection}. As this cap shrinks, testing becomes easier because fewer distributions under the alternative hypothesis are closer to the null distribution. Thus, a tester must look into fewer directions to detect deviations from the null distribution.

Let us consider two extreme cases to further clarify the role of the prediction. If $S_\epsilon \subseteq B_r$, see Figure \ref{fig:uninformative}, the prediction does not restrict the adversary. That is, hypotheses \eqref{eq:hypotheses_oracle_single_prediction_expo} reduce to \eqref{eq:gsn_seq_expo_large_r}, and we expect to attain the prediction-free rate. At the other extreme, if $S_\epsilon$ and $B_r$ are disjoint, see Figure \ref{fig:disjoint}, there is a unique point under the alternative hypothesis that is closest to the null distribution. Thus, one can attain parametric rates by performing a likelihood ratio test between the null distribution and the unique closest distribution under the alternative \citep{LeCamAsymptoticMethodsStatistical1986}: \begin{equation}\label{eq:theta_star}
\theta_*(r) = \arginf_{\theta \in B_r}\norm{\theta}_2^2. \end{equation}

In general, the hardness of distinguishing the hypotheses \eqref{eq:hypotheses_oracle_single_prediction_expo} must interpolate between these extreme cases. An oracle test allocates power according to the informativeness of the prediction. When the prediction is accurate, the test concentrates all the detection power along the predicted direction. Conversely, when the prediction is uninformative, the test spreads power across all directions. Intuitively, the key quantity governing this interpolation should be related to the amount of information that the prediction has about the data distribution. Henceforth, let $\hat{U}=\hat{\theta}/\norm{\hat{\theta}}_2$ be the direction of the prediction, and let $\Pi_{\hat{U}}$ denote the projection onto the span of the prediction. We let $\gamma_*(\epsilon,r)$ denote the length of the smallest projection onto the span of the prediction under the alternative hypothesis \begin{equation}\label{eq:projection_span_predictions_expo}
\gamma_*^2(\epsilon,r) = \inf_{\theta \in S_\epsilon \cap B_r}\norm{\Pi_{\hat{U}}\theta}_2^2.
\end{equation} The quantity is fundamentally linked to the amount of information about the data distribution that is contained in the projection onto the prediction's span. Henceforth, let $\Corr(\epsilon,r)$ denote the worst cosine similarity between the mean of the data $\theta$ and its projection onto the span of the prediction $\Pi_{\hat{U}}\theta$ under the alternative hypothesis: \begin{equation}\label{eq:worst_correlation}
\Corr(\epsilon,r)=\inf_{\theta \in S_\epsilon \cap B_r} \Corr\left(\Pi_{\hat{U}}\theta,\theta\right)  = \frac{\gamma_*(\epsilon,r)}{\epsilon} \quad \text{ if } S_\epsilon \cap B_r \neq \emptyset,
\end{equation} where we define the cosine similarity between two vectors $a$ and $b$ as $\Corr\left(a,b\right) = \inner{a}{b} / (\norm{a}_2 \cdot \norm{b}_2)$. Thus, as $\Corr(\epsilon,r)$ increases, the projection onto the prediction's span becomes more informative of the mean of the data distribution. 

The following theorem summarizes the oracle performance, which, by virtue of knowing the accuracy of the prediction, can decide optimally when to utilize it. We require that $B_r \not\subset B_\epsilon^\circ$, where $B_\epsilon^\circ$ is the interior of the ball $B_\epsilon$, i.e., $B_\epsilon^\circ=B_\epsilon \setminus S_\epsilon$, so that the alternative hypothesis \eqref{eq:hypotheses_oracle_single_prediction_expo} is not empty. 

\begin{theorem}\label{thm:oracle_rates_single_prediction} If $B_r \not\subset B_\epsilon^\circ$, the hypotheses \eqref{eq:hypotheses_oracle_single_prediction_expo} can be tested with nontrivial power if and only if \begin{equation}
n \gtrsim \begin{dcases}
\sqrt{d} \cdot \epsilon^{-2} &\textif  \Corr(\epsilon,r) \leq c \cdot d^{-1/4} \textand S_\epsilon \cap B_r \neq \emptyset.
\\
\gamma_*^{-2}(\epsilon,r) &\textif \Corr(\epsilon,r) \in \left[c \cdot d^{-1/4},1\right] \textand S_\epsilon \cap B_r \neq \emptyset. 
\\
\norm{\theta_*(r)}_2^{-2} &\text{if } S_\epsilon \cap B_r = \emptyset.
% \label{eq:disjoint_neigh}
\end{dcases} \end{equation} where $c$ is a positive constant. Furthermore, the rates depend continuously on the intersection $S_\epsilon \cap B_r$, and are ordered from worst to best, from top to bottom.
\end{theorem}

In broad strokes, Theorem \ref{thm:oracle_rates_single_prediction} states that if the worst cosine similarity between the mean of the data distribution and its projection onto the prediction is low, then the prediction is not informative enough, and we obtain the prediction-free rate. Geometrically, the first case can be foreseen by noting that when $\Corr(\epsilon,r) \lesssim d^{-1/4}$, an adversary is allowed to choose a data distribution under the alternative hypothesis whose mean is orthogonal to the prediction, see Figure \ref{fig:half}, making the prediction uninformative. 

However, if $\Corr(\epsilon,r)$ is large enough, in the second case, the oracle can project the data onto the prediction ($\Pi_{\hat{U}}X$) and make a decision based on this one-dimensional quantity, avoiding any dependence on the ambient dimension of the data. Finally, in the third case, when a prediction is extremely informative, $S_\epsilon$ and $B_r$ are disjoint, and the testing problem reduces to the binary test between $H_0: \theta = 0$ and $H_1: \theta=\theta_*$, as defined in \eqref{eq:theta_star}, which yields dimension-independent parametric rates. Hereafter, we will omit the case where $S_\epsilon$ and $B_r$ do not intersect from the introduction, since it is related to the well-studied binary hypothesis testing problem and is relevant to practitioners only when predictions are extremely accurate.

Practitioners may have access to multiple predictions, each of which may contain useful information about the data distribution. To model this setting, we extend the hypotheses in \eqref{eq:hypotheses_oracle_single_prediction_expo} to multiple predictions. Let $r \in \R_+^m$ denote the vector whose entries quantify the accuracy of the $m$ black-box predictions. We then define the hypotheses as follows:
\begin{equation}\label{eq:hypotheses_oracle_multiple_predictions_expo}
H_0: \theta=0,
\vs H_1: \norm{\theta}_2 \geq \epsilon(r) \textand \norm{\theta-\hat{\theta}_i}_2  \leq r_i\for i \in [m].
\end{equation} Note that we have written explicitly $\epsilon$ as a function of the accuracy vector $r \in \R_+^m$, as one can intuit that the more accurate the predictions are, the smaller $\epsilon(r)$ should be. We call the function $r \mapsto \epsilon(r)$ a separation curve, as it quantifies the separation between the null and alternative hypotheses.

Henceforth, we let $B_r$ denote the intersection of all prediction neighborhoods, $B_r = \cap_{i=1}^m B(\hat{\theta}_i,r_i)$, and let $\Pi_{\hat{U}}\theta$ in \eqref{eq:projection_span_predictions_expo} and \eqref{eq:worst_correlation} denote the projection onto the span of the predictions, where $\hat{U}$ is a matrix in $\R^{d\times m}$, whose $i$-th column is given by the $i$-th predicted direction $\hat{u}_i = \hat{\theta}_i/\norm{\hat{\theta}_i}_2$. Since the oracle knows the accuracy of all predictions, it can aggregate them optimally and perform a one-dimensional test by combining their projections: \begin{equation}\label{eq:optimal_aggregation}
\text{ reject $H_0$ if }\inner{\Pi_{\hat{U}}X}{W} \text{ is large,}
\end{equation} where $W$ is a vector of weights that depends on the vector of accuracies $r$ and the separation curve. Thus, when predictions are informative, the performance of the oracle test is independent of the dimension of the data.

\begin{theorem}[Optimal performance of oracle with access to arbitrary predictions]\label{thm:oracle_rates_orthogonal_predictions} If $\Rank \hat{U} \leq d/2$, the hypotheses \eqref{eq:hypotheses_oracle_multiple_predictions_expo} can be tested with nontrivial power if and only if  \begin{equation}\label{eq:oracle_rates_orthogonal_predictions}
\epsilon(r) \gtrsim \inf\left\{\epsilon\geq0: n \gtrsim \frac{1}{\gamma_*^2(\epsilon,r)}\wedge \frac{\sqrt{d}}{\epsilon^2} \textand S_\epsilon \cap B_r \not= \emptyset \right\},
\end{equation} where we omitted the case $S_\epsilon \cap B_r = \emptyset$.
\end{theorem}

We remark that, albeit written succinctly, the rates in Theorem \ref{thm:oracle_rates_orthogonal_predictions} match those in Theorem \ref{thm:oracle_rates_single_prediction} when only one prediction is available. Furthermore, the interpretation remains the same as in the single-prediction setting. Under the alternative hypothesis, when the worst cosine similarity between the mean of the data distribution and its projection onto the span of the predictions \eqref{eq:worst_correlation} is of order $o(d^{-1/4})$, then predictions do not offer any benefit, and the oracle test disregards them. We note that Theorems \ref{thm:oracle_rates_single_prediction} and \ref{thm:oracle_rates_orthogonal_predictions} follow from Lemmas \ref{lemma:upperbound_oracle_test_arbitrary_predictions} and \ref{lemma:lowerbound_oracle_test_arbitrary_predictions} in Section \ref{sec:oracle}, where we further discuss them.

\noindent\textbf{Predictions of unknown quality.} In practice, we do not know the accuracy of the predictions. Therefore, the main goal is to construct tests that do not require that information to operate optimally. We formalize this requirement through the adaptivity framework used in minimax theory \citep{ingsterAdaptiveChisquareTests2000,gineMathematicalFoundationsInfinitedimensional2016}, where one aims to distinguish the hypotheses \begin{equation}\label{eq:hypotheses_multiple_predictions_expo}
H_0: \theta = 0 \vs H_1: \exists r \in \R_+^m \st \norm{\theta}_2 \geq \epsilon(r) \textand \norm{
\theta-\hat{\theta}_i
}_2 \leq r_i  \for i \in [m].
\end{equation} Unlike the oracle setting \eqref{eq:hypotheses_oracle_multiple_predictions_expo}, the alternative hypothesis in \eqref{eq:hypotheses_multiple_predictions_expo} does not reveal the accuracy of the predictions. A test for \eqref{eq:hypotheses_multiple_predictions_expo} must therefore adapt to this unknown accuracy, so we call it an \textit{adaptive} test. The key question is the \textit{cost of adaptation}: how many additional observations does the test require because it does not know the true prediction accuracy?

\noindent\textbf{Orthogonal predictions.} To illustrate this phenomenon, let us first consider the case of orthogonal predictions when $m \leq d$ to draw a comparison with Theorem \ref{thm:oracle_rates_orthogonal_predictions}. The oracle avoids any dimension dependence due to optimally aggregating predictions into a one-dimensional weighted projection \eqref{eq:optimal_aggregation} using the accuracies of the predictions. An adaptive test lacks this information, so it cannot perform the aggregation. However, since the predictions are orthogonal, it is reasonable for the adaptive test to monitor for deviations from the null distribution in all predicted directions: \begin{equation}\label{eq:projection_onto_predictions_span}
\text{ reject $H_0$ if }\norm{\Pi_{\hat{U}}X}_2^2 \text{ is large, where } \norm{\Pi_{\hat{U}}X}_2^2=\sum_{i=1}^m\ \norm{\Pi_{\hat{u}_i}X}_2^2.
\end{equation} Note that all projections are equally weighted; thus, the test rejects whenever the data deviates from the null distribution on any of the predicted directions. Given that we hedge our bets in all predicted directions equally, our performance is expected to depend on the number of predictions.

\begin{theorem}[Performance of adaptive test with access to orthogonal predictions]\label{thm:upperbound_adaptive_test_orthogonal_predictions} If $m \leq d/2$ and predictions are orthogonal, the hypotheses \eqref{eq:hypotheses_multiple_predictions_expo} can be tested with nontrivial power if  
\begin{equation}
\epsilon(r) \gtrsim \inf\left\{\epsilon\geq 0: n \gtrsim \frac{\sqrt{m}}{\gamma_*^2(\epsilon,r)}\wedge \frac{\sqrt{d}}{\epsilon^2} \textand S_\epsilon \cap B_r \not= \emptyset \right\},
\end{equation} where we omitted the case $S_\epsilon \cap B_r = \emptyset$.
\end{theorem}

Theorem \ref{thm:upperbound_adaptive_test_orthogonal_predictions} shows that the adaptive test can require more observations than the oracle test in Theorem \ref{thm:oracle_rates_orthogonal_predictions}. The ratio between these sample sizes, called the \textit{adaptivity gap} in the literature, can be as large as $\sqrt{m}$. This gap arises because the adaptive test \eqref{eq:projection_onto_predictions_span} must search for deviations throughout the span of the predictions, whereas the oracle test \eqref{eq:optimal_aggregation} uses the predictions' accuracies to aggregate the predicted directions into a single direction. When only a single prediction is available ($m=1$), both tests search for deviations along the same predicted direction and therefore have the same rate-wise performance. As $m$ grows, however, orthogonal predictions increase the number of directions along which the adaptive test \eqref{eq:projection_onto_predictions_span} must search, while the oracle test \eqref{eq:optimal_aggregation} continues to detect deviations along a single aggregated direction. Thus, having many orthogonal predictions can increase the cost of adaptation rather than necessarily improve performance.

\noindent\textbf{Is there a free lunch?} Nevertheless, comparing the performance of the adaptive test in Theorem \ref{thm:upperbound_adaptive_test_orthogonal_predictions} to the prediction-free rate achieved by the chi-squared test \eqref{eq:expo_chi_squared_test} that ignores predictions might lead one to believe that using predictions is always beneficial. This advantage is true only in a rate-wise sense. That is, if the predictions are high-quality, the adaptive test will be more sample-efficient than the chi-squared test. However, if the predictions are low-quality, the adaptive test's sample efficiency will be at most a constant factor worse than that of the chi-squared test. 

This phenomenon is exemplified by the simulation shown in Figure \ref{fig:single_prediction}, where a single prediction is made available to each test. Details are deferred to Section \ref{sec:simulations_blackbox}; however, note that when the prediction is of high quality, the \textit{adaptive test} substantially improves power relative to the chi-squared test that ignores the prediction. However, when the prediction is low-quality, the power of the adaptive test is worse than the chi-squared test's power by a constant. Furthermore, note that when the prediction is high quality, the performance of the adaptive test is independent of the ambient dimension of the data.

\begin{figure}[!b]
\centering
\includegraphics[width=\linewidth]{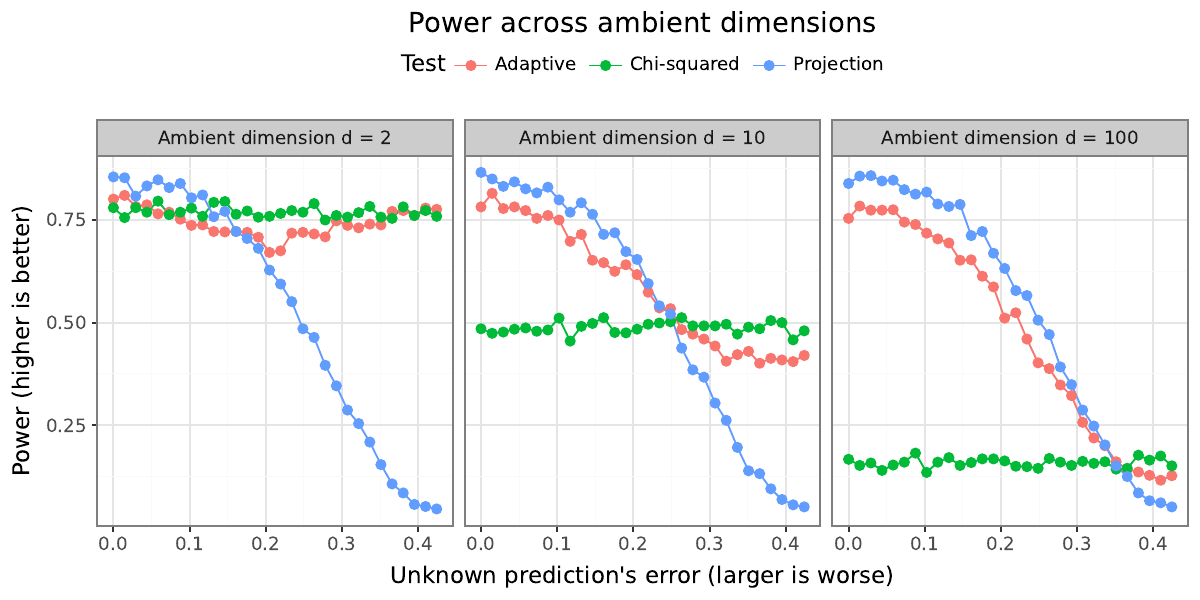}
\caption{Comparison between the chi-squared test \eqref{eq:chi_squared_test}, which ignores the prediction, the projection test \eqref{eq:projection_test}, which always uses the prediction, and the adaptive test \eqref{eq:bonferroni_test_with_standard_projection} that adapts to the unknown prediction's error.}
\label{fig:single_prediction}
\end{figure}

\noindent\textbf{Optimality of the adaptive test.} Although one can improve Theorem \ref{thm:upperbound_adaptive_test_orthogonal_predictions} for any fixed $r \in \R_+^m$, the dependence on the number of predictions is optimal among separation curves satisfying a natural invariance property. Namely, looking at the separation curves for both oracle (Theorem \ref{thm:oracle_rates_orthogonal_predictions}) and adaptive (Theorem \ref{thm:upperbound_adaptive_test_orthogonal_predictions}) tests, it is clear that they assign the same separation to accuracy vectors that induce projections of the same length under the alternative hypothesis. The following property quantifies this idea.

\begin{property*}[Property \ref{property:separation_curve_orthogonal_predictions} in Section \ref{sec:orthogonal}. Same separation for equivalent accuracy vectors] Let $r\to\epsilon_*(r)$ be a separation curve. If $\epsilon_*(r)=\epsilon$ for $r\in \R_+^m$, it holds that $\epsilon_*(r')=\epsilon$ for all $r'\in \R_+^m$ such that $\gamma_*(\epsilon,r)=\gamma_*(\epsilon,r')$.
\end{property*}

In other words, the property states that accuracy vectors that lead to equally hard worst-case distributions under the alternative hypotheses get assigned the same separation. The next theorem shows that, under Property \ref{property:separation_curve_orthogonal_predictions}, the adaptivity gap must grow with the ambient dimension of the prediction's span for orthogonal predictions.

\begin{theorem}\label{thm:lowerbound_adaptive_test_orthogonal_predictions} If $m \leq d/2$ and the predictions are orthogonal, for any separation curve that satisfies Property \ref{property:separation_curve_orthogonal_predictions}, the hypotheses \eqref{eq:hypotheses_multiple_predictions_expo} can be tested with nontrivial power only if \begin{equation}
\epsilon(r) \gtrsim \inf\left\{\epsilon\geq 0: n \gtrsim \frac{\sqrt{m}}{\gamma_*^2(\epsilon,r)}\wedge \frac{\sqrt{d}}{\epsilon^2} \textand  S_\epsilon \cap B_r \neq \emptyset \right\}.
\end{equation}
\end{theorem}

The proof is based on a variant of Le Cam's random alternative distributions \citep{tsybakovIntroductionNonparametricEstimation2009}, which must be supported on the appropriate spherical cap under the alternative hypothesis. These constructions differ substantially from those used in structure-agnostic functional estimation \citep{balakrishnanFundamentalLimitsStructureAgnostic2023}, and they explain the faster rates achieved by the testing procedures proposed in this work in comparison to the functional estimation rates.

\noindent\textbf{Arbitrary predictions.} Theorem \ref{thm:upperbound_adaptive_test_orthogonal_predictions} is most relevant when predictions come from different data sources, each of which provides independent information about the data distribution. In many applications, however, predictions are generated by the same model via resampling and are not expected to be orthogonal. 
 
A test such as \eqref{eq:projection_onto_predictions_span} need not be optimal when predictions are not orthogonal, because it assigns equal weight to every predicted direction. We should rather emphasize detecting deviations in directions most aligned with the majority of predicted directions, and essentially ignore those directions aligned with only a few predictions. 

Henceforth, let $\lambda_1\geq \dots \geq \lambda_k > 0$ denote the nonzero singular values of $\hat{U}$, and $\tilde{u}_1,\dots,\tilde{u}_k$ denote the corresponding orthonormal left singular vectors of $\hat{U}$. The squared singular values $\lambda_1^2\geq \dots \geq \lambda_k^2 > 0$ are the nonzero eigenvalues of the predictions' cosine similarity matrix $\Sigma = \hat{U}^T\hat{U}$. Our adaptive test relies on an operation that we call the \textit{intrinsic projection} of the data onto the span of the predictions: \begin{equation}\label{eq:intrinsinc_projection_span_of_predictions}
\text{ reject  }H_0 \text{ if }\norm{\Pi_{\hat{U}}^{\intrinsic} X}_2^2 \text{ is large, where } \norm{\Pi_{\hat{U}}^{\intrinsic} X}_2^2 = \sum_{i=1}^k \frac{\lambda_i^2}{\lambda_1^2}\cdot \norm{\Pi_{\tilde{u}_i}X}_2^2 .
\end{equation} The weighting automatically downweights directions that align poorly with the predictions. As a result, the test focuses on fewer directions in which it must detect deviations from the null hypothesis. The intrinsic projection test owes its name to its performance not depending on the ambient dimension of the span of the predictions, but on a notion of its intrinsic dimension. 

We define the intrinsic dimension of $\Sigma^2$ as  $\intrinsic \Sigma^2 = \trace \Sigma^2/\norm{\Sigma^2}_{\op} $ where the numerator is the trace of $\Sigma^2$ and the denominator is its operator norm \citep{ipsenStableRankIntrinsic2024}. This quantity measures the effective degree of orthogonality among the predictions. For instance, the intrinsic dimension of $\Sigma^2$ equals $1$ when all predictions are parallel, and equals $m$ when they are orthogonal. In this work, we show that the performance of an adaptive test depends on the intrinsic dimension of $\Sigma^2$ rather than the ambient dimension of the span of the predictions.

\begin{theorem}[Performance of adaptive test with access to arbitrary predictions]\label{thm:upperbound_adaptive_test_arbitrary_predictions} The hypotheses \eqref{eq:hypotheses_multiple_predictions_expo} can be tested with nontrivial power if  \begin{equation}
\epsilon(r) \gtrsim
\inf\left\{\epsilon\geq 0:
n \gtrsim
\frac{\sqrt{\intrinsic \Sigma^2}}{\tilde{\gamma}_*^2(\epsilon,r)}
\wedge
\frac{\sqrt{d}}{\epsilon^2}\textand  S_\epsilon \cap B_r \not= \emptyset\right\} 
\end{equation}where $\tilde{\gamma}^2_*(\epsilon,r)=\inf_{\theta \in S_\epsilon \cap B_r} \norm{\Pi^{\intrinsic}_{\hat{U}}\theta}_2^2$ and we omitted the case $S_\epsilon \cap B_r = \emptyset$.
\end{theorem}

Note that Theorem \ref{thm:upperbound_adaptive_test_orthogonal_predictions} depends on the ambient dimension of the predictions' span ($\Rank \hat{U}$), whereas Theorem \ref{thm:upperbound_adaptive_test_arbitrary_predictions} depends on the intrinsic dimension of their squared cosine similarity matrix ($\intrinsic \Sigma^2$). It always holds that  $\intrinsic \Sigma^2 \leq \Rank \hat{U}$, with equality being achieved when all predictions are orthogonal. Thus, orthogonal predictions can be considered a worst-case scenario for Theorem \ref{thm:upperbound_adaptive_test_arbitrary_predictions}.

Under orthogonality, Theorem \ref{thm:upperbound_adaptive_test_arbitrary_predictions} recovers Theorem \ref{thm:upperbound_adaptive_test_orthogonal_predictions}. Its main purpose, however, is to show that non-orthogonal predictions can improve adaptivity. When the predictions share substantial information, the adaptive test depends on the lower intrinsic dimension of the squared cosine similarity matrix rather than on the full ambient dimension of the predictions' span. This distinction matters in practice because predictions produced from the same predictive model often are not orthogonal, which can substantially reduce the intrinsic dimension. As an illustration, Figure \ref{fig:correlated_predictions} displays a simulation designed to be adversarial to the tests using intrinsic projection. Nevertheless, if the intrinsic dimension is low, adaptive methods based on intrinsic projection can outperform all competing methods. We defer the details to Section \ref{sec:simulations_blackbox}.

\begin{figure}[!t]
\centering
\includegraphics[width=\linewidth]{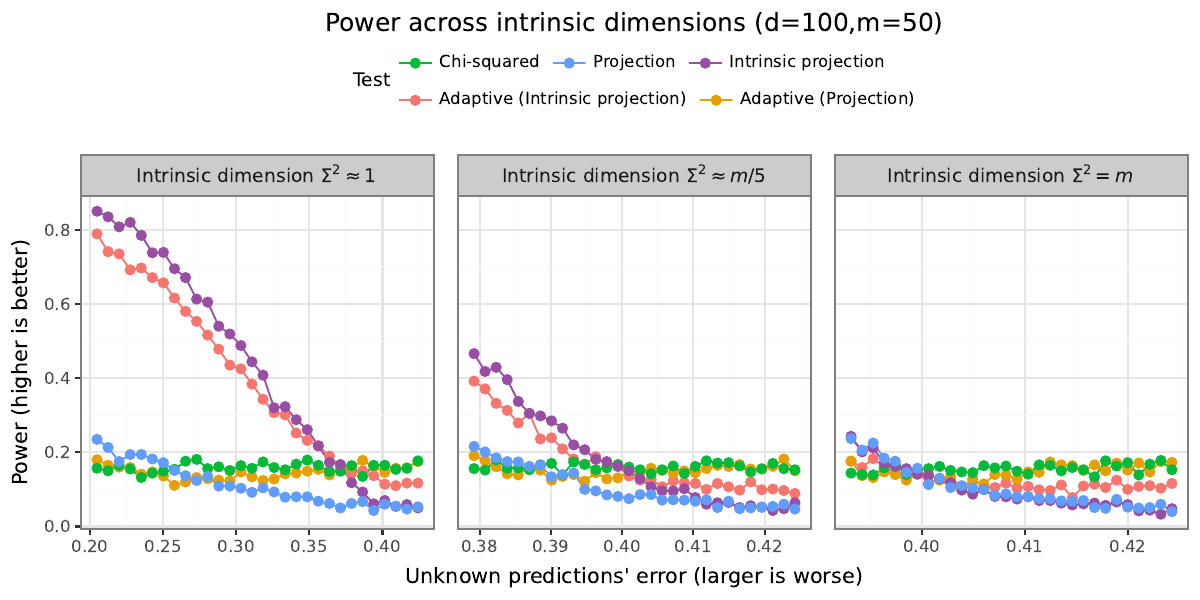}
\caption{Comparison between the chi-squared test \eqref{eq:chi_squared_test}, which ignores predictions, tests that always use predictions either via standard \eqref{eq:projection_test} or intrinsic \eqref{eq:intrinsic_projection_test} projection, and adaptive tests that use standard \eqref{eq:bonferroni_test_with_standard_projection} or intrinsic projection \eqref{eq:bonferroni_test_with_intrinsic_projection}.}
\label{fig:correlated_predictions}
\end{figure}

\subsection{Related work}

Our work is most directly related to the literature on property testing and goodness-of-fit testing with black-box predictions \citep{jankovaGoodnessoffitTestingHigh2020,lundborgProjectedCovarianceMeasure2024a,aliakbarpourOptimalAlgorithmsAugmented2024}. These works establish minimax limits for their respective problems. We complement them by studying adaptation: for orthogonal predictions, we characterize the cost of not knowing their accuracies.

To study this question, we draw on the framework of adaptive hypothesis testing \citep{spokoiny1996adaptive,ingsterAdaptiveChisquareTests2000}, which considers settings in which optimal testing procedures depend on unknown features of the alternative hypothesis, such as function smoothness. Adaptive tests avoid requiring these parameters and often retain nearly the same power as procedures that know them \citep{ingsterNonparametricGoodnessofFitTesting2003,gineMathematicalFoundationsInfinitedimensional2016}. Our setting departs from this usual phenomenon: for orthogonal predictions, the cost of adaptation is non-negligible and can grow with the number of available predictions.

More broadly, our paper contributes to the literature on structure-agnostic inference. Prior work has developed this perspective for functional estimation \citep{balakrishnanFundamentalLimitsStructureAgnostic2023,guOpenProblemStructureAgnostic2025,jinSharpStructureAgnosticLower2026} and extended it to causal inference \citep{jinItsHardBe2025,jinStructureagnosticOptimalityDoubly2024,bonviniDoublyrobustInferenceOptimality2024}. We extend this perspective to hypothesis testing with black-box predictions.

Our work is also closely connected to prediction-powered inference and related semi-supervised methods \citep{chakraborttyEfficientAdaptiveLinear2018,angelopoulosPredictionPoweredInference2023,chakraborttySemiSupervisedQuantileEstimation2024,chakraborttyGeneralFrameworkTreatment2024,zouGeneralizedPredictionPoweredInference2026,xuMorePredictionsImprove2026}. In particular, our results support the asymptotic conclusion of previous work \citep{zhangSemisupervisedInferenceGeneral2019,angelopoulosPPIEfficientPredictionPowered2024} that prediction-based estimators should be combined with estimators that do not use predictions to control worst-case error while preserving efficiency.

A related line of work studies model aggregation, beginning with \citet{tsybakov2003optimal}, which considers how to combine multiple candidate predictors or estimators to achieve near-oracle performance. This perspective was subsequently developed through mirror-averaging and exponential-weighting methods \citep{juditsky2008learning,dalalyan2007aggregation} and connected to sparse aggregation \citep{rigollet2012sparse}. Although this literature focuses primarily on prediction risk rather than hypothesis testing, it is conceptually relevant because our procedures likewise aggregate information across multiple black-box predictions.

Finally, our treatment of nonorthogonal predictions connects to a broader literature on statistical procedures whose performance depends on the intrinsic dimension or effective rank rather than the ambient dimension of the data. Existing results include asymptotic tests whose performance scales with the intrinsic dimension \citep{kimDimensionagnosticInferenceUsing2024}, as well as concentration bounds for self-adjoint operators and random matrices whose rates are governed by the intrinsic dimension or effective rank \citep{koltchinskiiAsymptoticsConcentrationBounds2015,weiEstimationCovarianceStructure2017,zhivotovskiyDimensionfreeBoundsSums2024,martinezIntrinsicdimensionEmpiricalBernstein2026}.

\section{Adaptation in the minimax framework}\label{sec:minimax}

Given data $X \sim P_\theta$ drawn from the distribution $P_\theta = \Normal(\theta, I_d/n)$ with unknown mean $\theta$, we wish to test whether $\theta=0$. The predictions may encode useful structural information about $\theta$, and an effective procedure should exploit this information to increase power. At the same time, the validity of the test must not depend on the accuracy of the prediction. In particular, the testing procedure should remain well defined even when the prediction accuracy is unknown. In this section, we formalize these desiderata within a minimax framework.

Ultimately, we aim to develop methods that optimally adapt to the unknown accuracy of the black-box prediction. However, to establish a baseline performance, let us first consider the case where the accuracy, denoted by $r$, is known. This setting is usually referred to as an \textit{oracle} setting in the literature, since algorithms may use accuracy to make decisions. That is, we consider the hypotheses \eqref{eq:hypotheses_oracle_multiple_predictions_expo}: \begin{equation}\label{eq:testing_with_known_accuracy}
H_0: P_\theta = P_0\vs H_1: P_\theta \in \mathcal{P}(\epsilon,r)
\end{equation} where $\mathcal{P}(\epsilon,r)=\left\{P_\theta: \norm{\theta}_2\geq \epsilon(r) \textand \norm{\theta-\hat{\theta}_i}_2 \leq r_i \for i \in [m]\right\}$ is the localized set of alternative distributions with respect to accuracy vector $r \in \R^m_+$ and the predictions $\hat{\theta}_1,\dots,\hat{\theta}_m$. However, we omit the dependence on the predictions to simplify the notation throughout the paper. This localization may improve power by reducing the size of the alternative set, while the test's validity remains unaffected since the null hypothesis does not depend on the accuracy of the prediction.

We are interested in characterizing the hardness of distinguishing the hypotheses \eqref{eq:testing_with_known_accuracy} with respect to the quality of the predictions. Let $
\Psi_\alpha = \left\{\psi : P_0(\psi(X)=1)\leq \alpha\right\}$ denote the set of all valid tests with type-I error at most $\alpha$. For any valid test $\psi$, we define its risk under the alternative hypothesis as the worst-case type-II error: $$
R(\mathcal{P}(\epsilon,r),\psi) = \sup_{P \in \mathcal{P}(\epsilon,r)} P(\psi(X)=0).$$ In particular, if a  valid test $\psi$ controls the worst-case type-II error $R(\mathcal{P}(\epsilon,r),\psi) < \beta$, for $\beta \in (0,1-\alpha)$, we say that $\psi$ is powerful. The minimax risk is then the best performance that we can achieve across all valid tests \begin{equation}\label{eq:minimax_risk_definition}
R_*(\mathcal{P}(\epsilon,r)) = \inf_{\psi \in \Psi_\alpha}R(\mathcal{P}(\epsilon,r),\psi).\end{equation} We can now define the \textit{optimal oracle separation}, also known as critical oracle separation or minimax oracle separation, as the smallest separation between the hypotheses \eqref{eq:testing_with_known_accuracy} for which there exists a valid and powerful test. \begin{definition}[Optimal oracle separation]\label{def:oracle_critical_separation} Let us fix $r \in \R_+^m$. We say that $\epsilon_*(r)$ is an oracle separation for hypotheses \eqref{eq:testing_with_known_accuracy}, if there exists a valid test that is powerful for distinguishing the null and alternative hypotheses that are $\epsilon_*(r)$ close. That is, $R_*(\mathcal{P}(\epsilon_*,r)) \leq \beta$. 

Furthermore, we say that $\epsilon_*$ is an optimal oracle separation, if no valid test can be powerful for hypotheses \eqref{eq:testing_with_known_accuracy} if $\epsilon(r) < \epsilon_*(r)$. That is, $R_*(\mathcal{P}(\epsilon,r)) > \beta$ if $\epsilon(r) < \epsilon_*(r)$.
\end{definition} 

We now turn to the realistic setting in which the testing algorithm does not know the prediction accuracy. Thus, the goal is to distinguish the hypotheses \eqref{eq:hypotheses_multiple_predictions_expo}: \begin{equation}\label{eq:testing_with_unknown_accuracy}
H_0: P_\theta=P_0 \vs H_1: \exists r \in \R_+^m \st P_\theta \in \mathcal{P}(\epsilon,r).
\end{equation} Note that under the alternative hypothesis, there is no fixed accuracy vector, which accounts for the fact that in practice we do not know the accuracy of the predictions. Analogously to Definition \ref{def:oracle_critical_separation}, we define an \textit{optimal separation curve} as the smallest separation between hypotheses \eqref{eq:testing_with_unknown_accuracy} for which there exists a valid and powerful test.
\begin{definition}[Optimal separation curve]\label{def:critical_separtion_curve} We say that $\epsilon_*$ is a separation curve for hypotheses \eqref{eq:testing_with_unknown_accuracy}, if there exists a valid test that distinguishes alternative hypotheses that are $\epsilon_*$ away. That is, $
R_*(\mathcal{P}(\epsilon_*)) \leq \beta$ where $\mathcal{P}(\epsilon_*) = \bigcup_{r\in \R_+^m}\mathcal{P}(\epsilon_*,r)$.

Furthermore, we say that $\epsilon_*$ is an optimal separation curve among a set of functions $\Epsilon$, if for any $\epsilon \in \Epsilon$ that significantly improves over $\epsilon_*$, no valid test can be powerful for hypotheses \eqref{eq:testing_with_unknown_accuracy}. That is, if there exists a positive constant $C$ that depends only on $\alpha$ and $\beta$  such that for every $\epsilon \in \Epsilon$, if there exists $r \in \R_+^m$ satisfying $\epsilon(r) < C\cdot \epsilon_*(r)$, then 
$R_*(\mathcal{P}(\epsilon)) > \beta$ where $\mathcal{P}(\epsilon) = \bigcup_{r\in \R_+^m}\mathcal{P}(\epsilon,r)$. 
\end{definition}

\section{Oracle inference with multiple predictions}\label{sec:oracle}

With a single prediction, there is no adaptivity gap. That is, a data-driven test can achieve the same rates as an oracle test that knows the predictions' accuracy. The structure of the problem forces a binary decision: either exploit the prediction with a projection test or ignore it with a chi-squared test. However, both tests must be used since the mean of the data distribution might be orthogonal to the prediction. Thus, the only advantage of the oracle test is knowing exactly when to run each test, which yields at most a constant-factor improvement in efficiency over non-oracle tests.

The situation changes when several predictions are available. Each additional prediction introduces a new direction along which deviations from the null hypothesis may occur. As a result, a test must determine not only whether to incorporate predictions, but also which ones to use to avoid diluting its power too much. An oracle, with knowledge of the accuracy of each prediction, can identify to what degree each prediction is informative and aggregate them optimally. A non-oracle test that lacks this information must hedge across all predicted directions and therefore cannot reliably achieve the same performance. As a result, oracle and non-oracle tests generally emphasize different directions in which to detect deviations from the null hypothesis, and this mismatch creates a non-constant performance gap.

Consider first the oracle setting, where under the alternative hypothesis we know the accuracies of the predictions:  \begin{equation}\label{eq:multiple_orthogonal_gaussian_seq_l2_testing}
H_0: \theta =0 \vs H_1: \norm{\theta}_2\geq \epsilon \textand \norm{\theta-\hat{\theta}_i}_2 \leq r_i \for i \in [m],
\end{equation}

Henceforth, let $B_r=\bigcap_{i=1}^{m}B_2^d(\hat{\theta}_{i},r_{i})$ denote the intersection of all the neighborhoods around the predictions. As commented in Section \ref{sec:overview_blackbox}, an oracle can detect when the predictions are uninformative, which geometrically corresponds to $S_\epsilon \subset B_r$, in which case hypotheses \eqref{eq:multiple_orthogonal_gaussian_seq_l2_testing} reduce to the well-studied hypotheses \eqref{eq:gsn_seq_expo_large_r}. In this case, the chi-squared test \eqref{eq:expo_chi_squared_test}, which discards all predictions, is optimal. That is, we reject whenever the $\ell_2$ norm of the data is larger than the expected deviations under the null hypothesis: \begin{equation}\label{eq:chi_squared_test}
\psi_{\chi^2}^{\alpha}(X)=I\left( n\cdot \norm{X}_2^2 \geq \chi^2_{d,1-\alpha} \right),
\end{equation} where $\chi^2_{d,1-\alpha}$ is the $(1-\alpha)$-quantile of the chi-squared distribution with $d$ degrees of freedom. 

Conversely, consider the case in which the predictions are highly informative; geometrically, this corresponds to the alternative hypothesis containing a single point closest to the null hypothesis. This happens whenever $S_\epsilon \cap B_r$ is empty or a singleton. In this case, using a likelihood-ratio test between $H_0: \theta=0$ and $H_1:\theta=\theta_*(r)$, see \eqref{eq:theta_star}, is optimal. The likelihood ratio test is equivalent to rejecting the null hypothesis whenever the projection of the data onto the direction of $\theta_*(r)$ is large enough: \begin{equation}\label{eq:likelihood_ratio_test}
\psi^{\alpha}_{\LR}(X) = I\left(\sqrt{n}\cdot \inner{X}{\frac{\theta_*(r)}{\norm{\theta_*(r)}_2}} \geq z_{1-\alpha}\right). 
\end{equation}

In general, using projections onto the predicted directions is a sensible choice, since under the alternative hypothesis, each predicted direction can serve as a means of detecting deviations from the null hypothesis.
\begin{remark}\label{remark:A_plus_definition} Let $h_i=\norm{\hat{\theta}_i}_2$ be the length of the $i$-th prediction, and $\hat{u}_i=\hat{\theta}_i/h_i$ its direction. Define $A_i(\epsilon,r)$ = $(\epsilon^2+h_i^2-r_i^2)/(2h_i)$. Then, under the alternative hypothesis \eqref{eq:multiple_orthogonal_gaussian_seq_l2_testing}, the length of the projection onto the $i$-th prediction is lower bounded:
$\norm{\Pi_{\hat{u}_i}\theta}_2^2  \geq A^2_{i,+}(\epsilon,r)$ where $A_{i,+}(\epsilon,r)=\max\left(0,A_{i}(\epsilon,r)\right)$ for $i \in [m]$.\end{remark} Given the geometric restrictions imposed by each prediction, it is not clear if the lower bounds in Remark \ref{remark:A_plus_definition} can be achieved. Henceforth, let $\gamma_*$ be the length of the smallest projection onto the span of the predictions under the alternative hypothesis \begin{equation}\label{eq:projection_prediction_span}
\gamma_*^2(\epsilon,r) = \inf_{\theta \in B_r\cap S_\epsilon} \norm{\Pi_{\hat{U}}\theta}_2^2\quad  \where \hat{U} = \left[\hat{u}_1,\dots,\hat{u}_m\right],
\end{equation} and let $\Rank \hat{U}=k$. If  $B_r\cap S_\epsilon\neq\emptyset$ and $k < d$, it holds that there exists a unique vector $A_*(\epsilon,r) \in \R^k$ that characterizes the projection of $\theta$ onto the span of the predictions, i.e., $\gamma_*^2(\epsilon,r)=\norm{A_*(\epsilon,r)}_2^2$. In particular, if all predictions are orthogonal, it holds that $A_*(\epsilon,r)=A_+(\epsilon,r)$, meaning that the lower bounds in Remark \ref{remark:A_plus_definition} are achieved. 

Furthermore, if $B_r\cap S_\epsilon\neq\emptyset$, hypotheses \eqref{eq:multiple_orthogonal_gaussian_seq_l2_testing} can be reduced to testing the hypotheses: \begin{equation}\label{eq:reduced_alternative}
H_0: \theta = 0 \vs H_1: \theta \in \arginf_{\theta \in S_\epsilon\cap B_r} \Corr\left(\Pi_{\hat{U}}\theta,\theta\right).
\end{equation} Therefore, to detect the alternative hypothesis, the oracle test only needs to measure how aligned $\Pi_{\hat{U}}X$ is with any element on the alternative hypothesis \eqref{eq:reduced_alternative}. This choice is immaterial because their projection onto the span of the predictions is fixed and given by $A_*(\epsilon,r)$. Thus, the natural test statistic for the oracle test is proportional to the plug-in estimator of the worst-case cosine similarity:  \begin{equation}\label{eq:correlation_test}
\psi_{\Corr}^{\alpha}(X) = I\left(\sqrt{n}\cdot \inner{\Pi_{\hat{U}}X}{\frac{\tilde{\theta}}{\norm{\tilde{\theta}}_2}}\geq z_{1-\alpha}\right) \where \tilde{\theta}=\tilde{U}_k A_*(\epsilon,r).
\end{equation} where $\tilde{U}_k=\left[\tilde{u}_1,\dots,\tilde{u}_k\right]$ are the $k$ left singular vectors of $\hat{U}$ associated with the non-zero singular values.

In summary, if $S_\epsilon$ and $B_r$ intersect, the oracle test runs the chi-squared test \eqref{eq:chi_squared_test} whenever the minimum projection onto the span of the predictions is too small, and the worst-case cosine similarity test \eqref{eq:correlation_test} otherwise. Alternatively, if $S_\epsilon$ and $B_r$ do not intersect, a likelihood ratio test \eqref{eq:likelihood_ratio_test} suffices since there is a unique point under the alternative that is closest to the null hypothesis. We define our oracle test as follows: \begin{equation}\label{eq:oracle_test}
\psi_{\text{Oracle}}(X) = \begin{cases}
\psi^{\alpha/2}_{\chi^2}(X) \lor \psi_{\Corr}^{\alpha/2}(X)&\textif S_\epsilon \cap B_r \neq\emptyset\\
\psi^{\alpha}_{\LR}(X) &\textif S_\epsilon \cap B_r=\emptyset
\end{cases}
\end{equation}

The following lemma summarizes the performance achieved by the oracle test. The proof is deferred to Appendix \ref{appx:ub_multiple_orthogonal_gaussian_seq_l2_testing} of the supplementary material.

\begin{lemma}[Performance of oracle test with multiple predictions]\label{lemma:upperbound_oracle_test_arbitrary_predictions} For hypotheses \eqref{eq:multiple_orthogonal_gaussian_seq_l2_testing}, the alternative hypothesis is non-vacuous if $B_r \not\subset B_\epsilon^\circ$. Furthermore, if $\Rank \hat{U} < d$, there exists a positive constant $C$ that depends only on $\alpha$ and $\beta$ such that the test \eqref{eq:oracle_test} is powerful whenever \begin{equation}
n \geq C \cdot \begin{dcases}
\frac{1}{\gamma_*^2(\epsilon,r)} \wedge \frac{\sqrt{d}}{\epsilon^2} &\textif S_\epsilon \cap B_r \neq \emptyset\\
%&\textif \gamma_*(\epsilon,r) \leq \epsilon\\
\frac{1}{\norm{\theta_*(r)}_2^2} &\textif S_\epsilon \cap B_r = \emptyset %\textif \gamma_*(\epsilon,r) > \epsilon
\end{dcases}\end{equation} 
We note that the rates are continuous since $\norm{\theta_*(r)}_2^2=\inf_{\theta \in B_r}\norm{\theta}_2^2=\inf_{\theta \in B_r}\norm{\Pi_{\hat{U}}\theta}_2^2$ if $S_\epsilon$ and $B_r$ do not intersect.
\end{lemma} 

The first case shows that the oracle test achieves the prediction-free rate when predictions are uninformative, which has full dependence on the ambient dimension of the data. Conversely, when predictions are informative, the oracle test avoids any dependence on the ambient dimension of the data ($d$) or the dimension of the span of the predictions ($\Rank \hat{U}$). This improvement should come as no surprise, given that the oracle test relies on a one-dimensional projection that aggregates all predictions \eqref{eq:correlation_test}. The second case shows that when there is a single closest distribution under the alternative hypothesis, the oracle test achieves parametric rates.

\subsection{Optimality of the oracle test}

Let us first recall what is known regarding distinguishing the hypotheses \eqref{eq:gsn_seq_expo_large_r}, for which the chi-squared test \eqref{eq:chi_squared_test} is optimal \citep{ingsterAsymptoticallyMinimaxHypothesisI1993}. An optimality proof follows from proving that no test can distinguish the null distribution $P_0$ from those distributions in the alternative hypotheses that are closest. Given that any point on the $S_\epsilon$ sphere is equally challenging to distinguish, we put a uniform measure on the sphere.
\begin{equation}\label{eq:classic_nonparametric_lower_bound}
H_0: \theta=0 \vs H_1: \theta \sim \text{Uniform}\left[S_\epsilon\right].
\end{equation}

We can guarantee that the above construction is valid under hypotheses \eqref{eq:multiple_orthogonal_gaussian_seq_l2_testing} whenever $S_\epsilon \subseteq B_r$, see Figure \ref{fig:uninformative}. Thus, for big enough $B_r$, the lower bound from \citet{ingsterAsymptoticallyMinimaxHypothesisI1993} must hold.

\begin{remark}\label{remark:nonparametric} For hypotheses \eqref{eq:multiple_orthogonal_gaussian_seq_l2_testing}, there exists a positive constant $C$ that depends only on $\alpha$ and $\beta$, such that no valid test can be powerful whenever \begin{equation}
n \leq C\cdot \frac{\sqrt{d}}{\epsilon^2} \textand S_\epsilon \subseteq B_r.
\end{equation}
\end{remark}

Conversely, let us consider the case where the predictions are highly informative. Geometrically, this corresponds to the intersection of $S_\epsilon$ and $B_r$ containing at most one point, implying that, under the alternative hypothesis, there is a unique point closest to the null hypothesis. Consequently, distinguishing the hypotheses \eqref{eq:multiple_orthogonal_gaussian_seq_l2_testing} is equivalent to distinguishing the simple null hypothesis $H_0: \theta =0$ from the simple alternative hypothesis $H_1:\theta=\theta_*(r)$, see \eqref{eq:theta_star}.

\begin{remark}\label{remark:parametric} For hypotheses \eqref{eq:multiple_orthogonal_gaussian_seq_l2_testing}, there exists a positive constant $C$ that depends only on $\alpha$ and $\beta$, such that no valid test can be powerful whenever \begin{equation}
n \leq C \cdot \frac{1}{\norm{\theta_*(r)}_2^2} \textand S_\epsilon \cap B_r \text{ is a singleton or the empty set.}
\end{equation}
\end{remark}

In summary, whenever the predictions are uninformative, we can use the unrestricted construction \eqref{eq:classic_nonparametric_lower_bound} since $S_\epsilon\subseteq  B_r$, and consequently, we obtain the prediction-free rate by constructing a uniform measure on $S_\epsilon$ under the alternative hypothesis. Conversely, when the prediction is highly informative, we can use a simple binary hypothesis-testing construction, akin to assigning a \textit{uniform measure} supported on a single point under the alternative.

For the general case, see Figure \ref{fig:intersection}, based on \eqref{eq:classic_nonparametric_lower_bound}, a natural idea is to put a uniform measure on the spherical cap $S_\epsilon \cap B_r$. That is, using the alternative hypothesis $H_1: \theta \sim \text{Uniform}[S_\epsilon \cap B_r]$, which interpolates the previously mentioned cases. Although that construction is valid, a sufficient approach is to focus only on those points of the spherical cap that minimize the length of the projection onto the span of the predictions \eqref{eq:projection_prediction_span}. Namely, we look to distinguish \begin{equation}\label{eq:adversarial_alternative}
H_0: \theta =0 \vs H_1: \theta \sim \text{Uniform}\left[\left\{\theta\in S_\epsilon \cap B_r: \norm{\Pi_{\hat{U}}\theta}_2^2 = \gamma_*^2(\epsilon,r)\right\}\right].
\end{equation} The distributions under the alternative hypothesis have both the property of being as close as possible to the null hypothesis, and minimizing the amount of information that predictions provide on them. Based on this construction, we prove the following limit on the performance of any test that attempts to distinguish hypotheses \eqref{eq:multiple_orthogonal_gaussian_seq_l2_testing}. The proof of the following lemma is deferred to Appendix \ref{appx:lb_multiple_orthogonal_gaussian_seq_l2_testing} of the supplementary material.

\begin{lemma}\label{lemma:lowerbound_oracle_test_arbitrary_predictions} For $\Rank \hat{U}\leq d/2$ and hypotheses \eqref{eq:multiple_orthogonal_gaussian_seq_l2_testing}, there exists a positive constant $C$ that depends only on $\alpha$ and $\beta$, such that no valid test can be powerful whenever \begin{equation}
n \leq C \cdot \left[\frac{1}{\gamma_*^{2}(\epsilon,r)} \wedge \frac{\sqrt{d}}{\epsilon^{2}}\right] \textand S_\epsilon \cap B_r \neq \emptyset.
\end{equation}
\end{lemma}

Together, Remarks \ref{remark:nonparametric} and \ref{remark:parametric}, and Lemma \ref{lemma:lowerbound_oracle_test_arbitrary_predictions} show that the performance of the oracle test in Lemma \ref{lemma:upperbound_oracle_test_arbitrary_predictions} is optimal up to constants. Furthermore, the remarks and lemmas \ref{lemma:upperbound_oracle_test_arbitrary_predictions} and \ref{lemma:lowerbound_oracle_test_arbitrary_predictions}  prove Theorems \ref{thm:oracle_rates_single_prediction} and \ref{thm:oracle_rates_orthogonal_predictions} in Section \ref{sec:overview_blackbox}.

\section{The cost of adapting to orthogonal predictions}\label{sec:orthogonal}

When predictions are orthogonal, each prediction contains some new information about the underlying data distribution. From the perspective of the adaptivity gap, this is a worst-case scenario. While the oracle test can optimally aggregate the predictions and avoid any dependence on the dimension of the span, non-oracle procedures might have to consider all possible predictions since each one of them potentially contains information about the data distribution not available in the others.

As discussed in Section \ref{sec:minimax}, we model the need to adapt to the unknown accuracies by requiring a test to distinguish between the hypotheses \eqref{eq:multiple_orthogonal_gaussian_seq_l2_testing} for any accuracy vector. Formally, we seek a test that distinguishes \begin{equation}\label{eq:union_gaussian_seq_l2_testing}
H_0: \theta = 0 \vs H_1: \exists r \in \R^{m}_+ \st \norm{\theta}_2\geq \epsilon(r) \textand \norm{\theta-\hat{\theta}_i}_2 \leq r_i \for i \in [m],
\end{equation} where the predictions are assumed to be orthogonal. Note that the separation curve $r\mapsto \epsilon(r)$ also depends on the sample size $n$. However, since $n$ remains fixed throughout this work, we omit this dependence from the notation.

We first approximate the projection statistic that drives the oracle test in \eqref{eq:correlation_test}. Under the alternative, by virtue of knowing the vector of accuracies, the oracle test can compute the worst cosine similarity \eqref{eq:correlation_test}. Without knowledge of the accuracy vector $r$, this information is not available. Nevertheless, we can still approximate the cosine similarity between $\Pi_{\hat{U}}\theta$ and $\theta$ by measuring the length of the projection of the data onto the span of the predictions:  \begin{equation}\label{eq:projection_test}
\psi^{\alpha}_{\proj}(X) = I\left(n \cdot \norm{\Pi_{\hat{U}}X}_2^2 \geq \chi^2_{\Rank \hat{U},1-\alpha}\right).
\end{equation} Analogously to the oracle test, we must also protect ourselves against the case where the predictions are uninformative. However, rather than running the chi-squared test \eqref{eq:chi_squared_test}, since $\norm{X}_2^2 = \norm{\Pi_{\hat{U}}X}_2^2 + \norm{\Pi_{\hat{U}^\perp}X}_2^2$, we can avoid reusing the data by simply rejecting the null hypothesis whenever the length of the projection onto the orthogonal complement of the predictions' span is large: \begin{equation}\label{eq:orthogonal_projection_test}
\psi^{\alpha}_{\proj^\perp}(X) = I\left(n \cdot \norm{\Pi_{\hat{U}^\perp}X}_2^2 \geq \chi^2_{d-\Rank \hat{U},1-\alpha}\right).
\end{equation} Thus, we might replace the oracle test \eqref{eq:oracle_test} with the following Bonferroni aggregation of \eqref{eq:projection_test} and \eqref{eq:orthogonal_projection_test}: \begin{equation}\label{eq:bonferroni_test_with_standard_projection}
\psi^{\alpha}(X) = \max\left(\psi^{\alpha/2}_{\proj}(X),\psi^{\alpha/2}_{\proj^\perp}(X)\right).
\end{equation}

The next lemma provides a sharp characterization of the performance of this test. Crucially, the required sample size depends on the dimension $m$ of the span of the predictions. Compared with the oracle bound in Lemma \ref{lemma:upperbound_oracle_test_arbitrary_predictions}, the required sample size increases by a factor of $\sqrt{m}$. This is expected since \eqref{eq:bonferroni_test_with_standard_projection} relies on a chi-squared test on the span of the predictions, which looks for deviations from the null hypothesis in all directions of the predictions' span. The proof is deferred to Appendix \ref{appx:ub_union_multiple_orthogonal_gaussian_seq_l2_testing} of the supplementary material.

\begin{lemma}[The cost of adaptation to orthogonal predictions]\label{lemma:upperbound_adaptive_test_orthogonal_predictions} For hypotheses \eqref{eq:union_gaussian_seq_l2_testing}, if the predictions are orthogonal, there exists a positive constant $C$ that depends only on $\alpha$ and $\beta$ such that the test \eqref{eq:bonferroni_test_with_standard_projection} is valid and powerful for the separation curve \begin{equation}\label{eq:epsilon_star_orthogonal_predictions}
\epsilon_*(r) = \inf\{\epsilon\geq 0 : n \geq C \cdot g(\epsilon,r)\}
\end{equation} where for  $m\leq d/2$, \begin{equation}
g(\epsilon,r)= \begin{dcases}
\frac{\sqrt{m}}{\gamma_*^{2}(\epsilon,r)} \wedge \frac{\sqrt{d}}{\epsilon^2} &\textif S_\epsilon \cap B_r \neq \emptyset \\%\gamma_*(\epsilon,r) \leq \epsilon\\
\frac{\sqrt{m}}{\norm{\theta_*(r)}_2^2} &\textif S_\epsilon \cap B_r = \emptyset %\gamma_*(\epsilon,r) > \epsilon
\end{dcases}
\end{equation} and for $m > d/2$,
$g(\epsilon,r)= \sqrt{d} \cdot \epsilon^{-2}$.
\end{lemma}

\subsection{Characterization of optimal adaptation}

In this section, we show that Lemma \ref{lemma:upperbound_adaptive_test_orthogonal_predictions} is optimal in the relevant adaptive sense. The separation curve \eqref{eq:epsilon_star_orthogonal_predictions}, achieved by the adaptive test \eqref{eq:bonferroni_test_with_standard_projection}, is pointwise improvable. This is typical in adaptive testing problems \citep{gineMathematicalFoundationsInfinitedimensional2016}. For example, fix an accuracy vector $r_* \in \R_+^m$. A Bonferroni correction allows us to combine the adaptive test \eqref{eq:bonferroni_test_with_standard_projection} with the oracle test \eqref{eq:oracle_test} designed for the alternative
$
H_1: \norm{\theta}_2 \geq \epsilon_*(r_*) \textand \theta \in B_{r_*}.
$
This combined test has greater power at $r_*$. Such a gain, however, can occur only at a fixed number of accuracy vectors, and therefore does not contradict adaptive optimality.

We therefore prove optimality within a class of separation curves. This class contains curves that assign the same separation to accuracy vectors that yield the same minimal projection onto the span of the predictions. In other words, if two accuracy vectors lead to worst-case projections of equal length onto the predictions' span, then their induced alternative hypotheses are considered equally hard to distinguish from the null hypothesis. \begin{property}[Same separation for equivalent accuracy vectors]\label{property:separation_curve_orthogonal_predictions} Let $\tilde{\epsilon}$ be a separation curve. If $\tilde{\epsilon}(r) = \epsilon$ for  $r\in \R_+^m$, then  $\tilde{\epsilon}(r') = \epsilon$ for all $r' \in \R_+^m$ such that $\gamma_*(\epsilon,r)=\gamma_*(\epsilon,r')$.\end{property} 

Note that the property is natural, given that both the separation curves associated with the oracle test in Lemma \ref{lemma:upperbound_oracle_test_arbitrary_predictions} and the adaptive test in Lemma \ref{lemma:upperbound_adaptive_test_orthogonal_predictions} satisfy it. The next lemma shows that, among all separation curves satisfying Property \ref{property:separation_curve_orthogonal_predictions}, the separation curve in Lemma \ref{lemma:upperbound_adaptive_test_orthogonal_predictions} cannot be meaningfully improved.

\begin{lemma}[Optimality of the adaptive test \eqref{eq:bonferroni_test_with_standard_projection}]\label{lemma:lowerbound_adaptive_test_orthogonal_predictions} For $m \leq d/2$ and orthogonal predictions, consider hypotheses \eqref{eq:union_gaussian_seq_l2_testing} with a separation curve $\tilde{\epsilon}$ that has Property \ref{property:separation_curve_orthogonal_predictions}. If there exists $r_* \in \R^m_+$ such that $\tilde{\epsilon}(r_*) < C \cdot \epsilon_*(r_*)$ where \begin{equation}\label{eq:epsilon_star_orthogonal_predictions_lowerbound}
\epsilon_*(r) = \inf\left\{\epsilon\geq 0: n \geq C \cdot\left[\frac{\sqrt{m}}{\gamma_*^{2}(\epsilon,r)} \wedge \frac{\sqrt{d}}{\epsilon^{2}}\right] \textand   S_\epsilon \cap B_r \neq \emptyset \right\},
\end{equation} then no valid test can be powerful.
\end{lemma}

The proof is deferred to Appendix \ref{appx:lb_union_multiple_orthogonal_gaussian_seq_l2_testing} of the supplementary material. Furthermore, note that Lemma \ref{lemma:lowerbound_adaptive_test_orthogonal_predictions} proves that Lemma \ref{lemma:upperbound_adaptive_test_orthogonal_predictions} is sharp up to constants for the $S_\epsilon \cap B_r\neq \emptyset$ case. Together Lemmas \ref{lemma:lowerbound_adaptive_test_orthogonal_predictions} and \ref{lemma:upperbound_adaptive_test_orthogonal_predictions} prove Theorems \ref{thm:upperbound_adaptive_test_orthogonal_predictions} and \ref{thm:lowerbound_adaptive_test_orthogonal_predictions} in Section \ref{sec:overview_blackbox}.

Lemma \ref{lemma:lowerbound_adaptive_test_orthogonal_predictions} strengthens the usual adaptivity guarantee, which only establishes optimality of the separation curve $\epsilon_*$ against any \textit{fully dominated} separation curve \citep{gineMathematicalFoundationsInfinitedimensional2016}. Such guarantees rule out any valid test with nontrivial power against a candidate curve $\tilde{\epsilon}$ satisfying $\tilde{\epsilon}(r) < C\cdot \epsilon_*(r)$ for every $r\in \R_+^m$. Our lower bound reaches the same conclusion under a weaker requirement: $\epsilon_*$ only needs to dominate $\tilde{\epsilon}$ on a particular subset of accuracy vectors.

To see why the dependence on the dimension of the prediction span is sharp, consider a weak oracle that, under the alternative hypothesis, knows the norm of the projection onto the prediction span but not its direction. This oracle must distinguish \begin{equation}\label{eq:detection_on_predictions_span_with_unknown_direction}
H_0: \hat{U}^T\theta = 0 \vs H_1: \exists y \in S_2^{m-1} \st  \norm{\hat{U}^T\theta}_2 = \gamma \textand \hat{U}^T\theta = \gamma \cdot y
\end{equation} where $\gamma$ is known. An adversary can make this problem difficult by randomizing the direction of $\hat{U}^T\theta$. In particular, under the alternative, the adversary may choose $
H_1: \hat{U}^T\theta/\gamma \sim \text{Uniform}(S_2^{m-1})$. This distribution spreads the projected signal uniformly across all directions in the prediction span. A useful test must therefore search over the entire unit sphere in $\R^m$ rather than inspect a known direction. This mirrors the lower bound in \eqref{eq:classic_nonparametric_lower_bound}, where randomizing the alternative forces the test to look for deviations in all directions. Consequently, the oracle can distinguish the hypotheses only when $n \geq \sqrt{m}\cdot\gamma^{-2}$. Thus, no test for \eqref{eq:detection_on_predictions_span_with_unknown_direction} can remove the $\sqrt{m}$ dependence.

A related construction yields a lower bound for the original problem. Under the hypotheses \eqref{eq:union_gaussian_seq_l2_testing}, an adversary can draw random points under the alternative hypothesis satisfying \begin{equation}\label{eq:worst_loose_alternative}
H_1: \norm{\theta}_2 = \epsilon \textand \frac{\hat{U}^T\theta}{\norm{\hat{U}^T\theta}_2} \sim \text{Uniform}(S_2^{m-1}).
\end{equation} This distribution captures the correct $\sqrt{m}$ dependence but does not correctly encode the role of $\gamma_*(\epsilon,r)=\norm{A_+(\epsilon,r)}_2$. The issue arises because $\hat{U}^T\theta$ may contain negative coordinates, which increase $\norm{\hat{U}^T\theta}_2$ relative to the optimal norm $\norm{(\hat{U}^T\theta)_+}_2$. To obtain a tight lower bound, we must restrict the distribution to be supported on the positive orthant. However, replacing $S_2^{m-1}$ by $S_2^{m-1} \cap \{\theta \in \R^m : \theta \geq 0\}$ in \eqref{eq:worst_loose_alternative}  does not work because the uniform measure on the positive orthant remains too diffuse. That is, the concentration of one-dimensional projections of said random vectors depends on the dimension $m$. Instead, we must support the distribution on a smaller structured subset of the positive orthant. Let $\mathcal{S}$ denote the collection of all subsets of size $\sqrt{m}$ from $[m]$. Under the alternative hypothesis, the adversary chooses \begin{equation}\label{eq:worst_alternative}
H_1:\norm{\theta}_2 = \epsilon \textand  \frac{\hat{U}^T\theta}{\norm{\hat{U}^T\theta}_2} = \mathbb{1}_S \cdot m^{-1/4} \where  S \sim \text{Uniform}(\mathcal{S})
\end{equation} and $[\mathbb{1}_S]_i=I(i \in S)$. Lemma \ref{lemma:lowerbound_adaptive_test_orthogonal_predictions} utilizes this random alternative to prove that the rates attained by test \eqref{eq:bonferroni_test_with_standard_projection} are optimal for the $S_\epsilon \cap B_r \neq \emptyset$ case.

\section{The cost of adapting to arbitrary predictions}\label{sec:arbitrary}

In the previous section, we considered only orthogonal predictions. In that setting, the test's performance depends on the ambient dimension of the span of predictions, because each prediction provides information about the alternative that the others cannot represent. This becomes overly pessimistic when predictions are not orthogonal, since many predictions may cluster tightly, especially when they are obtained by retraining the same predictive method on different subsamples of the available data. 

Henceforth, we consider the realistic scenario in which the accuracy of the predictions is unknown, and the predictions are not assumed to be orthogonal. That is, we distinguish hypotheses \eqref{eq:union_gaussian_seq_l2_testing} where the predictions span a $k$-dimensional subspace. That is, $\Rank \hat{U}= k \leq \min(m,d)$.

To understand how non-orthogonality changes the geometry of the problem, consider using the projection test \eqref{eq:bonferroni_test_with_standard_projection} we employed for the case of orthogonal predictions. The following lemma shows that the performance of the test depends on the ambient dimension of the predictions' span, even if all predictions are highly aligned. The proof is deferred to Appendix \ref{appx:ub_union_multiple_gaussian_seq_l2_testing} of the supplementary material.

\begin{lemma}[Performance of the adaptive test that uses standard projection]\label{lemma:performance_adaptive_test_with_standard_projection} For hypotheses \eqref{eq:union_gaussian_seq_l2_testing}, there exists a positive constant $C$ that depends only on $\alpha$ and $\beta$ such that the adaptive test using standard projection \eqref{eq:bonferroni_test_with_standard_projection} is powerful for the separation curve \begin{equation}\label{eq:epsilon_star_arbitary_predictions_with_projection_test}
\epsilon_*(r) = \inf\{\epsilon\geq 0 : n \geq C \cdot g(\epsilon,r)  \} \end{equation} where \begin{equation}
g(\epsilon,r)=  \begin{dcases}
\frac{\sqrt{k}}{\gamma_*^2(\epsilon,r)} \wedge \frac{\sqrt{d}}{\epsilon^{2}} &\textif S_\epsilon \cap B_r \neq \emptyset\\ %\gamma_*(\epsilon,r) \leq \epsilon \\
\frac{\sqrt{k}}{\norm{\theta_*(r)}_2^{2}} &\textif S_\epsilon \cap B_r = \emptyset %\gamma_*(\epsilon,r) > \epsilon
\end{dcases}
\end{equation} %and $\norm{\theta_*(r)}_2^2=\inf\{\epsilon^2\geq 0: \epsilon^2\geq \gamma_*^2(\epsilon,r)\}$.
\end{lemma}

We could have foreseen these results by noting that the test discards useful information by considering all the dimensions of the span of the predictions equally important, since $
\norm{\Pi_{\hat{U}}X}_2^2 = \sum_{i=1}^k \norm{\Pi_{\tilde{u}_i}X}_2^2$ where $\tilde{u}_1,\dots,\tilde{u}_k$ are the left singular vectors of $\hat{U}$ associated with the non-zero singular values, and therefore it must pay for looking for deviations from the null on each of them.

A better approach to reduce the strong dependence on the dimension of predictions' span is to hedge our bets and rely mostly on those directions in the span of predictions that are aligned with most of the predictions. As mentioned in the introduction, this idea leads to the test statistic that we call \textit{intrinsic projection} of the data onto the span of the predictions \eqref{eq:intrinsinc_projection_span_of_predictions}: $
\norm{\Pi^{\intrinsic}_{\hat{U}}X}_2^2=\sum_{i=1}^k (\lambda_i^2/\lambda^2_1) \cdot \norm{\Pi_{\tilde{u}_i}X}_2^2$. When the predictions are orthogonal, intrinsic projection coincides with standard projection $\norm{\Pi^{\intrinsic}_{\hat{U}}X}_2^2=\norm{\Pi_{\hat{U}}X}_2^2$ since the cosine similarity matrix becomes the identity matrix. However, when this is not the case, the weighting automatically disregards directions that are not aligned with most of the predictions, thereby avoiding reliance on the span of the predictions' dimension. 

Under the null hypothesis, the test statistic follows a mixture of chi-squared distributions for which the appropriate quantile can be computed. Thus, we can use the following test on the intrinsic projection.  \begin{equation}\label{eq:intrinsic_projection_test}
\psi^{\alpha}_{\intrinsic}(X) = I\left(n\cdot \norm{\Pi^{\intrinsic}_{\hat{U}}X}_2^2 \geq q_{1-\alpha}\right) 
\end{equation} where $q_{1-\alpha}$ is the $(1-\alpha)$-quantile of $\sum_{i=1}^k (\lambda_i^2/\lambda^2_1) \cdot \chi^2_1$. The term intrinsic reflects the fact that the performance of the test depends not on the dimension of the predictions' span itself, but on the intrinsic dimension of the squared cosine similarity matrix: \begin{equation}\label{eq:intrinsic_dimension}
\intrinsic \Sigma^2  = \frac{\trace \Sigma^2}{\norm{\Sigma^2}_{\op}} = \sum_{i=1}^k\frac{\lambda_i^4}{\lambda_1^4}
\end{equation} 

This quantity provides a stable estimate of the effective dimensionality of the span of the predictions, as it uses the full singular value distribution rather than only counting nonzero directions. If all predictions coincide, their span has both an intrinsic and an ambient dimension of $1$. A small perturbation can sharply increase the ambient dimension even when the predictions remain highly aligned. However, the intrinsic dimension of $\Sigma^2$ will remain close to $1$ since its eigenvalues will decay rapidly.

Analogously to the case of orthogonal predictions, we expect the intrinsic projection test \eqref{eq:intrinsic_projection_test} to have no power when the alternative lives in the orthogonal complement of the span of the predictions, which can be detected by using the chi-squared test \eqref{eq:chi_squared_test}. Thus, we consider the aggregation test \begin{equation}\label{eq:bonferroni_test_with_intrinsic_projection}
\psi^{\alpha}(X) = \max\left( \psi_{\intrinsic}^{\alpha/2}(X), \psi_{\chi^2}^{\alpha/2}(X) \right).
\end{equation} The following lemma summarizes the rates achieved by the test \eqref{eq:bonferroni_test_with_intrinsic_projection}. The proof is deferred to Appendix \ref{appx:ub_union_multiple_gaussian_seq_l2_testing} of the supplementary material. Furthermore, we note that Theorem \ref{thm:upperbound_adaptive_test_arbitrary_predictions} in Section \ref{sec:overview_blackbox} is a consequence of the following lemma.

\begin{lemma}[Performance of the adaptive test that uses intrinsic projection ]\label{lemma:upperbound_adaptive_test_arbitrary_predictions}For hypotheses \eqref{eq:union_gaussian_seq_l2_testing}, there exists a positive constant $C$ that depends only on $\alpha$ and $\beta$ such that the adaptive test using intrinsic projection  \eqref{eq:bonferroni_test_with_intrinsic_projection} is powerful for the separation curve \begin{equation}\label{eq:epsilon_star_arbitary_predictions}
\epsilon_*(r) = \inf\{\epsilon\geq 0 : n \geq C \cdot g(\epsilon,r)  \} \end{equation} where \begin{equation}
g(\epsilon,r)=  \begin{dcases}
\frac{\sqrt{\intrinsic \Sigma^2}}{\tilde{\gamma}_*^2(\epsilon,r)} \wedge \frac{\sqrt{d}}{\epsilon^{2}} &\textif S_\epsilon \cap B_r \neq \emptyset\\ %\gamma_*(\epsilon,r) \leq \epsilon \\
\frac{\sqrt{\intrinsic \Sigma^2}}{\norm{\tilde{\theta}_*(r)}_2^{2}} &\textif S_\epsilon \cap B_r = \emptyset %\gamma_*(\epsilon,r) > \epsilon
\end{dcases} 
\end{equation} and \begin{equation}
\tilde{\gamma}_*^2(\epsilon,r)=\inf_{\theta \in S_\epsilon \cap B_r}\norm{\Pi^{\intrinsic}_{\hat{U}}\theta}_2^2  \textand \norm{\tilde{\theta}_*(r)}_2^{2}=\inf_{\theta \in B_r}\norm{\Pi^{\intrinsic}_{\hat{U}}\theta}_2^2.
\end{equation}
\end{lemma} When the predictions are orthogonal, the intrinsic projection reduces to the standard projection. Consequently, Lemma \ref{lemma:upperbound_adaptive_test_arbitrary_predictions} reduces to Lemma \ref{lemma:upperbound_adaptive_test_orthogonal_predictions}, since when predictions are orthogonal it holds that $\tilde{\gamma}_*^2(\epsilon,r) = \gamma_*^2(\epsilon,r) = \norm{A_+(\epsilon,r)}_2^2$ and $\intrinsic \Sigma^2 = \Rank \hat{U} = k = m$.

Comparing the performance of intrinsic projection in Lemma \ref{lemma:upperbound_adaptive_test_arbitrary_predictions} to that of the standard projection in Lemma \ref{lemma:performance_adaptive_test_with_standard_projection}, the dependence on the dimension of the space is improved since it always holds that $\intrinsic \Sigma^2 \leq \Rank \hat{U} = k$. However, this comes at the cost of a smaller projection since $\norm{\Pi^{\intrinsic}_{\hat{U}}\theta}_2^2 \leq \norm{\Pi_{\hat{U}}\theta}_2^2$ for every $\theta \in \R^d$. 

\subsection{Dependence on the intrinsic dimension}

From an information standpoint, the projection test \eqref{eq:projection_test} discards valuable information about the relative importance of each orthogonal direction within the predictions' span. The intrinsic projection test \eqref{eq:intrinsic_projection_test} does not discard this information. Intuitively, we expect the test based on intrinsic projection to perform better than the test based on standard projection when the predictions are aligned, since their span will have a low intrinsic dimension.

To intuitively understand why we should expect its performance to depend on the intrinsic dimension of $\Sigma^2$, note that the intrinsic projection test \eqref{eq:intrinsic_projection_test} measures the norm of the random vector $(\Lambda_k/\lambda_1)\cdot\tilde{U}_k^TX$ whose mean is $(\Lambda_k/\lambda_1)\cdot\tilde{U}_k^T\theta$ and has covariance matrix $(\Lambda_k^2/\lambda_1^2)/n$. Based on the work of \citet{laurentNonAsymptoticMinimax2012}, we expect to be able to detect a deviation on the intrinsic projection of the data ($
H_0: \theta = 0 \vs H_1: \norm{\Pi^{\intrinsic}_{\hat{U}}\theta}_2 \geq \gamma$) as soon as $n \geq \sqrt{\trace \Lambda_k^4/\lambda_1^4}/\gamma^2$, where $\trace \Lambda_k^4/\lambda_1^4 = \intrinsic \Sigma^2$. The constructions used by \citet{laurentNonAsymptoticMinimax2012} do not directly apply to hypotheses \eqref{eq:union_gaussian_seq_l2_testing} but can be adapted to our problem. First, consider all the separation curves that satisfy the following property.

\begin{property}\label{property:separation_curve_arbitary_predictions} Let $\tilde{\epsilon}$ be a separation curve, and $\tilde{\gamma}(\epsilon,r)=\norm{A(\epsilon,r)}_2/\norm{\hat{U}}_{\op}$. If $\tilde{\epsilon}(r) = \epsilon$ for  $r\in \R_+^m$, then  $\tilde{\epsilon}(r') = \epsilon$ for all $r' \in \R_+^m$ such that $\tilde{\gamma}(\epsilon,r)=\tilde{\gamma}(\epsilon,r')$.
\end{property}

The following lemma provides some evidence for the potential dependence on the intrinsic dimension for adaptive tests. Compared to Lemma \ref{lemma:upperbound_adaptive_test_arbitrary_predictions}, the lower bound on the separation curve provided by Lemma \ref{lemma:lb_union_multiple_gaussian_seq_l2_testing} is loose since it depends on $\tilde{\gamma}$ rather than $\tilde{\gamma}_*$. Nevertheless, it conveys the message that the performance of adaptive tests should depend on the intrinsic dimension of $\Sigma^2$. The proof is deferred to Appendix \ref{appx:lb_union_multiple_gaussian_seq_l2_testing} of the supplementary material.

\begin{lemma}[Dependence of any adaptive test on the intrinsic dimension]\label{lemma:lb_union_multiple_gaussian_seq_l2_testing} For $\Rank \hat{U} \leq d/2$, consider hypotheses \eqref{eq:union_gaussian_seq_l2_testing} with a separation curve $\tilde{\epsilon}$ that satisfies Property \ref{property:separation_curve_arbitary_predictions}. If there exists $r \in \R_+^m$ such that $\tilde{\epsilon}(r) < C \cdot \tilde{\epsilon}_*(r)$ where \begin{equation}\label{eq:epsilon_star_arbitary_predictions_lowerbound}
\tilde{\epsilon}_*(r)= \inf\left\{\epsilon\geq 0 : n \geq C' \cdot \left(\frac{\sqrt{\intrinsic \Sigma^2}}{\tilde{\gamma}^{2}(\epsilon,r)} \wedge \frac{\sqrt{d}}{\epsilon^{2}}\right) \textand \epsilon \geq \sqrt{\frac{m}{\intrinsic \Sigma^2}}\cdot \tilde{\gamma}(\epsilon,r)\right\}, \end{equation} then no test can be valid and powerful.
\end{lemma}

\section{Simulations}\label{sec:simulations_blackbox}

In this section, we compare tests that differ in how they use the available predictions. The chi-squared test \eqref{eq:chi_squared_test} ignores predictions, whereas the projection test \eqref{eq:projection_test} and the intrinsic projection test \eqref{eq:intrinsic_projection_test} always use the predictions regardless of their quality. We also consider two adaptive tests based on standard projection \eqref{eq:bonferroni_test_with_standard_projection} and intrinsic projection \eqref{eq:bonferroni_test_with_intrinsic_projection}.

A trial consists of sampling a set of $m$ predictions $\{\hat{\theta}_1,\dots,\hat{\theta}_m\}$ that span an $m$-dimensional subspace and whose pairwise cosine similarities are all equal to $\rho\in[0,1)$, obtaining an observation $X$ from an adversarial distribution under the alternative hypothesis, providing the observation and the predictions to a test, and recording whether the test rejects the null hypothesis.

Henceforth, let $n\in\mathbb{N}$ be the sample size, $d\in \mathbb{N}$ be the ambient dimension of the data, $m\in [d-1]$ be the number of predictions, $\epsilon\in\R_+$ be the required separation, $r=\bar{r}\cdot\mathbb{1}_m\in\R_+^m$ be the vector of accuracies, so that all predictions have the same accuracy, and $h=\bar{h}\cdot\mathbb{1}_m\in\R_+^m$ be the vector of norms, so that all predictions have the same norm. We assume that $r$ has been chosen so that $S_\epsilon\cap B_r\neq\emptyset$ is satisfied.

We first sample $m$ predicted directions that span an $m$-dimensional subspace and whose pairwise cosine similarities are all equal to $\rho$. Let $\Sigma_\rho=(1-\rho)\cdot I_m+\rho\cdot\mathbb{1}_m\mathbb{1}_m^T$ be the target cosine-similarity matrix, and let $\Sigma_\rho=LL^T$ be its Cholesky decomposition. We sample $G\in\R^{d\times m}$ with $G_{i,j}\iid\Normal(0,1)$, let $G=QR$ be its reduced QR decomposition, and set $\hat{U}=\begin{bmatrix}\hat{u}_1,\dots,\hat{u}_m\end{bmatrix}=QL^T$. Since $Q^TQ=I_m$, it follows that $\hat{U}^T\hat{U}=\Sigma_\rho$. Therefore, $\norm{\hat{u}_i}_2=1$ for every $i\in[m]$ and $\Corr(\hat{u}_i,\hat{u}_j)=\rho$ for every distinct $i,j\in[m]$. Thus, the $m$ predicted directions $\{\hat{u}_1,\dots,\hat{u}_m\}$ are random vectors on the unit sphere with pairwise cosine similarity $\rho$. Furthermore, we let $\hat{\theta}_i=\bar{h}\cdot \hat{u}_i$ for $i \in [m]$.

Second, we sample an adversarially chosen $\theta$ based on the construction used to prove Lemma \ref{lemma:lowerbound_oracle_test_arbitrary_predictions}:
\begin{equation}\label{eq:simulated_theta}
\theta=\tilde{U}_m A_*(\epsilon,r)+\tilde{U}_m^\perp z\cdot\sqrt{\epsilon^2-\norm{A_*(\epsilon,r)}_2^2}.
\end{equation}
Here, the columns of $\tilde{U}_m$ form an orthonormal basis for the span of $\{\hat{u}_1,\dots,\hat{u}_m\}$, the columns of $\tilde{U}_m^\perp$ form an orthonormal basis for its orthogonal complement, and $z$ is sampled from $\text{Uniform}(S_2^{d-m-1})$. Since $m<d$ and $S_\epsilon\cap B_r\neq\emptyset$, $A_*(\epsilon,r)$ is the unique $m$-dimensional vector satisfying $\norm{A_*(\epsilon,r)}_2^2=\inf_{\theta\in B_r\cap S_\epsilon}\norm{\Pi_{\hat{U}}\theta}_2^2$. Because the predictions $\{\hat{\theta}_1,\dots,\hat{\theta}_m\}$ have the same accuracy, norm, and pairwise cosine similarity, it holds that $\norm{A_*(\epsilon,r)}_2=\gamma_*(\epsilon,r)=\tilde{\gamma}_*(\epsilon,r)$. Thus, every $\theta$ sampled from \eqref{eq:simulated_theta} is adversarial to tests that rely on standard or intrinsic projection. Finally, we sample $X|\theta\sim\Normal(\theta, I_d/n)$ and provide the observation $X$ and the predictions to the test under evaluation.

Throughout the simulations, we set $\alpha=0.05$, $n=100$, and $\epsilon=\bar{h}=0.3$, and vary $\bar{r}$ from $r_{\min}$ to $r_{\max}$. We choose $r_{\min}$ so that the assumption $S_\epsilon\cap B_r\neq\emptyset$ holds; that is, $r_{\min}=\min\{\tilde{r}\in\R_+:S_\epsilon\cap B_{\tilde{r}\cdot\mathbb{1}_m}\neq\emptyset\}=\epsilon\cdot[2-2\sqrt{\left(1+(m-1)\rho\right)/m}]^{1/2}$. Furthermore, we choose $r_{\max}=\epsilon\cdot\sqrt{2}$ so that $\norm{A_*(\epsilon,r_{\max}\cdot \mathbb{1}_m)}_2=0$, which implies that any projection onto the span of the predictions contains no useful information. For every configuration and test, we estimate power using $1{,}000$ independent trials.

% Therefore, $A_+(\epsilon,r)=A(\epsilon,r)=\mathbb{1}_m\cdot(2\epsilon^2-\bar{r}^2)/(2\epsilon)\geq0$.

\noindent\textbf{Consider the case of a single prediction.} We set $m=1$ and consider ambient dimensions $d\in\{2,10,100\}$. When $m=1$, the preceding construction samples one predicted direction $\hat{u}_1$ uniformly from the unit sphere. Figure \ref{fig:single_prediction} shows the results. Note that the chi-squared test suffers from the curse of dimensionality, whereas the projection test \eqref{eq:projection_test} does not, provided that the prediction is of high quality. However, for a low-quality prediction, the projection test performs worse than the chi-squared test. This is expected, as the projection test alone is not adaptive. A good trade-off between the projection and chi-squared tests is obtained by the adaptive test \eqref{eq:bonferroni_test_with_standard_projection}, which closely tracks the power of the projection test when the prediction is good but remains close to the power of the chi-squared test when the prediction is uninformative. This simulation illustrates that the adaptive test \eqref{eq:bonferroni_test_with_standard_projection} achieves a rate-wise free lunch, as described in Section \ref{sec:overview_blackbox}.

\begin{figure}[!b]
\centering
\includegraphics[width=\linewidth]{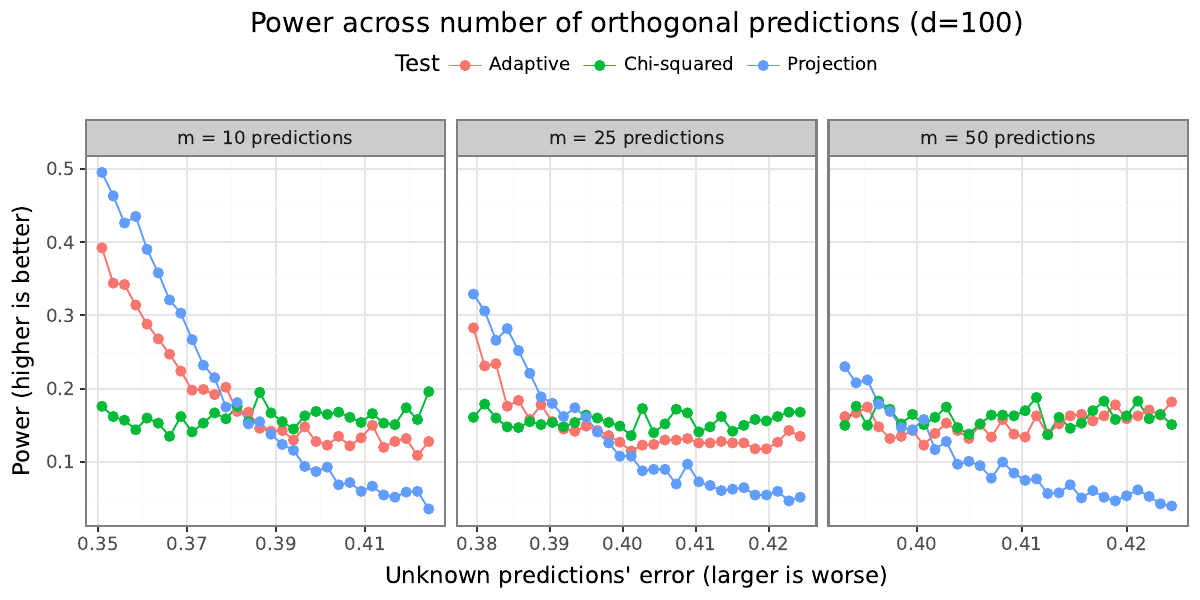}
\caption{Comparison between the chi-squared test \eqref{eq:chi_squared_test}, which ignores predictions, the projection test \eqref{eq:projection_test}, which always uses all predictions, and the adaptive test \eqref{eq:bonferroni_test_with_standard_projection}, which optimally adapts to the predictions' errors.}
\label{fig:orthogonal_predictions}
\end{figure}

\noindent\textbf{For orthogonal predictions}, we set $d=100$ and sample $\{\hat{u}_1,\dots,\hat{u}_m\}$ for $m\in\{10,25,50\}$ from the unit sphere such that they are orthogonal $(\rho=0)$. Figure \ref{fig:orthogonal_predictions} shows the results. Note that we progressively increase the number of orthogonal predicted directions until the number of predictions equals half of the ambient dimension of the data, since Lemma \ref{lemma:upperbound_adaptive_test_orthogonal_predictions} indicates that the adaptive test offers no rate-wise advantage over the chi-squared test at that point. Since the predictions are orthogonal, they represent the worst case from a geometric perspective, as this forces both the projection test \eqref{eq:projection_test} and the adaptive test \eqref{eq:bonferroni_test_with_standard_projection} to look for deviations from the null hypothesis in all predicted directions. Nevertheless, when the predictions are of high quality, prediction-based tests can still outperform the chi-squared test that ignores them.

\noindent\textbf{For aligned predictions}, we set $d=100$ and sample $\{\hat{u}_1,\dots,\hat{u}_m\}$ for $m=50$ from the unit sphere such that they span an $m$-dimensional subspace and all pairwise cosine similarities are equal to $\rho$. In this setting, it holds that \begin{equation}
\intrinsic \Sigma_\rho^2 = 1 + (m-1)\cdot \left(\frac{1-\rho}{1+(m-1)\rho}\right)^2.
\end{equation} Thus, by controlling their cosine similarity, we can control the intrinsic dimension of $\Sigma^2_\rho$. For instance, if predictions are orthogonal ($\rho=0$), it follows that $\intrinsic \Sigma^2_\rho=m$, while if predictions are almost parallel ($\rho\approx 1$), we have $\intrinsic \Sigma^2_\rho \approx 1$. For this experiment, we choose values for $\rho$ so that  $\intrinsic \Sigma^2_\rho$ is approximately $1$, $m/5$, or $m$. Figure \ref{fig:correlated_predictions} shows the results. First, note that when the predictions are orthogonal, methods based on intrinsic and standard projection perform equally. However, when there is high cosine similarity among the predictions, tests based on intrinsic projection, such as \eqref{eq:intrinsic_projection_test} and \eqref{eq:bonferroni_test_with_intrinsic_projection}, benefit from this alignment and outperform those based on standard projection. Furthermore, the adaptive intrinsic-projection test \eqref{eq:bonferroni_test_with_intrinsic_projection} appropriately relies on or ignores the predictions, providing a good trade-off between the intrinsic projection test \eqref{eq:intrinsic_projection_test} and the chi-squared test \eqref{eq:chi_squared_test}.

\section{Discussion}\label{sec:discussion_blackbox}

In this work, we have sought to provide answers to the question of optimal data aggregation with predictions, such as those produced by machine learning models. Although our results concern the Gaussian sequence model, the fundamental ideas are geometric in nature and therefore provide a way to generalize to other models. That is, insofar as a space is equipped with an inner product, both standard and intrinsic projections can be realized. 

At a high level, it is worth recalling that the problems of estimating and testing with access to black-box predictions are substantially different. As expected, when testing with black-box predictions, we are achieving faster rates than comparable estimation problems with similar guarantees. More importantly, testing requires only information about the \textit{direction} of a prediction. As a result, a prediction may remain useful for testing even when it lies far from the target, provided that it aligns well with the target of inference. Thus, decision-making can rely on predictions with weaker guarantees than those needed for estimation.

This observation suggests a practical criterion for training predictive models when their outputs feed into downstream decisions. In that setting, high cosine similarity with the target should matter more than distance to the target itself. This shift in evaluation could also yield computational gains, since in many modern machine learning pipelines, the cost of generating a prediction often increases with its quality. Therefore, decision procedures that use weaker yet well-aligned predictions may reduce computational cost without sacrificing operational performance.

Finally, other works have approached the problem of decision-making with black-box predictions by constructing confidence sets and then making decisions based on them. Although equivalent to the testing framework explored in this work, due to the duality between confidence sets and testing procedures, the width of said confidence sets is not fully indicative of the performance of decision-making based on them. This is because their width is fundamentally dominated by the rates of functional estimation, which typically differ from those of testing. 

Confidence sets remain highly useful for human-in-the-loop decision-making because they communicate information about the uncertainty involved in the decision-making process. However, as we move into automatic decision-making applications, the testing perspective becomes increasingly relevant, as it is solely concerned with the operational performance of large volumes of decisions without the need to interpret any single decision. Thus, we expect that the testing perspective will be a worthwhile framework for advancing research on large-scale automated decision-making with black-box predictions. 

\noindent\textbf{Acknowledgments} The authors thank Alexandra Carpentier, Ryan Tibshirani, Arun Kumar Kuchibhotla, Mikael Kuusela, Chad Schafer, Tudor Manole, and Richard Samworth for useful comments and discussions during the development of this work.

% \pagebreak

\begin{center}

{\large\bf SUPPLEMENTARY MATERIAL}

\end{center}

\appendix
\addcontentsline{toc}{section}{Supplementary material}
\etocdepthtag.toc{mtappendix}
\etocsettagdepth{mtchapter}{none}
\etocsettagdepth{mtappendix}{subsubsection}
\etocsettagdepth{mtreferences}{section}
\renewcommand{\contentsname}{Table of contents}
{\hypersetup{linkcolor=black}
\tableofcontents}

\renewcommand{\thesection}{\Alph{section}}
\numberwithin{equation}{section}
\numberwithin{theorem}{section}
\numberwithin{corollary}{section}
\numberwithin{proposition}{section}
\numberwithin{lemma}{section}

% Unique internal hyperlink destinations
\renewcommand{\theHsection}{supp.\Alph{section}}
\renewcommand{\theHequation}{supp.\Alph{section}.\arabic{equation}}

\pagebreak

\section{General upper bound on the performance of tests}

The following proposition provides a means to determine when a test controls both type-I and type-II errors uniformly over the null and alternative hypotheses, respectively. In particular, \eqref{eq:moment_condition} is good enough to capture, up to constants that depend on type-I and type-II errors, the minimum distance between the hypotheses required for a test to distinguish them.

\begin{proposition}\label{prop:chebyshev_bound} Let $\mathcal{P}_0$ and $\mathcal{P}_1(\epsilon)$ be sets of distributions, and given data $X\sim P$ consider the hypotheses \begin{equation}
H_0: P \in \mathcal{P}_0 \vs H_1: P \in \mathcal{P}_1(\epsilon).
\end{equation} Let $T$ be a test statistic, and  $q_{1-\alpha}(P,T)$ denote the $(1-\alpha)$-quantile of the test statistic under the distribution $P$ \begin{equation}
q_{1-\alpha}(P,T) = \inf\left\{u : P(T > u)\leq \alpha \right\}.
\end{equation} Define the test \begin{equation}
\psi(X) = I(T(X) > t) \where t = \sup_{P \in \mathcal{P}_0} q_{1-\alpha}(P,T),
\end{equation} then the test controls the type-I error uniformly over the null hypothesis: \begin{equation}
\sup_{P\in \mathcal{P}_0} P(\psi(X) = 1)\leq \alpha \period
\end{equation} Furthermore, if it holds that \begin{equation}\label{eq:power_condition_for_test_statistic}
t \leq \inf_{P \in \mathcal{P}_1} E_{P}[T(X)] - \sqrt{\frac{V_{P}[T(X)]}{\beta}},
\end{equation} the test controls the type-II error uniformly over the alternative hypothesis \begin{equation}
\sup_{P\in \mathcal{P}_1} P(\psi(X) = 0)\leq \beta \period
\end{equation} Consequently, the critical separation $\epsilon_*$ is upper bounded by $\epsilon$: \begin{equation}
\epsilon_* = \inf\{\epsilon \geq 0 : \inf_{\psi \in \Psi_{\alpha}\left(\mathcal{P}_0\right)}\sup_{P \in \mathcal{P}_1(\epsilon)}P(\psi(X)=0) \leq \beta\} \leq \epsilon.
\end{equation} Finally, since $q_{1-\alpha}(P,T)\leq E_P[T(X)] +  \sqrt{V_P[T(X)]/\alpha}$, it follows that \eqref{eq:power_condition_for_test_statistic} is implied by \begin{equation}\label{eq:moment_condition}
\sup_{P \in \mathcal{P}_0}E_{P}[T(X)] + \sqrt{\frac{V_{P}[T(X)]}{\alpha}} \leq \inf_{P \in \mathcal{P}_1}  E_{P}[T(X)] - \sqrt{\frac{V_{P}[T(X)]}{\beta}}\period
\end{equation}
\end{proposition}

\subsection{Bonferroni correction}\label{appx:Bonferroni}

\begin{restatable}{proposition}{Bonferroni}\label{prop:Bonferroni} Let $\mathcal{P}_0$ and $\mathcal{P}_1(\epsilon)$ be sets of distributions, and given data $X\sim P$ consider the hypotheses \begin{equation}
H_0: P \in \mathcal{P}_0 \vs H_1: P \in \mathcal{P}_1(\epsilon).
\end{equation} Let $\psi^{\alpha}_i$ for $i \in \{1,2\}$ be tests such that they control the type-I error uniformly by $\alpha$. Additionally, they control the type-II error uniformly by $\beta$ whenever $\epsilon \geq \epsilon_i$ respectively. Then, the test \begin{equation}
\psi^{\alpha}(X) = \psi^{\alpha \cdot C}_1(X) \lor \psi^{\alpha\cdot (1-C)}_2(X) \quad \for C \in (0,1)
\end{equation} controls the type-I error uniformly by $\alpha$. Furthermore, it controls the type-II error uniformly by $\beta$ whenever $\epsilon \geq \epsilon_1 \wedge \epsilon_2$.
\end{restatable} \begin{proof}
Under the null hypothesis, we have that \begin{align}
\sup_{P \in \mathcal{P}_0}E_P\left[\psi^{\alpha}\right] &= \sup_{P \in \mathcal{P}_0}E_P\left[\psi^{\alpha \cdot C}_1 \lor \psi^{\alpha\cdot (1-C)}_2\right]\\
&=\sup_{P \in \mathcal{P}_0} E_P\left[\psi^{\alpha \cdot C}_1 \right] + E_P\left[\psi^{\alpha\cdot (1-C)}_2\right] - E_P\left[\psi^{\alpha \cdot C}_1 \wedge \psi^{\alpha\cdot (1-C)}_2\right]\\
&\leq \alpha
\end{align} Under the alternative hypothesis, it follows that \begin{equation}
\inf_{P \in \mathcal{P}_1}E_P\left[\psi^{\alpha}\right] \geq \inf_{P \in \mathcal{P}_1(\epsilon)} \left(E_P\left[\psi^{\alpha \cdot C}_1 \right] \lor E_P\left[\psi^{\alpha\cdot (1-C)}_2\right]\right) \geq 1-\beta
\end{equation} whenever $
\epsilon \geq \epsilon_1 \wedge \epsilon_2.
$
\end{proof}

\section{General lower bound on the performance of tests}

\begin{proposition}[\citet{ingsterNonparametricGoodnessofFitTesting2003},\citet{tsybakovIntroductionNonparametricEstimation2009}]\label{prop:two_fuzzy_hypotheses} Let $P_0$ be a distribution and $\mathcal{P}_1(\epsilon)$ be a set of distributions, and given data $X\sim P$ consider the hypotheses \begin{equation}
H_0: P = P_0 \vs H_1: P \in \mathcal{P}_1(\epsilon).
\end{equation} Let $\pi$ be a distribution supported on $\mathcal{P}_1(\epsilon)$. It follows that minimax risk \eqref{eq:minimax_risk_definition} is lower bounded \begin{equation}
R_*(\mathcal{P}_1(\epsilon)) > \beta  \text{ if }\chi^2\left(E_{P \sim \pi}P,P_0\right)\leq C_{\alpha,\beta}
\end{equation} where $C_{\alpha,\beta}$ is a positive constant that depends only on $\alpha$ and $\beta$. 
\end{proposition}

\section{Oracle inference}

\subsection{Performance of the oracle test (Lemma \ref{lemma:upperbound_oracle_test_arbitrary_predictions})}\label{appx:ub_multiple_orthogonal_gaussian_seq_l2_testing}

The proof of Lemma \ref{lemma:upperbound_oracle_test_arbitrary_predictions} depends on the performance of the chi-squared test (Lemma \ref{lemma:chi_squared_test}) and the quantification of the worst-case projection onto the span of the predictions under the alternative hypothesis (Lemma \ref{lemma:Astar}). In the following, we state both lemmas and prove Lemma \ref{lemma:upperbound_oracle_test_arbitrary_predictions}. The proof of Lemma \ref{lemma:Astar} is provided at the end of this appendix.

The following lemma quantifies the performance of the chi-squared test.

\begin{lemma}[\citet{ingsterAsymptoticallyMinimaxHypothesisI1993}]\label{lemma:chi_squared_test} For hypotheses \eqref{eq:multiple_orthogonal_gaussian_seq_l2_testing}, there exists a positive constant $C$ that depends only on $\alpha$ and $\beta$ such that the chi-squared test \eqref{eq:chi_squared_test} is powerful whenever \begin{equation}\label{eq:chi_squared_test_rate}
n \geq C \cdot \frac{\sqrt{d}}{\epsilon^2}.
\end{equation}
\end{lemma}

The next lemma quantifies the worst-case projection onto the span of the predictions under the alternative hypothesis.
\begin{lemma}\label{lemma:Astar} If $\Rank \hat{U} < d$ and $S_\epsilon \cap B_r \neq \emptyset$, it follows that there exists a unique $(\Rank \hat{U})$-dimensional vector $A_*(\epsilon,r)$ such that \begin{equation}
A_*(\epsilon,r) = \arginf_{\theta \in S_\epsilon \cap B_r} \norm{\Pi_{\hat{U}}\theta}_2^2. 
\end{equation}
\end{lemma}

We proceed to prove Lemma \ref{lemma:upperbound_oracle_test_arbitrary_predictions}.

\begin{proof}[Proof of Lemma \ref{lemma:upperbound_oracle_test_arbitrary_predictions}]

In the following, we indicate the conditions under which the chi-squared test \eqref{eq:chi_squared_test}, the worst-case cosine similarity test \eqref{eq:correlation_test}, and the likelihood-ratio test \eqref{eq:likelihood_ratio_test} are valid and powerful. Finally, we combine the conditions to state when the oracle test \eqref{eq:oracle_test} is valid and powerful.

Throughout the proof, we let  $\Rank \hat{U}=k$, and let the reduced singular value decomposition of $\hat{U}$ be \begin{equation}\label{eq:reduced_svd_1}
\hat{U}= \tilde{U}_k\Lambda_k\tilde{V}_k^T,
\end{equation} where the columns of $\tilde{U}_k$ span the same $k$-dimensional subspace spanned by $\hat{U}$. Furthermore, let the columns of $\tilde{U}_k^\perp$ be a basis for its orthogonal complement, a $(d-k)$-dimensional subspace.

\noindent\textbf{Analysis when $S_\epsilon$ and $B_r$ are not disjoint.} If $S_\epsilon \cap B_r \neq \emptyset$, the oracle test \eqref{eq:oracle_test} performs a Bonferroni correction between the chi-squared test \eqref{eq:chi_squared_test} and the worst-case cosine similarity test \eqref{eq:correlation_test}: \begin{equation}\label{eq:oracle_test_under_nonempty_intersection}
\psi(X) = \psi_{\chi^2}^{\alpha/2}(X) \lor \psi_{\Corr}^{\alpha/2}(X) \quad \textif S_\epsilon \cap B_r \neq \emptyset.
\end{equation} Let us first analyze the chi-squared test. By Lemma \ref{lemma:chi_squared_test}, it follows that the oracle test is powerful whenever \begin{equation}\label{eq:cond_chi2_power}
n \geq C_0 \cdot \frac{\sqrt{d}}{\epsilon^2} \textand S_\epsilon \cap B_r \neq \emptyset.
\end{equation} where $C_0$ is a constant that depends only on $\beta$ and $\alpha$. Turning to the  worst cosine similarity test \eqref{eq:correlation_test}, it holds that  $\psi_{\Corr}^{\alpha/2}(X) = I\left(T_r(X) \geq z_{1-{\alpha/2} }\right)$ where \begin{equation} 
T_r(X)=\sqrt{n}\cdot \inner{\tilde{U}_k^TX}{\tilde{u}(\epsilon,r)} \textand \tilde{u}(\epsilon,r) = \frac{A_*(\epsilon,r)}{\norm{A_*(\epsilon,r)}_2}.
\end{equation} By \eqref{eq:oracle_test_under_nonempty_intersection}, Proposition \ref{prop:chebyshev_bound}, and the fact that \begin{equation}
E_{P_\theta}[T_r(X)] = \sqrt{n}\cdot \inner{\tilde{U}_k^T\theta}{\tilde{u}(\epsilon,r)} \textand V_{P_\theta}[T_r(X)] = 1\quad \forall \theta \in \R^d,
\end{equation}  it holds that the oracle test is powerful whenever \begin{equation}\label{eq:power_condition_correlation_test}
\inf_{\theta \in H_1(\epsilon,r)} E_{P_\theta}[T_r(X)] \geq C_1 \cdot 1 \textand S_\epsilon \cap B_r \neq\emptyset
\end{equation} where $
H_1(\epsilon,r) = \{\theta \in \R^d: \norm{\theta}_2 \geq \epsilon \textand \theta \in B_r\}$ and $C_1$ is a constant that depends on $\beta$ and $\alpha$. Since $S_\epsilon \cap B_r \neq\emptyset$, it holds that \begin{equation}
\inf_{\theta \in H_1(\epsilon,r)} E_{P_\theta}[T_r(X)] = \inf_{\theta \in S_\epsilon \cap B_r} E_{P_\theta}[T_r(X)]. 
\end{equation} We compute the RHS of \eqref{eq:power_condition_correlation_test}. Since $k < d$ and $S_\epsilon \cap B_r \neq \emptyset$, it follows by Lemma \ref{lemma:Astar} that there exists a unique $k$-dimensional vector $A_*(\epsilon,r)$ such that \begin{equation}
\norm{A_*(\epsilon,r)}_2^2 = \inf_{\theta \in B_r \cap S_\epsilon}\norm{\Pi_{\hat{U}}\theta}_2^2 = \inf_{\theta \in B_r \cap S_\epsilon}\norm{\tilde{U}_k^T\theta}_2^2 = \gamma_*^2(\epsilon,r).
\end{equation} Consequently, it holds that \begin{equation}
\inner{\tilde{U}_k^T\theta-A_*(\epsilon,r)}{A_*(\epsilon,r)} \geq 0 \quad \forall \theta \in B_r \cap S_\epsilon,
\end{equation} and \begin{align}
\inf_{\theta \in B_r \cap S_\epsilon} \inner{\tilde{U}_k^T\theta}{A_*(\epsilon,r)} &= \inf_{\theta \in B_r \cap S_\epsilon} \norm{A_*(\epsilon,r)}_2^2 + \inner{\tilde{U}_k^T\theta-A_*(\epsilon,r)}{A_*(\epsilon,r)}\\
&=\norm{A_*(\epsilon,r)}_2^2.
\end{align} Therefore, it holds that \begin{equation}\label{eq:power_condition_correlation_test_2}
\inf_{\theta \in S_\epsilon \cap B_r} E_{P_\theta}[T_r(X)] = \sqrt{n}\cdot \norm{A_*(\epsilon,r)}_2= \sqrt{n}\cdot \gamma_*(\epsilon,r).
\end{equation} Thus, by \eqref{eq:power_condition_correlation_test} and \eqref{eq:power_condition_correlation_test_2}, it holds that the oracle test \eqref{eq:oracle_test} is powerful whenever \begin{equation}\label{eq:cond_proj_power}
\sqrt{n}\cdot \gamma_*(\epsilon,r) \geq C_1 \textand S_\epsilon \cap B_r \neq\emptyset,
\end{equation} where $C_1$ is a positive constant that depends only on $\alpha$ and $\beta$. Finally, by \eqref{eq:cond_chi2_power}, \eqref{eq:cond_proj_power} and Proposition \ref{prop:Bonferroni}, the oracle test is valid and powerful whenever \begin{equation}\label{eq:cond_bonferroni_power}
n \geq C_2 \cdot \left(\frac{1}{\gamma_*^2(\epsilon,r)} \wedge \frac{\sqrt{d}}{\epsilon^2}\right) \textand S_\epsilon \cap B_r \neq\emptyset,
\end{equation} where $C_2$ is a positive constant that depends only on $\alpha$ and $\beta$.

\noindent\textbf{Analysis of oracle test when $S_\epsilon$ and $B_r$ are disjoint.} When $S_\epsilon \cap B_r = \emptyset$, there is a unique point under the alternative hypothesis that is closest to the null hypothesis, and the oracle test performs a likelihood ratio between them: \begin{equation}
\psi(X) = \psi^{\alpha}_{\LR}(X) \quad \textif S_\epsilon \cap B_r = \emptyset.
\end{equation} It follows that the oracle test is powerful whenever \begin{equation}\label{eq:cond_lrt_power}
n \geq C_3 \cdot \norm{\theta_*(r)}_2^{-2} \textand S_\epsilon \cap B_r = \emptyset,
\end{equation} where $C_3$ is a positive constant that depends only on $\alpha$ and $\beta$.

\noindent\textbf{In summary}, by \eqref{eq:cond_bonferroni_power} and  \eqref{eq:cond_lrt_power}, it holds that the oracle test \eqref{eq:oracle_test} is valid and powerful whenever \begin{equation}
n \geq C\cdot \begin{dcases}
\frac{1}{\gamma_*^2(\epsilon,r)} \wedge \frac{\sqrt{d}}{\epsilon^2} &\textif S_\epsilon \cap B_r \neq \emptyset\\
\frac{1}{\norm{\theta_*(r)}_2^2}&\textif S_\epsilon \cap B_r = \emptyset
\end{dcases}
\end{equation} where $C$ is a positive constant that depends only on $\alpha$ and $\beta$. 

\end{proof}

We finish this appendix by proving Lemma \ref{lemma:Astar}.

\begin{proof}[Proof of Lemma \ref{lemma:Astar}]

Let $\Rank \hat{U}=k$, using the reduced singular value decomposition of $\hat{U}$ \eqref{eq:reduced_svd_1}, we can rewrite $A_*(\epsilon,r)$ as  \begin{equation}\label{eq:Astar_def}
A_*(\epsilon,r) = \arginf_{a \in \mathcal{A}(\epsilon,r)} \norm{a}_2^2 \where \mathcal{A}(\epsilon,r) = \left\{\tilde{U}_k^T\theta: \theta \in S_\epsilon \cap B_r\right\} \subseteq \R^k.
\end{equation} Since $S_\epsilon \cap B_r \neq \emptyset$, it follows that $\mathcal{A}(\epsilon,r) \neq \emptyset$. We prove that if $k < d$, it holds that \begin{equation}
\mathcal{A}(\epsilon,r) = \left\{a \in \R^k : \norm{a}_2 \leq \epsilon \textand \hat{U}^T\tilde{U}_ka \geq A(\epsilon,r)\right\}.
\end{equation} Thus, $\mathcal{A}(\epsilon,r)$ is the intersection of a convex ball and closed half-spaces, so it is closed and convex. Consequently, the minimizer of \eqref{eq:Astar_def} is unique. 

First, we prove that for any $a \in \R^k$ such that $\norm{a}_2 \leq \epsilon$ and $\hat{U}^T\tilde{U}_ka \geq A(\epsilon,r)$, there exists a $\theta \in S_\epsilon \cap B_r$ such that $\tilde{U}_k^T\theta=a$. Since $k < d$, there exists unit vector $z \in \Span \tilde{U}_k^\perp$. Let \begin{equation}
\theta  = \tilde{U}_ka + z \cdot \sqrt{\epsilon^2-\norm{a}_2^2}.
\end{equation} By definition, it follows that $\norm{\theta}_2 = \epsilon$ and $\tilde{U}_k^T\theta=a$. Furthermore, for $i \in [m]$, it holds that \begin{align}
\norm{\theta-\hat{\theta}_i}_2^2 &= \norm{\theta}_2^2 + h_i^2 - 2h_i\cdot \inner{\theta}{\hat{u}_i}\\
&= \norm{\theta}_2^2 + h_i^2 - 2h_i\cdot \inner{\tilde{U}_ka}{\hat{u}_i}\\
&= \norm{\theta}_2^2 + h_i^2 - 2h_i\cdot [\hat{U}^T\tilde{U}_ka]_i
\end{align} Since $\hat{U}^T\tilde{U}_ka \geq A(\epsilon,r)$, it follows that \begin{equation}
\norm{\theta-\hat{\theta}_i}_2^2 \leq \epsilon^2 + h_i^2 - 2h_i\cdot A_i(\epsilon,r) = r_i^2.
\end{equation} Thus, $\theta \in B_r$, and consequently,  $\tilde{U}_k^T\theta=a$ and $\theta \in B_r \cap S_\epsilon$.

For the reverse inclusion, we prove that for any $\theta \in S_\epsilon \cap B_r$ such that $\tilde{U}_k^T\theta=a$, it holds that $\norm{a}_2 \leq \epsilon$ and $\hat{U}^T\tilde{U}_ka \geq A(\epsilon,r)$. Decompose $\theta$ as \begin{equation}
\theta = \tilde{U}_ka + z \where z \in \Span \tilde{U}_k^\perp.
\end{equation} Since $\theta \in S_\epsilon$, it holds that $\epsilon^2 = \norm{\theta}_2^2 = \norm{a}_2^2 + \norm{z}_2^2$. Thus, $\norm{a}_2 \leq \epsilon$. Furthermore, since $\theta \in B_r$, for $i \in [m]$, it holds that \begin{equation}
r_i^2 \geq \norm{\theta-\hat{\theta}_i}_2^2 = \epsilon^2 + h_i^2 - 2h_i\cdot [\hat{U}^T\tilde{U}_ka]_i, 
\end{equation} which is equivalent to \begin{equation}
[\hat{U}^T\tilde{U}_ka]_i \geq A_i(\epsilon,r) \for i \in [m].
\end{equation} That is, $\hat{U}^T\tilde{U}_ka \geq A(\epsilon,r)$. In summary, $\norm{a}_2 \leq \epsilon$ and $\hat{U}^T\tilde{U}_ka \geq A(\epsilon,r)$.

\end{proof}

\subsection{Lower bound on the performance of oracle tests (Lemma \ref{lemma:lowerbound_oracle_test_arbitrary_predictions})}\label{appx:lb_multiple_orthogonal_gaussian_seq_l2_testing}

We will need the following result before proving Lemma \ref{lemma:lowerbound_oracle_test_arbitrary_predictions}.

\begin{proposition}[\citet{ingsterNonparametricGoodnessofFitTesting2003}]\label{prop:IngsterSuslina} Let $P_\theta$ denote the distribution $\Normal(\theta, I_d/n)$ for $\theta \in \R^d$. Consider $P_0$ and $P_\pi = E_{\theta \in \pi}[P_{\theta}]$. It holds that \begin{equation}
1 + \chi^2\left(P_\pi,P_0\right) = E_{\theta,\theta'\iid \pi}\left[h(\theta,\theta')\right] \where h(\theta,\theta') = \exp{\left(n \cdot \inner{\theta}{\theta'}\right)}.
\end{equation}
\end{proposition}

\noindent We proceed to prove Lemma \ref{lemma:lowerbound_oracle_test_arbitrary_predictions}.

\begin{proof}[Proof of Lemma \ref{lemma:lowerbound_oracle_test_arbitrary_predictions}]

Let $\Rank \hat{U}=k$, and let the reduced singular value decomposition  of $\hat{U}$ be \begin{equation}
\hat{U}= \tilde{U}_k\Lambda_k\tilde{V}_k^T  \where \Lambda_k = \text{Diagonal}(\lambda_1,\dots,\lambda_k),
\end{equation} $\lambda_1^2 \geq \dots \geq \lambda_k^2 >0$, $\tilde{U}_k$ spans the same $k$-dimensional subspace spanned by $\hat{U}$. Furthermore, let $\tilde{U}_k^\perp$ be a basis for its orthogonal complement, a $(d-k)$-dimensional subspace. 

Since $k \leq d/2$ and $S_\epsilon \cap B_r \neq \emptyset$, it follows  by Lemma \ref{lemma:Astar} that there exists a unique $k$-dimensional vector $A_*(\epsilon,r)$ that satisfies \begin{equation}\label{eq:definition_A_tilde}
\norm{A_*(\epsilon,r)}_2^2 = \min_{\theta \in S_\epsilon \cap B_r} \norm{\Pi_{\hat{U}}\theta}_2^2 = \min_{\theta \in S_\epsilon \cap B_r} \norm{\tilde{U}_k^T\theta}_2^2.
\end{equation} Define $\pi$ to be the distribution of $\theta$ where
\begin{equation}\label{eq:multiple_orthogonal_pi_definition}
\theta = \tilde{U}_k \cdot A_*(\epsilon,r) + \frac{\rho(\epsilon)}{\sqrt{d-k}}\cdot \tilde{U}_k^\perp \cdot z,
\end{equation} where $z \sim \text{Uniform}\left(\sqrt{d-k} \cdot S_2^{d-k-1}\right)$ and $\rho(\epsilon)=\sqrt{\epsilon^2-\norm{A_*(\epsilon,r)}_2^2}$. Note that by \eqref{eq:definition_A_tilde}, \begin{equation}
\norm{A_*(\epsilon,r)}_2^2 = \min_{\theta \in S_\epsilon \cap B_r}\norm{\Pi_{\tilde{U}_k}\theta}_2^2 \leq  \min_{\theta \in S_\epsilon \cap B_r}\norm{\theta}_2^2  =\epsilon^2.
\end{equation} Thus, $\rho(\epsilon)$ is well defined. Furthermore, by the definition of $\pi$, it holds that \begin{equation}
\norm{\theta}_2 = \epsilon \textand \hat{U}^T\theta =  \tilde{V}_k\Lambda_kA_*(\epsilon,r)  \quad \forall \theta \in \supp \pi.
\end{equation} By \eqref{eq:definition_A_tilde}, it follows that there exists $\tilde{\theta} \in B_r \cap S_\epsilon$ such that \begin{equation}
A_*(\epsilon,r) = \tilde{U}_k^T\tilde{\theta}.
\end{equation} Consequently, we have that \begin{equation}
\norm{\theta}_2 = \epsilon \textand \hat{U}^T\theta =  \tilde{V}_k\Lambda_k\tilde{U}_k^T\tilde{\theta} = \hat{U}^T\tilde{\theta} \geq A(\epsilon,r) \quad \forall \theta \in \supp \pi,
\end{equation} since any $\tilde{\theta} \in S_\epsilon \cap B_r$ satisfies the constraint  $\hat{U}^T\tilde{\theta}\geq A(\epsilon,r)$. This implies that $\pi$ is supported on the alternative hypothesis since \begin{equation}
\norm{\theta}_2 = \epsilon \textand \norm{\theta-\hat{\theta}_i}_2\leq r_i \for i \in [m] \iff \norm{\theta}_2 = \epsilon \textand\hat{U}^T\theta \geq A(\epsilon,r).
\end{equation} By Proposition \ref{prop:IngsterSuslina} and \eqref{eq:multiple_orthogonal_pi_definition}, it holds that \begin{equation}
1 + \chi^2\left(P_\pi,P_0\right) = E_{\theta,\theta'\iid \pi}\exp{\left(n \cdot \inner{\theta}{\theta'}\right)}
\end{equation} where \begin{equation}
\inner{\theta}{\theta'} = \norm{A_*(\epsilon,r)}_2^2 + \frac{\rho^2(\epsilon)}{d-k} \cdot \inner{z}{z'} \textand z,z' \iid \text{Uniform}\left(\sqrt{d-k} \cdot S_2^{d-k-1}\right).
\end{equation} Hence, we have that \begin{equation}
1 + \chi^2\left(P_\pi,P_0\right) = \exp\left(n \cdot \norm{A_*(\epsilon,r)}_2^2 \right) \cdot E_{z,z'\iid \text{Uniform}\left(\sqrt{d-k} \cdot S_2^{d-k-1}\right)}\exp\left(\frac{n\rho^2(\epsilon)}{d-k} \cdot \inner{z}{z'} \right).
\end{equation} Since the distribution of $z$ is $C_0$-sub-Gaussian, where $C_0$ is a positive constant that does not depend on $d$, it follows that \begin{equation}
1 + \chi^2\left(P_\pi,P_0\right) \leq \exp\left(n \cdot \norm{A_*(\epsilon,r)}_2^2 \right) \cdot E_{z'}\exp\left(\frac{C^2_0}{2}\cdot \frac{n^2}{(d-k)^2} \cdot \rho^4(\epsilon) \cdot \norm{z'}_2^2 \right)
\end{equation} where $\norm{z'}_2^2=d-k$. Thus, for any $\eta>0$, there exists $C_\eta$ such that $ \chi^2\left(P_\pi,P_0\right) \leq \eta$, whenever \begin{equation}
n \cdot \norm{A_*(\epsilon,r)}_2^2 + \frac{C^2_0}{2}\cdot \frac{n^2}{d-k}\cdot \rho^4(\epsilon) \leq C_\eta.
\end{equation} Using the assumption that $\Rank \hat{U} = k \leq d/2$ and the fact that $\rho^4(\epsilon) \leq \epsilon^4$, the above equation is implied by \begin{equation}
n \cdot \norm{A_*(\epsilon,r)}_2^2 + C^2_0\cdot \frac{n^2}{d}\cdot \epsilon^4 \leq C_\eta.
\end{equation} The statement of the lemma follows by Proposition \ref{prop:two_fuzzy_hypotheses}.
\end{proof}

\section{Adapting to orthogonal predictions}

\subsection{Performance of the adaptive test (Lemma \ref{lemma:upperbound_adaptive_test_orthogonal_predictions})}\label{appx:ub_union_multiple_orthogonal_gaussian_seq_l2_testing}

If tighter control on the type-I error is needed, we can alternatively aggregate the test statistics of the test \eqref{eq:projection_test} and \eqref{eq:orthogonal_projection_test}:  \begin{equation}\label{eq:no_bonferroni_test_with_standard_projection}
\psi^{\alpha}(X)= I\left(n \cdot \frac{\norm{\Pi_{\hat{U}}X}_2^2}{\sqrt{\Rank \hat{U}}} +  n \cdot \frac{\norm{\Pi_{\hat{U}^\perp}X}_2^2}{\sqrt{d-\Rank \hat{U}}} \geq q_{1-\alpha}\right)
\end{equation} where $q_{1-\alpha}$ is the $(1-\alpha)$-quantile of the mixture $$\frac{\chi^2_{\Rank \hat{U}}}{\sqrt{\Rank \hat{U}}} + \frac{\chi^2_{d-\Rank \hat{U}}}{\sqrt{d-\Rank \hat{U}}}.$$ The test \eqref{eq:no_bonferroni_test_with_standard_projection} achieves the same type-I and type-II error guarantees as the test \eqref{eq:bonferroni_test_with_standard_projection} up to constants. Thus, we prove Lemma \ref{lemma:upperbound_adaptive_test_orthogonal_predictions} using the test \eqref{eq:no_bonferroni_test_with_standard_projection}, but note that the same result holds for the test \eqref{eq:bonferroni_test_with_standard_projection}.

\begin{proof}[Proof of Lemma \ref{lemma:upperbound_adaptive_test_orthogonal_predictions}]

Since the predictions are mutually orthogonal, $\Rank \hat{U} = k = m$. Thus, $k$ and $m$ are equal. Throughout the proof, we use $k$ rather than $m$.

\noindent\textbf{Power analysis for the test \eqref{eq:no_bonferroni_test_with_standard_projection}.} For the test statistic of the adaptive test \eqref{eq:no_bonferroni_test_with_standard_projection}, it holds that \begin{equation}
n \cdot \frac{\norm{\Pi_{\hat{U}}X}_2^2}{\sqrt{k}} +  n \cdot \frac{\norm{\Pi_{\hat{U}^\perp}X}_2^2}{\sqrt{d-k}} \sim \frac{\chi^2_k\left(n \cdot \norm{\Pi_{\hat{U}}\theta}_2^2 \right)}{\sqrt{k}} + \frac{\chi^2_{d-k}\left(n \cdot \norm{\Pi_{\hat{U}^\perp}\theta}_2^2 \right)}{\sqrt{d-k}}  \quad \forall \theta \in \R^d.
\end{equation} Let $H_1$ denote the set of distributions under the alternative hypothesis \begin{align}
H_1 &= \bigcup_{r \in \R_+^k}H_1(\epsilon(r),r)\\
\where H_1(\epsilon,r) &= \left\{\theta \in \R^d: \norm{\theta}_2\geq \epsilon \textand \norm{\theta-\hat{\theta}_i}_2 \leq r_i \for i \in [k]\right\}.
\end{align} By Proposition \ref{prop:chebyshev_bound}, it follows that the test \eqref{eq:no_bonferroni_test_with_standard_projection} controls the type-II error by $\beta$ whenever \begin{equation}\label{eq:union_orthogonal_projection_test_condition}
n  \cdot \inf_{\theta \in H_1} f(\theta) \geq C \where f(\theta) = \frac{\norm{\Pi_{\hat{U}}\theta}_2^2}{\sqrt{k}} + \frac{\norm{\Pi_{\hat{U}^\perp}\theta}_2^2}{\sqrt{d-k}}
\end{equation} and $C$ is a constant that only depends on $\alpha$ and $\beta$.

\noindent\textbf{Minimum projection.} We proceed to characterize the minimum projection onto the span of $\hat{U}$. For any $\theta \in H_1$, there exists $r \in \R_+^k$ such that $\theta \in H_1(\epsilon(r),r)$. Furthermore, note that \begin{equation}
\gamma_*^2(\epsilon(r),r) \leq \epsilon^2(r) \iff S_{\epsilon(r)} \cap B_r \not= \emptyset  \where B_r = \bigcap_{i=1}^k B_2^d(\hat{\theta}_i,r_i).
\end{equation} In particular, if $\gamma_*^2(\epsilon(r),r) \leq \epsilon^2(r)$, we have that \begin{equation}
\inf_{\theta \in H_1(\epsilon(r),r)} f(\theta) = \inf_{\theta \in H_1(\epsilon(r),r)} \frac{\norm{\Pi_{\hat{U}}\theta}_2^2}{\sqrt{k}} + \frac{\epsilon^2(r)-\norm{\Pi_{\hat{U}}\theta}_2^2}{\sqrt{d-k}} \textand 
\inf_{\theta \in H_1(\epsilon(r),r)} \norm{\Pi_{\hat{U}}\theta}_2^2 = \gamma_*^2(\epsilon(r),r).
\end{equation} Therefore, the solution to the optimization is \begin{align}
\inf_{\theta \in H_1(\epsilon(r),r)}f(\theta) &= \inf_{\norm{\Pi_{\hat{U}}\theta}_2^2 \in [\gamma_*^2(\epsilon(r),r),\epsilon^2(r)]} \frac{\norm{\Pi_{\hat{U}}\theta}_2^2}{\sqrt{k}} + \frac{\epsilon^2(r)-\norm{\Pi_{\hat{U}}\theta}_2^2}{\sqrt{d-k}}\\
&= \begin{dcases}
\frac{\epsilon^2(r)}{\sqrt{k}} &\textif k > d/2\\
% \frac{2\epsilon^2(r)}{\sqrt{d-k}+\sqrt{k}} &\textif \gamma_*^2(\epsilon(r),r) \leq \epsilon^2(r) \cdot c_k \\
\frac{\gamma_*^2(\epsilon(r),r)}{\sqrt{k}} + \frac{\epsilon^2(r)-\gamma_*^2(\epsilon(r),r)}{\sqrt{d-k}} &\textif k \leq d/2%\textif \epsilon^2(r) \cdot c_k\leq \gamma_*^2(\epsilon(r),r) \leq \epsilon^2(r)
\end{dcases}\\
&\geq \begin{dcases}
\frac{\epsilon^2(r)}{\sqrt{d}} &\textif k > d/2 \textor \left(k \leq d/2 \textand \gamma_*^2(\epsilon(r),r) \leq \epsilon^2(r) \cdot c_k\right)\\
\frac{\gamma_*^2(\epsilon(r),r)}{\sqrt{k}}  &\textif k \leq d/2 \textand \epsilon^2(r) \cdot c_k\leq \gamma_*^2(\epsilon(r),r) \leq \epsilon^2(r)
\end{dcases}
% \label{eq:union_case1}
\end{align} where $c_k=\left[\sqrt{\frac{d}{k}-1}+1\right]^{-1}$. Equivalently, it holds that \begin{equation}
\inf_{\theta \in H_1(\epsilon(r),r)}f(\theta) 
\geq \begin{dcases}
\frac{\epsilon^2(r)}{\sqrt{d}} &\textif k > d/2\\
\frac{\gamma_*^2(\epsilon(r),r)}{\sqrt{k}}  \wedge \frac{\epsilon^2(r)}{\sqrt{d}} &\textif k \leq d/2 \textand S_{\epsilon(r)} \cap B_r \not= \emptyset
\end{dcases}
\label{eq:union_case1}
\end{equation}

\noindent\textbf{Disjoint neighborhoods.} Conversely, it holds that \begin{equation}
\gamma_*^2(\epsilon(r),r) > \epsilon^2(r) \iff S_{\epsilon(r)} \cap B_r = \emptyset \empty \where B_r = \bigcap_{i=1}^k B_2^d(\hat{\theta}_i,r_i)
\end{equation} Thus, for all $\theta \in B_r$ it holds that $\norm{\theta}_2 > \epsilon(r)$. Thus, $H_1(\epsilon(r),r)=B_r$, and the weighted minimum projection of $\theta$ onto the spans of $\hat{U}$ and $\hat{U}^\perp$ does not depend on $\epsilon(r)$: \begin{equation}
\inf_{\theta \in H_1(\epsilon(r),r)}f(\theta) = \inf_{\theta \in B_r}f(\theta) =\inf_{\theta \in B_r}\frac{\norm{\Pi_{\hat{U}}\theta}_2^2}{\sqrt{k}} + \frac{\norm{\Pi_{\hat{U}^\perp}\theta}_2^2}{\sqrt{d-k}}
\end{equation} We can rewrite this optimization. Namely, let \begin{equation}
\theta = \hat{U}x + \hat{U}^\perp z \where x \in \R^k \textand z \in \R^{d-k}.
\end{equation} Since $\hat{U}$ has orthonormal columns, it follows that \begin{equation}
\norm{\theta}_2^2 = \norm{x}_2^2 + \norm{z}_2^2 \textand \norm{\theta-\hat{\theta}_i}_2^2=\norm{x-h_i\cdot e_i}_2^2 + \norm{z}_2^2 \for i \in [k].
\end{equation} Thus, we can rewrite the optimization problem as \begin{equation}
\inf_{x \in \R^k, z \in \R^{d-k}} \frac{\norm{x}_2^2}{\sqrt{k}} +  \frac{\norm{z}_2^2}{\sqrt{d-k}}\st \norm{x}_2^2 \geq \sum_{i=1}^k\left(\frac{\norm{x}_2^2+h_i^2-r_i^2+\norm{z}_2^2}{2h_i}\right)_+^2.
\end{equation} Note that $z\neq 0$ increases  $\norm{z}_2^2$ and the norm of $x$ due to the constraint. Thus, the optimizer sets $z=0$. Therefore, we have that \begin{equation}
\inf_{\theta \in H_1(\epsilon(r),r)}f(\theta) = \inf_{\theta \in B_r}f(\theta) =\inf_{\theta \in B_r}\frac{\norm{\Pi_{\hat{U}}\theta}_2^2}{\sqrt{k}} + \frac{\norm{\Pi_{\hat{U}^\perp}\theta}_2^2}{\sqrt{d-k}} = \inf_{\theta \in B_r}\frac{\norm{\Pi_{\hat{U}}\theta}_2^2}{\sqrt{k}} 
\end{equation} Recalling that $\inf_{\theta \in B_r}\norm{\Pi_{\hat{U}}\theta}_2^2=\inf_{\theta \in B_r}\norm{\theta}_2^2$, it follows that \begin{equation}\label{eq:union_case2}
\inf_{\theta \in H_1(\epsilon(r),r)}f(\theta) = \inf_{\theta \in B_r}f(\theta) = \inf_{\theta \in B_r}\frac{\norm{\theta}_2^2}{\sqrt{k}} = \frac{\norm{\theta_*(r)}_2^2}{\sqrt{k}}
\end{equation} where $\theta_*(r)$ is the unique closest point to the null hypothesis, defined in \eqref{eq:theta_star}.

\noindent\textbf{In summary,} by \eqref{eq:union_case1} and \eqref{eq:union_case2}, for any $\theta \in H_1$, there exists $r \in \R_+^k$ such that $\theta \in H_1(\epsilon(r),r)$, and it holds that \begin{equation}\label{eq:union_minimum_projection}
\inf_{\theta \in H_1(\epsilon(r),r)}f(\theta)\geq \begin{dcases}
\frac{\epsilon^2(r)}{\sqrt{d}} &\textif k > d/2\\
\frac{\gamma_*^2(\epsilon(r),r)}{\sqrt{k}}  \wedge \frac{\epsilon^2(r)}{\sqrt{d}}&\textif k \leq d/2 \textand S_{\epsilon(r)} \cap B_r \not= \emptyset \\
\frac{\norm{\theta_*(r)}_2^2}{\sqrt{k}} &\textif S_{\epsilon(r)} \cap B_r = \emptyset
\end{dcases}\end{equation} Consequently, by \eqref{eq:union_minimum_projection} and the definition of $\epsilon(r)$, see \eqref{eq:epsilon_star_orthogonal_predictions}, it follows that \begin{equation}
n\cdot \inf_{\theta \in H_1}f(\theta) = \inf_{r \in \R_+^k}\left[n\cdot \inf_{\theta \in H_1(\epsilon(r),r)}f(\theta)\right]\geq \inf_{r \in \R_+^k} C \geq C.
\end{equation} Thus, by \eqref{eq:union_orthogonal_projection_test_condition}, the type-II error is uniformly controlled under the alternative hypothesis by the projection test \eqref{eq:no_bonferroni_test_with_standard_projection}.

\end{proof}

\subsection{Lower bound on the performance of adaptive tests (Lemma \ref{lemma:lowerbound_adaptive_test_orthogonal_predictions})}\label{appx:lb_union_multiple_orthogonal_gaussian_seq_l2_testing}

\begin{proof}[Proof of Lemma \ref{lemma:lowerbound_adaptive_test_orthogonal_predictions}]

Given any separation curve $\tilde{\epsilon}$, define \begin{equation}
\mathcal{P}(\tilde{\epsilon},r) = \left\{P_\theta: \norm{\theta}_2 \geq \tilde{\epsilon}(r) \textand \norm{\theta-\hat{\theta}_i}_2 \leq r_i \for i \in [m]\right\},
\end{equation} and let $\mathcal{P}(\tilde{\epsilon})$ denote the set of distributions under the alternative hypothesis \begin{equation}\label{eq:P_epsilon_tilde_orthogonal_predictions}
\mathcal{P}(\tilde{\epsilon}) = \bigcup_{r \in \R^m_+} \mathcal{P}(\tilde{\epsilon},r).
\end{equation} For any $\tilde{\epsilon}$ separation curve that satisfies Property \ref{property:separation_curve_orthogonal_predictions}, we need to show that if \begin{equation}
\exists r\in \R^m_+ \st \tilde{\epsilon}(r) \leq C\cdot \epsilon_*(r) \implies R_*(\mathcal{P}(\tilde{\epsilon}))>\beta.
\end{equation} We prove this statement using the contrapositive. For any $\tilde{\epsilon}$ separation curve that satisfies Property \ref{property:separation_curve_orthogonal_predictions}, we show that if \begin{equation}
R_*(\mathcal{P}(\tilde{\epsilon}))\leq \beta \implies \forall r\in \R^m_+,\ \tilde{\epsilon}(r) \geq C\cdot \epsilon_*(r).
\end{equation}

Consider a separation curve $\tilde{\epsilon}$ that satisfies Property \ref{property:separation_curve_orthogonal_predictions} and $R_*(\mathcal{P}(\tilde{\epsilon}))\leq \beta$. Fix $r \in \R_+^m$. Then by Lemma \ref{lemma:point_lb_union_multiple_orthogonal_gaussian_seq_l2_testing}, for $\epsilon=\tilde{\epsilon}(r)$ it holds that \begin{equation}\label{eq:lb_condition_1}
n \geq C \cdot \left(\frac{\sqrt{d}}{\epsilon^2} \wedge \frac{\sqrt{m}}{\gamma_*^2(\epsilon,r)}\right) \textand \gamma_*(\epsilon,r) \leq \epsilon
\end{equation} and $C$ is a positive constant that depends only on $\alpha$ and $\beta$.  Since $S_\epsilon \cap B_r \neq \emptyset \iff \gamma_*(\epsilon,r) \leq \epsilon$, \eqref{eq:lb_condition_1} is equivalent to \begin{equation}\label{eq:lb_condition_2}
n \geq C \cdot \left(\frac{\sqrt{d}}{\epsilon^2} \wedge \frac{\sqrt{m}}{\gamma_*^2(\epsilon,r)}\right) \textand S_\epsilon \cap B_r \neq \emptyset.
\end{equation} By definition of $\epsilon_*(r)$, see \eqref{eq:epsilon_star_orthogonal_predictions_lowerbound}, it holds that \begin{equation} \label{eq:change_constant_orthogonal_2}
\epsilon_*(r) = \inf\left\{\epsilon\geq 0: n \geq C \cdot \left(\frac{\sqrt{d}}{\epsilon^2} \wedge \frac{\sqrt{m}}{\gamma_*^2(\epsilon,r)}\right) \textand S_\epsilon \cap B_r \neq \emptyset\right\}.
\end{equation} where we used the fact that $S_\epsilon \cap B_r \not= \emptyset \iff \gamma_*(\epsilon,r) \leq \epsilon$. Thus, it must hold that $
\tilde{\epsilon}(r) = \epsilon \geq  \epsilon_*(r).$ Since $r \in \R_+^m$ in the previous argument was arbitrary, the argument is valid for all $r \in \R_+^m$, which concludes the proof.
\end{proof}

Before proving Lemma \ref{lemma:point_lb_union_multiple_orthogonal_gaussian_seq_l2_testing}, we prove a looser result to warm up and get a simplified idea of the lower bound construction. First, we introduce a looser version of Property \ref{property:separation_curve_orthogonal_predictions}: \begin{property}[Same separation for equivalent accuracy vectors]\label{property:separation_curve_orthogonal_predictions_loose}
Let $\gamma(\epsilon,r)=\norm{A(\epsilon,r)}_2$. If $\tilde{\epsilon}(r) = \epsilon$ for  $r\in \R_+^m$, then  $\tilde{\epsilon}(r') = \epsilon$ for all $r' \in \R_+^m$ such that $\gamma(\epsilon,r)=\gamma(\epsilon,r')$.\end{property}

Note that the only difference between the sharp result of Lemma \ref{lemma:point_lb_union_multiple_orthogonal_gaussian_seq_l2_testing} and the looser result of Lemma \ref{lemma:alt_point_lb_union_multiple_orthogonal_gaussian_seq_l2_testing} is that the former uses $\norm{A_+(\epsilon,r)}_2$ in the rates while the latter uses the sub-optimal choice of $\norm{A(\epsilon,r)}_2$.

\begin{lemma}\label{lemma:alt_point_lb_union_multiple_orthogonal_gaussian_seq_l2_testing} For $m\leq d/2$, given separation curve $\tilde{\epsilon}$ with Property \ref{property:separation_curve_orthogonal_predictions_loose} and $r \in \R^m_+$, it holds that if \begin{equation}
n < C \cdot \left(\frac{\sqrt{m}}{\gamma^2(\epsilon,r)}\wedge \frac{\sqrt{d}}{\epsilon^2}\right) \textand \gamma(\epsilon,r)\leq \epsilon 
\end{equation} where $\epsilon = \tilde{\epsilon}(r)$, $\gamma(\epsilon,r)=\norm{A(\epsilon,r)}_2$,
and $C$ is a positive constant that depends on $\alpha$ and $\beta$, then the minimax risk \eqref{eq:minimax_risk_definition} over $\mathcal{P}(\tilde{\epsilon})$ \eqref{eq:P_epsilon_tilde_orthogonal_predictions} is not controlled:
$R_*(\mathcal{P}(\tilde{\epsilon})) > \beta$.
\end{lemma}

\begin{proof}[Proof of Lemma \ref{lemma:alt_point_lb_union_multiple_orthogonal_gaussian_seq_l2_testing}] Since the predictions are mutually orthogonal, $\Rank \hat{U} = k = m$. Thus, $k$ and $m$ are equal. Throughout the proof, we use $k$ rather than $m$. Henceforth, fix $r_* \in \R_+^k$ and let \begin{equation}
\epsilon = \tilde{\epsilon}(r_*) \textand \gamma = \gamma(\epsilon,r_*) = \norm{A(\epsilon,r_*)}_2.
\end{equation}

\par\noindent\textbf{Constructing accuracy vectors that match the desired norm.} Given $y \in S_2^{k-1}(0,1)$, $\gamma$ and $\epsilon$, define $k$-dimensional vector $r(\epsilon,\gamma,y)$ as follows \begin{equation}
r_i = \sqrt{\epsilon^2 + h_i^2 - 2h_i y_i \gamma} \for i \in [k].
\end{equation} If $0 \leq \gamma \leq \epsilon$, it follows that \begin{equation}
\epsilon^2 + h_i^2 - 2h_i y_i \gamma \geq (\epsilon-h_i)^2 \geq 0 \for i \in [k],
\end{equation} which guarantees that $r(\epsilon,\gamma,y) \in \R^k_+$. Furthermore, it holds that \begin{equation}
A(\epsilon ,r) = \gamma \cdot y \textand \norm{A(\epsilon,r)}_2 = \gamma \for r=r(\epsilon,\gamma,y).
\end{equation}

\noindent\textbf{Constructing distribution supported on the alternative hypothesis.} Recall that $\hat{U}$ spans a $k$-dimensional subspace. Henceforth, let $\hat{U}^\perp$ be a basis for its orthogonal complement, a $(d-k)$-dimensional subspace. Let $\pi$ be the distribution of \begin{equation}\label{eq:alt_union_multiple_orthogonal_pi_definition}
\theta = \hat{U}A(\epsilon ,r)  + \hat{U}^\perp z \cdot \frac{\rho(\gamma)}{\sqrt{d-k}}
\end{equation} where $r=r\left(\epsilon,\gamma,\frac{y}{\sqrt{k}}\right)$, $\rho(\gamma) = \sqrt{\epsilon^2-\gamma^2}$ and \begin{equation}
y \sim \text{Uniform}\left(\sqrt{k} \cdot S_2^{k-1}\right), z \sim \text{Uniform}\left(\sqrt{d-k} \cdot S_2^{d-k-1}\right) \textand y \perp\!\!\!\perp z.
\end{equation} By definition of $\pi$ it holds that \begin{equation}
\forall \theta \in \supp \pi, \exists y \in S_2^{k-1} \st \hat{U}^T\theta = A(\ \epsilon\ ,\ r(\epsilon,\gamma,y) \ ) \textand \norm{\theta}_2=\epsilon,
\end{equation} thus \begin{equation}
\forall \theta \in \supp \pi, \exists y \in S_2^{k-1} \st \norm{\theta}_2=\epsilon \textand \norm{\theta-\hat{\theta}_i}_2\leq r_i(\epsilon,\gamma,y) \for i \in [k].
\end{equation} Finally, Property \ref{property:separation_curve_orthogonal_predictions_loose} implies that $\pi$ is supported on the alternative hypothesis.

\noindent\textbf{Chi-squared distance between null hypothesis and mixture.} By Proposition \ref{prop:IngsterSuslina} and  \eqref{eq:alt_union_multiple_orthogonal_pi_definition}, it holds that \begin{equation}
1 + \chi^2\left(P_\pi,P_0\right) = E_{\theta,\theta'\iid \pi}\exp{\left(n \cdot \inner{\theta}{\theta'}\right)}
\end{equation} where \begin{equation}
\inner{\theta}{\theta'} = \frac{\gamma^2}{k} \inner{y}{y'} + \frac{\rho^2(\gamma)}{d-k} \cdot \inner{z}{z'}.
\end{equation} Hence, we have that \begin{equation}
1 + \chi^2\left(P_\pi,P_0\right) = E_{y,y'}\exp\left(\frac{n\gamma^2}{k} \cdot \inner{y}{y'} \right) \cdot E_{z,z'}\exp\left(\frac{n\rho^2(\gamma)}{d-k} \cdot \inner{z}{z'} \right).
\end{equation} Since the distributions of $y$ and $z$ are $C_0$-sub-Gaussian, where $C_0$ is a positive constant that does not depend on $d$, it follows that \begin{equation}
1 + \chi^2\left(P_\pi,P_0\right) \leq E_{y'}\exp\left(\frac{C^2_0}{2}\cdot \frac{n^2}{k^2} \cdot \gamma^4 \cdot \norm{y'}_2^2 \right) \cdot E_{z'}\exp\left(\frac{C^2_0}{2}\cdot \frac{n^2}{(d-k)^2} \cdot \rho^4(\gamma) \cdot \norm{z'}_2^2 \right)
\end{equation} where $\norm{y'}_2^2=k$ and $\norm{z'}_2^2=d-k$. Thus, for any $\eta>0$, there exists $C_\eta$ such that $ \chi^2\left(P_\pi,P_0\right) \leq \eta$, whenever \begin{equation}
n^2 \cdot \frac{\gamma^4}{k} + \frac{n^2}{d-k}\cdot \rho^4(\gamma) \leq C_\eta
\end{equation} where $C_\eta$ is a positive constant that depends only on $\eta$. Thus, by Proposition \ref{prop:two_fuzzy_hypotheses}, it follows that there exists a positive constant $C$ that depends only on $\alpha$ and $\beta$ such that whenever \begin{equation}
n < C \cdot \min\left(\frac{\sqrt{k}}{\gamma^2},\frac{\sqrt{d-k}}{\epsilon^2-\gamma^2}\right) \textand \gamma^2 \leq \epsilon^2,
\end{equation} it follows that $
R_*(\mathcal{P}(\tilde{\epsilon})) > \beta$. The statement of the lemma follows by noting that for $k \leq d/2$, it holds that \begin{equation}
\min\left(\frac{\sqrt{k}}{\gamma^2},\frac{\sqrt{d-k}}{\epsilon^2-\gamma^2}\right) \geq \tilde{C} \cdot \min\left(\frac{\sqrt{k}}{\gamma^2},\frac{\sqrt{d}}{\epsilon^2}\right) %\begin{dcases}
% \sqrt{d}\cdot \epsilon^{-2} &\textif \gamma^{2} \leq \epsilon^2 \cdot c_k\\
% \sqrt{k} \cdot \gamma^{-2} &\textif \epsilon^2 \cdot c_k \leq \gamma^{2} \leq \epsilon^2
% \end{dcases}
\end{equation} where $\tilde{C}$ is a positive constant.
\end{proof}

We are ready to improve the lower bound argument by proving Lemma \ref{lemma:point_lb_union_multiple_orthogonal_gaussian_seq_l2_testing}. The construction in Lemma \ref{lemma:alt_point_lb_union_multiple_orthogonal_gaussian_seq_l2_testing} is loose because for some $\theta$ in the support of $\pi$, some of the entries of $\hat{U}^T\theta$ are allowed to be negative. Thus, we should restrict the support of $\pi$ so that  $\hat{U}^T\theta \geq 0$ element-wise. A natural idea is to restrict $y$ in \eqref{eq:alt_union_multiple_orthogonal_pi_definition} to belong to the positive orthant, but doing so would destroy the sub-Gaussian property on which the lemma relies. This is because $y \sim \text{Uniform}(\{x\in \R^m \st \norm{x}=\sqrt{m} \textand x \geq 0\})$ is $C$-sub-Gaussian but $C$ depends on the dimension $m$.

To correct this issue in the proof of Lemma \ref{lemma:lowerbound_adaptive_test_orthogonal_predictions}, we will construct the support of $\pi$ using \textit{sparse} $y$ vectors that belong to the positive orthant. We will use the standard fact about the hypergeometric and binomial distributions. \begin{proposition}[Convex ordering of hypergeometric and binomial random variables]\label{prop:convex_ordering}
For $X \sim \text{Hypergeometric}(N,k,n)$ and $Y \sim \text{Binomial}(n,\frac{k}{N})$, it holds that $E[f(X)]\leq E[f(Y)]$ for every convex function $f$. Consequently, for every $\lambda\geq 0$ it holds that \begin{equation}
E[\exp(\lambda X)] \leq E[\exp(\lambda Y)] = \left(1-\frac{k}{N}+\frac{k}{N}\exp(\lambda)\right)^n \leq \exp\left(\frac{k\cdot n}{N}(e^\lambda-1)\right)
\end{equation}
\end{proposition} In the following, we prove Lemma \ref{lemma:point_lb_union_multiple_orthogonal_gaussian_seq_l2_testing}.

\begin{lemma}\label{lemma:point_lb_union_multiple_orthogonal_gaussian_seq_l2_testing} For $m\leq d/2$, given separation curve $\tilde{\epsilon}$ with Property \ref{property:separation_curve_orthogonal_predictions} and $r \in \R^m_+$, it holds that if \begin{equation}
n < C \cdot \left(\frac{\sqrt{d}}{\epsilon^2} \wedge \frac{\sqrt{m}}{\gamma_*^2(\epsilon,r)}\right) \textand \gamma_*(\epsilon,r) \leq \epsilon, 
\end{equation}where $\epsilon = \tilde{\epsilon}(r)$, $\gamma_*(\epsilon,r)=\norm{A_+(\epsilon,r)}_2$, and $C$ is a positive constant that depends on $\alpha$ and $\beta$, then the minimax risk \eqref{eq:minimax_risk_definition} over $\mathcal{P}(\tilde{\epsilon})$ \eqref{eq:P_epsilon_tilde_orthogonal_predictions} is not controlled: $
R_*(\mathcal{P}(\tilde{\epsilon})) > \beta.$\end{lemma}
\begin{proof} Since the predictions are mutually orthogonal, $\Rank \hat{U} = k = m$. Thus, $k$ and $m$ are equal. Throughout the proof, we use $k$ rather than $m$. Henceforth, fix $r_* \in \R^k_+$, and let \begin{equation}
\epsilon = \tilde{\epsilon}(r_*) \textand \gamma = \gamma_*(\epsilon,r_*) = \norm{A_+(\epsilon,r_*)}_2.
\end{equation}

\par\noindent\textbf{Constructing accuracy vectors that match the desired norm.} Given a subset $S \subseteq \{1,\dots,k\}$ of size $M \in [k]$,  $\gamma$ and $\epsilon$, define the $k$-dimensional vector $r(\epsilon,\gamma,S)$ as follows \begin{equation}
r_i = \sqrt{\epsilon^2 + h_i^2 - 2h_i\gamma\cdot \frac{I(i \in S)}{\sqrt{M}}} \for i \in [k].
\end{equation} If $0 \leq \gamma \leq \epsilon$, it follows that \begin{equation}
\epsilon^2 + h_i^2 - 2h_i \gamma \cdot \frac{I(i \in S)}{\sqrt{M}} \geq (\epsilon-h_i)^2 \geq 0 \for i \in [k],
\end{equation} which guarantees that $r(\epsilon,\gamma,S) \in \R^k_+$. Furthermore, it holds that \begin{equation} A_i(\epsilon,r) = \gamma \cdot \frac{I(i\in S)}{\sqrt{M}} \geq 0\textand \norm{A_+(\epsilon,r)}_2 = \gamma \for r=r(\epsilon,\gamma,S).
\end{equation}

\noindent\textbf{Constructing a distribution supported on the alternative hypothesis.} Recall that $\hat{U}$ spans a $k$-dimensional subspace. Henceforth, let $\hat{U}^\perp$ be a basis for its orthogonal complement, a $(d-k)$-dimensional subspace. Furthermore, let $\mathcal S$ denote the collection of all $\binom{k}{M}$ subsets of $[k]$ having size $M$. Let $\pi$ be the distribution of \begin{equation}\label{eq:union_multiple_orthogonal_pi_definition}
\theta = \hat{U}A(\epsilon ,r)  + \hat{U}^\perp z \cdot \frac{\rho(\gamma)}{\sqrt{d-k}}
\end{equation} where $r=r\left(\epsilon,\gamma,S\right)$, $\rho(\gamma) = \sqrt{\epsilon^2-\gamma^2}$, and \begin{equation}
S \sim \text{Uniform}\left(\mathcal{S}\right), z \sim \text{Uniform}\left(\sqrt{d-k} \cdot S_2^{d-k-1}\right) \textand S \perp\!\!\!\perp z.
\end{equation} By definition of $\pi$ it holds that \begin{equation}
\forall \theta \in \supp \pi, \exists S \in \mathcal{S} \st \hat{U}^T\theta = A(\ \epsilon\ ,\ r(\epsilon,\gamma,S) \ ) \textand \norm{\theta}_2=\epsilon,
\end{equation} thus \begin{equation}
\forall \theta \in \supp \pi, \exists S \in \mathcal{S} \st \norm{\theta}_2=\epsilon \textand \norm{\theta-\hat{\theta}_i}_2\leq r_i(\epsilon,\gamma,S) \for i \in [k].
\end{equation} Due to Property \ref{property:separation_curve_orthogonal_predictions}, it holds that $\pi$ is supported on the alternative hypothesis.

\noindent\textbf{Chi-squared distance between null hypothesis and mixture.} By Proposition \ref{prop:IngsterSuslina} and \eqref{eq:union_multiple_orthogonal_pi_definition}, it holds that \begin{equation}
1 + \chi^2\left(P_\pi,P_0\right) = E_{\theta,\theta'\iid \pi}\exp{\left(n \cdot \inner{\theta}{\theta'}\right)}
\end{equation} where \begin{equation}
\inner{\theta}{\theta'} = \frac{\gamma^2}{M}  \cdot |S\cap S'| + \frac{\rho^2(\gamma)}{d-k} \cdot \inner{z}{z'}.
\end{equation} Hence, we have that \begin{equation}\label{eq:chi2}
1 + \chi^2\left(P_\pi,P_0\right) = E_{S,S'}\exp\left(\frac{n\gamma^2}{M} \cdot |S\cap S'| \right) \cdot E_{z,z'}\exp\left(\frac{n\rho^2(\gamma)}{d-k} \cdot \inner{z}{z'} \right).
\end{equation}

\noindent\textbf{Bound for second term in RHS of \eqref{eq:chi2}.} Since the distribution of $z$ is $C_0$-sub-Gaussian, where $C_0$ is a positive constant that does not depend on $d$, it follows that \begin{align}
E_{z,z'}\exp\left(\frac{n\rho^2(\gamma)}{d-k} \cdot \inner{z}{z'} \right) &\leq E_{z'}\exp\left(\frac{C^2_0}{2}\cdot \frac{n^2}{(d-k)^2} \cdot \rho^4(\gamma) \cdot \norm{z'}_2^2 \right)\\ &= \exp\left(\frac{C^2_0}{2}\cdot \frac{n^2}{d-k} \cdot \rho^4(\gamma) \right)
\end{align} where we used the fact that $\norm{z'}_2^2=d-k$ in the last equality.

\noindent\textbf{Bound for first term in RHS of \eqref{eq:chi2}.} Fix $S'$, and note that $|S\cap S'|$ can be rewritten as follows \begin{equation}
|S\cap S'| = \sum_{i \in S'}I(i \in S)
\end{equation} Fix $S'$. Consider an urn containing $k$ balls labeled by $[k]$, with the $M=|S'|$ balls indexed by $S'$ colored red and the remaining $k-M$ balls colored blue. Drawing $M=|S|$ balls uniformly without replacement produces the subset $S$. Therefore, $|S\cap S'|$ is the number of red balls drawn, and hence \begin{equation}
|S\cap S'|\ |S' \sim \text{Hypergeometric}(k,M,M).
\end{equation} Since the RHS does not depend on $S'$, it follows that \begin{equation}
|S\cap S'| \sim \text{Hypergeometric}(k,M,M).
\end{equation} By Proposition \ref{prop:convex_ordering}, it holds that \begin{equation}
E_{S,S'}\exp\left(\frac{n\gamma^2}{M} \cdot |S\cap S'| \right) \leq \exp\left(\frac{M^2}{k}\cdot \left(\exp\left(\frac{n\gamma^2}{M}\right)-1\right)\right)
\end{equation}

\noindent\textbf{Bound for \eqref{eq:chi2}.} Choosing $M=\sqrt{k}$, it follows that \begin{equation}
1 + \chi^2\left(P_\pi,P_0\right) \leq  \exp\left(\exp\left(\frac{n\gamma^2}{\sqrt{k}}\right)-1+\frac{C^2_0}{2}\cdot \frac{n^2}{d-k} \cdot \rho^4(\gamma) \right)
\end{equation} Thus, for any $\eta>0$, there exists $C_\eta$ such that $ \chi^2\left(P_\pi,P_0\right) < \eta$, whenever \begin{equation}
\exp\left(\frac{n\gamma^2}{\sqrt{k}}\right)-1+\frac{C^2_0}{2}\cdot \frac{n^2}{d-k} \cdot \rho^4(\gamma) < C_\eta.
\end{equation} where $C_\eta$ is a positive constant that depends on $\eta$. Thus, by Proposition \ref{prop:two_fuzzy_hypotheses}, it follows that there exists a positive constant $C$ that depends on $\alpha$ and $\beta$ such that whenever \begin{equation}
n < C \cdot \min\left(\frac{\sqrt{k}}{\gamma^2},\frac{\sqrt{d-k}}{\epsilon^2-\gamma^2}\right) \textand \gamma^2 \leq \epsilon^2,
\end{equation} it follows that $
R_*(\mathcal{P}(\tilde{\epsilon})) > \beta$. The statement of the lemma follows by noting that for $k \leq d/2$, it holds that \begin{equation}
\min\left(\frac{\sqrt{k}}{\gamma^2},\frac{\sqrt{d-k}}{\epsilon^2-\gamma^2}\right) \geq \tilde{C} \cdot \min\left(\frac{\sqrt{k}}{\gamma^2},\frac{\sqrt{d}}{\epsilon^2}\right)
\end{equation} where $\tilde{C}$ is a positive constant.
\end{proof}

\section{Adapting to arbitrary predictions}

\subsection{Performance of adaptive tests  (Lemmas \ref{lemma:performance_adaptive_test_with_standard_projection} and \ref{lemma:upperbound_adaptive_test_arbitrary_predictions})}\label{appx:ub_union_multiple_gaussian_seq_l2_testing}

In the following, we provide a proof of Lemma \ref{lemma:upperbound_adaptive_test_arbitrary_predictions}. We omit the proof of Lemma \ref{lemma:performance_adaptive_test_with_standard_projection} since it is analogous.

\begin{proof}[Proof of Lemma \ref{lemma:upperbound_adaptive_test_arbitrary_predictions}]

\noindent\textbf{Distribution of test statistic of the intrinsic projection test \eqref{eq:intrinsic_projection_test}.}
Recall that $\hat{U}$ is the $d\times m$ matrix where the $i$-th column is the prediction $\hat{u}_i$. Let $k=\Rank(\hat{U}) \leq \min(m,d)$. Then by the singular value decomposition we have that \begin{equation}\label{eq:reduced_svd}
\hat{U} = \tilde{U}_k\Lambda_k\tilde{V}_k^T \where \tilde{U}_k \in \R^{d \times k}, \Lambda_k\in \R^{k \times k} \textand \tilde{V}_k \in \R^{m \times k} 
\end{equation} where both $\tilde{U}_k$ and $\tilde{V}_k$ have orthonormal columns, and \begin{equation}
\Lambda_{k} = \text{Diagonal}(\lambda_1,\dots,\lambda_k) \st \lambda_1^2 \geq \lambda_2^2 \geq \dots \geq \lambda_k^2 > 0.
\end{equation} Therefore, it holds that \begin{equation}
T_{\hat{U}}(X)= \frac{n}{\norm{\hat{U}}^2_{\op}} \cdot \norm{\hat{U}^TX}_2^2 = \sum_{i=1}^k \frac{\lambda_i^2}{\lambda_1^2} \cdot \norm{\sqrt{n}\cdot \Pi_{\tilde{u}_i}X}_2^2.
\end{equation} It always holds that \begin{equation}
T_{\hat{U}}(X) \sim \sum_{i=1}^k \frac{\lambda_i^2}{\lambda_1^2} \cdot \chi^2_1\left(n \cdot \norm{\Pi_{\tilde{u}_i}\theta}_2^2\right) \quad \forall \theta \in \R^d
\end{equation} where each term is independent. Thus, the expectation and variance are \begin{align}
\lambda_1^2 \cdot E_{P_\theta}\left[T_{\hat{U}}(X)\right] &=  \norm{\lambda}_2^2 + n \cdot \sum_{i=1}^k \lambda_i^2 \cdot \norm{\Pi_{\tilde{u}_i}\theta}_2^2\\
&=  \norm{\lambda}_2^2 + n \cdot \norm{\Lambda_k \tilde{U}_k^T\theta}_2^2\\
&=  \norm{\lambda}_2^2 + n \cdot \norm{ \hat{U}^T\theta}_2^2, \label{eq:expectation_projection_test}\\
\lambda_1^4 \cdot V_{P_\theta}\left[T_{\hat{U}}(X)\right] &= 2 \cdot \norm{\lambda}_4^4 + 4n \cdot \sum_{i=1}^k \lambda_i^4 \cdot \norm{\Pi_{\tilde{u}_i}\theta}_2^2,\\
&\leq 2 \cdot \norm{\lambda}_4^4 + 4 \cdot \lambda_1^2 \cdot n \cdot \norm{\Lambda_k \tilde{U}_k^T\theta}_2^2\\
&=2 \cdot \norm{\lambda}_4^4 + 4 \cdot \lambda_1^2 \cdot n \cdot \norm{\hat{U}^T\theta}_2^2. \label{eq:variance_projection_test}
\end{align}

\noindent\textbf{Condition for the intrinsic projection test to be powerful.} Henceforth, let $\epsilon(r)$ be an arbitrary separation curve, and let $H_1$ denote the set of distributions under the alternative hypothesis \begin{align}
H_1 = \bigcup_{r \in \R_+^m}H_1(\epsilon(r),r)\where H_1(\epsilon,r) = \left\{\theta \in \R^d: \norm{\theta}_2\geq \epsilon \textand \theta \in B_r\right\}.
\end{align} By \eqref{eq:expectation_projection_test}, \eqref{eq:variance_projection_test}, and Proposition \ref{prop:chebyshev_bound}, it follows that the intrinsic projection test \eqref{eq:intrinsic_projection_test} is powerful whenever \begin{equation}\label{eq:type2_control_cond}
n  \cdot \inf_{\theta \in H_1}\norm{\hat{U}^T\theta}_2^2 \geq C \cdot \left(\norm{\lambda}_4^2 + \lambda_1^2\right)
\end{equation} where $C$ is a constant that only depends on $\alpha$ and $\beta$. Noting that $\norm{\lambda}_4^2 \geq \lambda_1^2$, the above is implied by \begin{equation}
n \cdot \inf_{\theta \in H_1}\norm{\hat{U}^T\theta}_2^2 \geq 2C \cdot \norm{\lambda}_4^2.
\end{equation} Equivalently, the condition can be written as follows \begin{equation}\label{eq:power_condition_for_arbitary_predictions}
n \cdot \inf_{\theta \in H_1}\frac{\norm{\hat{U}^T\theta}_2^2}{\lambda_1^2} \geq 2C \cdot \sqrt{\intrinsic \Sigma^2}\ \text{ since }\frac{\norm{\lambda}_4^2}{\lambda_1^2} =  \sqrt{\frac{\trace \Sigma^2}{\norm{\Sigma^2}_{\op}}} =  \sqrt{\intrinsic \Sigma^2}.
\end{equation}

\noindent\textbf{The minimum intrinsic projection.}  For any $\theta \in H_1$, there exists $r \in \R^m_+$ such that $\theta \in H_1(\epsilon(r),r)$. Since $\theta \in H_1(\epsilon(r),r)$, it holds that $H_1(\epsilon(r),r)\neq\emptyset$ and consequently $B_r \not\subset B_{\epsilon(r)}^\circ$. Therefore, one of two situations must happen: either $B_r \cap S_{\epsilon(r)} \neq \emptyset$, or $B_r \cap S_{\epsilon(r)} = \emptyset$. Consequently,  \begin{equation}\label{eq:minimum_projection}
\inf_{\theta \in H_1(\epsilon(r),r)} \frac{\norm{\hat{U}^T\theta}_2^2}{\lambda_1^2} = \begin{cases}
\tilde{\gamma}^2_*(\epsilon(r),r) &\textif B_r \cap S_{\epsilon(r)} \neq \emptyset\\
\norm{\tilde{\theta}_*(r)}_2^2 &\textif B_r \cap S_{\epsilon(r)} = \emptyset
\end{cases}
\end{equation}

\noindent\textbf{Condition for the aggregated test to be powerful.} Since \begin{equation}
\inf_{\theta \in H_1}\norm{\hat{U}^T\theta}_2^2 = \inf_{r \in \R_+^m}\inf_{\theta \in H_1(\epsilon(r),r)}\norm{\hat{U}^T\theta}_2^2,
\end{equation} by \eqref{eq:power_condition_for_arbitary_predictions} and \eqref{eq:minimum_projection}, the test is powerful whenever \begin{equation}\label{eq:power_condition_for_projection_test_arbitary_predictions}
2C \leq n \cdot \inf_{r \in \R_+^m} \frac{1}{\sqrt{\intrinsic \Sigma^2}} \cdot \begin{cases}
\tilde{\gamma}^2_*(\epsilon(r),r) &\textif B_r \cap S_{\epsilon(r)} \neq \emptyset\\
\norm{\tilde{\theta}_*(r)}_2^2 &\textif B_r \cap S_{\epsilon(r)} = \emptyset
\end{cases}
\end{equation} Furthermore, by Lemma \ref{lemma:chi_squared_test}, we know that the chi-squared test is powerful whenever \begin{equation}\label{eq:power_condition_for_chi_squared_test_arbitary_predictions}
n \cdot \inf_{r \in \R_+^m}\frac{\epsilon^2(r)}{\sqrt{d}} \geq C'
\end{equation} where $C'$ is a positive constant that depends on $\alpha$ and $\beta$. Consequently, by \eqref{eq:power_condition_for_projection_test_arbitary_predictions}, \eqref{eq:power_condition_for_chi_squared_test_arbitary_predictions}, and Proposition \ref{prop:Bonferroni} it holds that the aggregated test \eqref{eq:bonferroni_test_with_intrinsic_projection} is powerful whenever \begin{equation}
2C \lor C' \leq n \cdot \left(\left[\inf_{r \in \R_+^m}\frac{\epsilon^2(r)}{\sqrt{d}} \right] \wedge  \inf_{r \in \R_+^m}\frac{1}{\sqrt{\intrinsic \Sigma^2}} \cdot \begin{cases}
\tilde{\gamma}^2_*(\epsilon(r),r) &\textif B_r \cap S_{\epsilon(r)} \neq \emptyset\\
\norm{\tilde{\theta}_*(r)}_2^2 &\textif B_r \cap S_{\epsilon(r)} = \emptyset
\end{cases}.\right)
\end{equation} which is implied by \begin{equation}
2C \lor C' \leq n \cdot \inf_{r \in \R_+^m}\left(\frac{\epsilon^2(r)}{\sqrt{d}} \wedge  \frac{1}{\sqrt{\intrinsic \Sigma^2}} \cdot \begin{cases}
\tilde{\gamma}^2_*(\epsilon(r),r) &\textif B_r \cap S_{\epsilon(r)} \neq \emptyset\\
\norm{\tilde{\theta}_*(r)}_2^2 &\textif B_r \cap S_{\epsilon(r)} = \emptyset
\end{cases}.\right)
\end{equation} The lemma is proved by noting that $\epsilon(r) = \epsilon_*(r)$, see \eqref{eq:epsilon_star_arbitary_predictions}, satisfies the above condition.

\end{proof}

\subsection{Lower bound on the performance of adaptive tests (Lemma \ref{lemma:lb_union_multiple_gaussian_seq_l2_testing})}\label{appx:lb_union_multiple_gaussian_seq_l2_testing}

The proof of Lemma \ref{lemma:lb_union_multiple_gaussian_seq_l2_testing} follows by an argument analogous to that used for Lemma \ref{lemma:lowerbound_adaptive_test_orthogonal_predictions}.

\begin{proof}[Proof of Lemma \ref{lemma:lb_union_multiple_gaussian_seq_l2_testing}]

Given any separation curve $\tilde{\epsilon}$, let \begin{equation}
\mathcal{P}(\tilde{\epsilon},r) = \left\{P_\theta: \norm{\theta}_2 \geq \tilde{\epsilon}(r) \textand \norm{\theta-\hat{\theta}_i}_2 \leq r_i \for i \in [m]\right\},
\end{equation} and \begin{equation}\label{eq:P_epsilon_tilde_arbitary_predictions}
\mathcal{P}(\tilde{\epsilon}) = \bigcup_{r \in \R^m_+} \mathcal{P}(\tilde{\epsilon},r).
\end{equation} For any $\tilde{\epsilon}$ separation curve that satisfies Property \ref{property:separation_curve_arbitary_predictions}, we need to show that if \begin{equation}
\exists r\in \R^m_+ \st \tilde{\epsilon}(r) \leq C\cdot \tilde{\epsilon}_*(r) \implies R_*(\mathcal{P}(\tilde{\epsilon}))>\beta.
\end{equation} We prove this statement using the contrapositive. For any $r\mapsto \tilde{\epsilon}(r)$ separation curve that satisfies Property \ref{property:separation_curve_arbitary_predictions}, we show that if \begin{equation}
R_*(\mathcal{P}(\tilde{\epsilon}))\leq \beta \implies \forall r\in \R^m_+,\ \tilde{\epsilon}(r) \geq C\cdot \tilde{\epsilon}_*(r).
\end{equation}

Consider a separation curve $\tilde{\epsilon}$ that satisfies Property \ref{property:separation_curve_arbitary_predictions} and $R_*(\mathcal{P}(\tilde{\epsilon}))\leq \beta$. Fix $r \in \R_+^m$; then, by Lemma \ref{lemma:alt_point_lb_union_multiple_gaussian_seq_l2_testing}, for $\epsilon=\tilde{\epsilon}(r)$, it holds that \begin{equation}
n \geq C \cdot \left(\frac{\sqrt{\intrinsic \Sigma^2}}{\tilde{\gamma}^2(\epsilon,r)}\wedge \frac{\sqrt{d}}{\epsilon^2}\right)   \textand \tilde{\gamma}(\epsilon,r) \cdot \sqrt{\frac{m}{\intrinsic \Sigma^2}} \leq \epsilon,
\end{equation} where $C$ is a positive constant that depends only on $\alpha$ and $\beta$. Finally, by definition of $\tilde{\epsilon}_*(r)$, see \eqref{eq:epsilon_star_arbitary_predictions_lowerbound}, it holds that \begin{equation}
\tilde{\epsilon}_*(r) = \inf\left\{\epsilon\geq 0: n \geq C \cdot \left(\frac{\sqrt{\intrinsic \Sigma^2}}{\tilde{\gamma}^2(\epsilon,r)}\wedge \frac{\sqrt{d}}{\epsilon^2}\right)   \textand \tilde{\gamma}(\epsilon,r) \cdot \sqrt{\frac{m}{\intrinsic \Sigma^2}} \leq \epsilon\right\},
\end{equation} Thus, it must hold that $
\tilde{\epsilon}(r) \geq \tilde{\epsilon}_*(r)$.  Since $r \in \R_+^m$ in the previous argument was arbitrary, the argument is valid for all $r \in \R_+^m$, which concludes the proof.
\end{proof}

In the following, we prove the necessary Lemma \ref{lemma:alt_point_lb_union_multiple_gaussian_seq_l2_testing}.

\begin{lemma}\label{lemma:alt_point_lb_union_multiple_gaussian_seq_l2_testing} For $k\leq d/2$, given separation curve $\tilde{\epsilon}$ with Property \ref{property:separation_curve_arbitary_predictions} and $r \in \R^m_+$, it holds that if \begin{align}
n < C \cdot \left(\frac{\sqrt{\intrinsic \Sigma^2}}{\tilde{\gamma}^2(\epsilon,r)}\wedge \frac{\sqrt{d}}{\epsilon^2}\right) \textand \tilde{\gamma}(\epsilon,r) \cdot \sqrt{\frac{m}{\intrinsic \Sigma^2}} \leq \epsilon\end{align}
where $\epsilon = \tilde{\epsilon}(r)$, $\tilde{\gamma}^2(\epsilon,r)=\norm{A(\epsilon,r)}^2_2/\norm{\hat{U}}^2_{\op}$, and $C$ is a positive constant that depends on $\alpha$ and $\beta$, then the minimax risk \eqref{eq:minimax_risk_definition} on $\mathcal{P}(\tilde{\epsilon})$ \eqref{eq:P_epsilon_tilde_arbitary_predictions} is not controlled:
$R_*(\mathcal{P}(\tilde{\epsilon})) > \beta$.
\end{lemma}

\begin{proof}[Proof of Lemma \ref{lemma:alt_point_lb_union_multiple_gaussian_seq_l2_testing}]Henceforth, fix $r_* \in \R_+^m$ and let \begin{equation}
\epsilon = \tilde{\epsilon}(r_*) \textand \gamma = \gamma(\epsilon,r_*) = \norm{A(\epsilon,r_*)}_2.
\end{equation}

\par\noindent\textbf{Constructing accuracy vectors that match the desired norm.} Given $y \in S_2^{m-1}$, $\gamma$ and $\epsilon$, define the $m$-dimensional vector $r(\epsilon,\gamma,y)$ as follows \begin{equation}
r_i = \sqrt{\epsilon^2 + h_i^2 - 2h_i y_i\gamma } \for i \in [m].
\end{equation} Since $\max_{i \in [m]}|y_i|\leq 1$, if $0 \leq \gamma \leq \epsilon$, it follows that \begin{equation}
\epsilon^2 + h_i^2 - 2h_i y_i \gamma \geq (\epsilon-h_i)^2 \geq 0 \for i \in [m],
\end{equation} which guarantees that $r(\epsilon,\gamma,y) \in \R^m_+$. Furthermore, it holds that \begin{equation}
A(\epsilon ,r) = \gamma \cdot y \textand \norm{A(\epsilon,r)}_2 = \gamma \for r=r(\epsilon,\gamma,y).
\end{equation}

\noindent\textbf{Constructing distribution supported on the alternative hypothesis.} Recall that $
\hat{U} = \tilde{U}_k\Lambda_k\tilde{V}_k^T$ is the reduced-form singular value decomposition of $\hat{U}$, see \eqref{eq:reduced_svd}, where $\tilde{U}_k$ spans a $k$-dimensional subspace. Henceforth, let $\tilde{U}_k^\perp$ be a basis for its orthogonal complement, a $(d-k)$-dimensional subspace. Let $\pi$ be the distribution of \begin{equation}\label{eq:alt_union_multiple_arbitary_pi_definition}
\theta = \tilde{U}_k\Lambda^{-1}_k\cdot y_w \cdot \gamma   + \tilde{U}_k^\perp z \cdot \frac{\rho(\gamma)}{\sqrt{d-k}}
\end{equation} where \begin{equation}
\rho(\gamma) = \sqrt{\epsilon^2-\norm{\Lambda^{-1}_k y_w}_2^2\cdot \gamma^2} \textand z \sim \text{Uniform}\left(\sqrt{d-k} \cdot S_2^{d-k-1}\right)
\end{equation} and \begin{equation}
y_w = \frac{\Lambda_k^2w}{\norm{\lambda}_4^2},\ w \in \{-1,1\}^k \textand w_i \iid \text{Uniform}\left(\{-1,1\}\right) \for i \in [k].
\end{equation} We require that $\epsilon^2 \geq \norm{\Lambda^{-1}_k y_w}_2^2\cdot \gamma^2$.
Noting that \begin{equation}
\norm{\Lambda^{-1}_k y_w}_2^2 = \frac{\norm{\lambda}_2^2}{\norm{\lambda}_4^4} = \frac{\trace \Sigma}{\trace \Sigma^2}=\frac{m}{\trace \Sigma^2}=\frac{m}{m+\sum_{i\neq j}\inner{\hat{u}_i}{\hat{u}_j}^2},
\end{equation} said condition is satisfied whenever \begin{equation}
\epsilon^2 \geq \frac{m}{\intrinsic \Sigma^2}\cdot \frac{\gamma^2}{\norm{\hat{U}}_{\op}^2}.
\end{equation} By definition of $\pi$ it holds that, for all $\theta \in \supp \pi$,\begin{equation}
\exists w \in \{-1,1\}^k \st \norm{\theta}_2=\epsilon \textand \hat{U}^T\theta = \gamma \cdot \tilde{V}_ky_w = A(\epsilon,r),\end{equation} where $r=r(\epsilon,\gamma,\tilde{V}_k y_w)$ and $\norm{\tilde{V}_k y_w}_2=1$ since $\norm{y_w}_2^2 = 1$ and $\tilde{V}_k$ has orthonormal columns. Thus \begin{equation}
\forall \theta \in \supp \pi, \exists w \in \{-1,1\}^k \st \norm{\theta}_2=\epsilon \textand \norm{\theta-\hat{\theta}_i}_2\leq r_i(\epsilon,\gamma,\tilde{V}_ky_w) \for i \in [m]
\end{equation} Finally, by Property \ref{property:separation_curve_arbitary_predictions}, it holds that $\pi$ is supported on the alternative hypothesis.

\noindent\textbf{Chi-squared distance between null hypothesis and mixture.} By Proposition \ref{prop:IngsterSuslina} and  \eqref{eq:alt_union_multiple_arbitary_pi_definition}, it holds that \begin{equation}
1 + \chi^2\left(P_\pi,P_0\right) = E_{\theta,\theta'\iid \pi}\exp{\left(n \cdot \inner{\theta}{\theta'}\right)}
\end{equation} where \begin{equation}
\inner{\theta}{\theta'} = \frac{\gamma^2}{\norm{\lambda}_4^4} \cdot \sum_{i=1}^k \lambda_i^{2}\cdot w_i w_i' + \frac{\rho^2(\gamma)}{d-k} \cdot \inner{z}{z'}.
\end{equation} Hence, we have that \begin{equation}
1 + \chi^2\left(P_\pi,P_0\right) = E_{w,w'}\exp\left(\frac{n\gamma^2}{\norm{\lambda}_4^4} \cdot \sum_{i=1}^k \lambda_i^{2}\cdot w_i w_i'\right) \cdot E_{z,z'}\exp\left(\frac{n\rho^2(\gamma)}{d-k} \cdot \inner{z}{z'} \right).
\end{equation} Since the distribution of $z$ is $C_0$-sub-Gaussian, where $C_0$ is a positive constant that does not depend on $d$, it follows that \begin{equation}
E_{z,z'}\exp\left(\frac{n\rho^2(\gamma)}{d-k} \cdot \inner{z}{z'} \right) \leq E_{z'}\exp\left(\frac{C^2_0}{2}\cdot \frac{n^2}{(d-k)^2} \cdot \rho^4(\gamma) \cdot \norm{z'}_2^2 \right),
\end{equation} where $\norm{z'}_2^2=d-k$. Furthermore, it holds that \begin{equation}
E_{w,w'}\exp\left(\frac{n\gamma^2}{\norm{\lambda}_4^4} \cdot \sum_{i=1}^k \lambda_i^{2}\cdot w_i w_i'\right) = \prod_{j=1}^k \cosh\left(\frac{n\gamma^2\lambda_j^2}{\norm{\lambda}_4^4}\right) \leq \exp\left(\frac{n^2\gamma^4}{2\norm{\lambda}_4^4}\right).
\end{equation}Thus, for any $\eta>0$, there exists $C_\eta$ such that $ \chi^2\left(P_\pi,P_0\right) \leq \eta$, whenever \begin{equation}
n^2 \cdot \frac{\gamma^4}{\norm{\lambda}_4^4} + \frac{n^2}{d-k}\cdot \rho^4(\gamma) \leq C_\eta,
\end{equation} where $C_\eta$ is a positive constant that depends on $\eta$. Using the assumption that $k \leq d/2$, by Proposition \ref{prop:two_fuzzy_hypotheses}, it follows that there exists a positive constant $C$ that depends on $\alpha$ and $\beta$ such that whenever \begin{equation}
n < C \cdot \min\left(\frac{\norm{\lambda}_4^2}{\gamma^2},\frac{\sqrt{d}}{\epsilon^2}\right),\ \epsilon^2 \geq \frac{m}{\intrinsic \Sigma^2}\cdot \frac{\gamma^2}{\norm{\hat{U}}_{\op}^2},\textand k\leq d/2,
\end{equation} it follows that $
R_*(\mathcal{P}(\tilde{\epsilon})) > \beta$. The statement of the lemma follows by noting that \begin{equation}
\frac{\norm{\lambda}_4^2}{\gamma^2} = \frac{\sqrt{\intrinsic \Sigma^2}}{\gamma^2/\norm{\Sigma}_{\op}} = \frac{\sqrt{\intrinsic \Sigma^2}}{\gamma^2/\norm{\hat{U}}^2_{\op}}  \text{ since } \intrinsic \Sigma^2 = \frac{\trace \Sigma^2}{\norm{\Sigma^2}_{\op}} = \frac{\norm{\lambda}_4^4}{\norm{\Sigma}_{\op}^2}.
\end{equation}
\end{proof}

\pagebreak

\etocdepthtag.toc{mtreferences}
\addcontentsline{toc}{section}{References}
\bibliography{guess}

\end{document}